\documentclass{amsart}

\usepackage{graphicx}
\usepackage{enumitem}
\usepackage{array}
\usepackage{arydshln}
\usepackage[foot]{amsaddr}
\usepackage{ dsfont }
\usepackage[usenames, dvipsnames]{xcolor}
\usepackage{url}
\usepackage{amsmath,amssymb,amsfonts}
\usepackage{mathrsfs}
\usepackage{bbm,upgreek}

\newtheorem{theorem}{Theorem}[section]
\newtheorem{proposition}[theorem]{Proposition} 
\newtheorem{corollary}[theorem]{Corollary}
\newtheorem{lemma}[theorem]{Lemma}

\newtheorem{remark}{Remark}[section]

\newtheorem{algo}{Algorithm}

\newcommand{\ie}[0]{\textit{i.e.}}
\newcommand{\eg}[0]{\textit{e.g.}}

\def\<{\langle}
\def\>{\rangle}

\DeclareMathOperator{\R}{\mathbb{R}}

\DeclareMathOperator{\N}{\mathbb{N}}

\DeclareMathOperator{\bbS}{\mathbb{S}} 

\newcommand{\norm}[1]{\left\lVert#1\right\rVert}

\DeclareMathOperator{\supp}{supp}
\DeclareMathOperator*{\argmin}{argmin}
\DeclareMathOperator{\diag}{\mathrm{diag}}

\DeclareMathOperator{\sign}{sign}

\DeclareMathOperator{\prox}{prox}

\DeclareMathOperator{\RadonD}{\vec{\mathcal{R}}}
\DeclareMathOperator{\radonSD}{\mathcal{R}^{\text{sd}}}
\DeclareMathOperator{\radonSDth}{\mathcal{R}_{\theta}^{\text{sd}}}
\newcommand{\Npxl}{N_{\text{pxl}}}
\newcommand{\Ndtc}{N_{\text{dtc}}}
\newcommand{\Nth}{N_{\text{ref}}}
\newcommand{\Nsh}{N_{\text{sh}}}

\DeclareMathOperator{\sh}{SH}
\DeclareMathOperator{\shD}{{\bf SH}}

\newcommand{\mbB}{\mathbf{B}} 

\newcommand{\mcP}{\mathcal{P}}  

\newcommand{\msH}{\mathscr{H}} 
\newcommand{\msR}{\mathscr{R}}
\newcommand{\msSC}{\mathscr{SC}} 
\newcommand{\msSH}{\mathscr{S \!\! H}}  

\renewcommand{\vec}[1]{\boldsymbol{#1}} 

\newcommand{\f}{\vec{f}} 
\newcommand{\w}{\vec{w}} 
\newcommand{\vsdiff}{\vec{r}} 
\newcommand{\g}{\vec{g}} 
\newcommand{\Wop}{\vec{W}} 
\newcommand{\Mop}{\vec{M}} 
\newcommand{\x}{\vec{x}} 
\newcommand{\y}{\vec{y}} 
\newcommand{\z}{\vec{z}} 
\newcommand{\thetab}{\vec{\theta}} 
\newcommand{\gNd}{\g_N^\delta} 
\newcommand{\faNd}{\f_{\alpha,N}^\delta} 
\newcommand{\vi}{\vec{v}} 

\newcommand{\wtR}{\widetilde{R}}
\newcommand{\whR}{\widehat{R}}
\newcommand{\bR}{\breve{R}}

\newcommand{\imag}{\text{Im}}

\def\coef{c_{\lambda, p, s, d}}

\newcommand{\nsG}[1]{\left\|#1 \right\|_{\text{sG}}}
\newcommand{\bu}{{\bf u}}
\newcommand{\Abu}{A_{{\bf u}}}
\newcommand{\Sbu}{S_{{\bf u}}}
\newcommand{\E}{\mathbb{E}}
\def\fdan{f^\delta_{\alpha,N}}
\def\fdanp{f^{\delta,p}_{\alpha,N}}
\def\fdano{f^{\delta,1}_{\alpha,N}}
\def\sdiff{r}
\def\pdan{\sdiff^\delta_{\alpha,N}}
\def\Jdan{J^\delta_{\alpha,N}}
\def\Jdanp{J^{\delta,p}_{\alpha,N}}
\def\Jdano{J^{\delta,1}_{\alpha,N}}

\def\gdan{{\bf g}^\delta_N}
\newcommand{\LO}{\mathcal{L}}
\def\coef{c_{\lambda, p, s, d}}
\def\coefb{c_{\lambda, q, -s, d}}

\graphicspath{{./images/}}

\title{Shearlet-based regularization in statistical inverse learning with an application to X-ray tomography}
\author[T. A. Bubba]{Tatiana A. Bubba}
\email{tb715@cam.ac.uk}

\author[L. Ratti]{Luca Ratti}
\email{luca.ratti@unige.it}
\date{}

\address[T.~A.~Bubba]{Department of Applied Mathematics and Theoretical Physics, University of Cambridge, Cambridge, United Kingdom}

\address[L. Ratti]{Department of Mathematics, University of Genoa, Genoa, Italy}


\begin{document}

\maketitle

\begin{abstract}
Statistical inverse learning theory, a field that lies at the intersection of inverse problems and statistical learning, has lately gained more and more attention. 
In an effort to steer this interplay more towards the variational regularization framework, convergence rates have recently been proved for a class of convex, $p$-homogeneous regularizers with $p \in (1,2]$, in the symmetric Bregman distance.

Following this path, we take a further step towards the study of sparsity-promoting regularization and extend the aforementioned convergence rates to work with $\ell^p$-norm regularization, with $p \in (1,2)$, for a special class of non-tight Banach frames, called shearlets, and possibly constrained to some convex set. The $p = 1$ case is approached as the limit case $(1,2) \ni p \rightarrow 1$, by complementing numerical evidence with a (partial) theoretical analysis, based on arguments from $\Gamma$-convergence theory.  
We numerically demonstrate our theoretical results in the context of X-ray tomography, under random sampling of the imaging angles, using both simulated and measured data. 
\end{abstract}

\section{Introduction} 

The task of identifying an unknown quantity $f$ from point-wise evaluations of a related observable $g$, understood as the outcome of an indirect and possibly noisy measurement process, is a common ground between inverse problems and statistical learning.
This class of problems, which has recently gained more and more attention (see, for example, \cite{Blanchard18}), is generally referred to as \textit{statistical inverse learning problems}. As the name prompts, it blends learning theory together with inverse problems.

On the one hand, the usual formulation of (deterministic) inverse problems deals with the recovery of an unknown quantity $f$ from some measurements $g$, and the link between $f$ and $g$ is described by a mathematical model based on physical laws. In this paper, we consider a linear observation model (which, in turn, is capable of describing several inverse problems of interest, \eg{}, in medical imaging) and provide the following general formulation
\[
\textbf{(inverse problem)} \quad \textit{given the noisy measurements}\quad g^\delta: \quad g^\delta = A f + \varepsilon, \quad \textit{recover} \quad f,
\]
where $\delta$ describes the noise level, namely $\|\varepsilon\|_Y \leq \delta$.
We suppose $A\colon X \rightarrow Y$ where $X$ and $Y$ are Banach spaces: in typical applications, such a mathematical model is derived in a continuous setting, hence it is natural to consider $X$ and $Y$ as infinite-dimensional spaces. We suppose in particular that $Y$ is a function space, and more precisely a space of functions from $U$ to $V$ (possibly, finite-dimensional spaces). To fix ideas, one can consider X-ray tomographic imaging, which is our guiding example throughout this paper. In this case,
the unknown $f$ represents the X-rays attenuation in a physical object (hence, $X$ is a suitable space of functions on $\Omega \subset \R^d$, $d=2,3$), and the operator $A$ is the well-known Radon transform~\cite{Natterer01}, which associates $f$ with the sinogram $g$, that is, a collection of integrals of $f$ along lines with different angles and offsets. For future reference, the sinogram $g \in Y$ can be also considered as a function associating each angle $u \in U$ with the collection of line integrals  at that angle, and with several different offsets.

On the other hand, the traditional formulation of a (direct) statistical learning problem deals with the task of identifying an underlying function $g\colon U \rightarrow V$ from point-wise evaluations, possibly corrupted by noise:
\[
\textbf{(statistical learning problem)} \quad \textit{given a sample} \quad \{(u_n,v_n)\}_{n=1}^N : v_n = g(u_n)+\varepsilon_n, \quad \textit{recover} \quad g.
\]
In this context, we are not interested in inferring a causal, model-based relation between the variables $u$ and $v$, but rather in describing their connection. Nevertheless, to do so, we can take advantage of a statistical model for $u$ and $v$: in particular, we assume that $\{(u_n,v_n)\}_{n=1}^N$ is an i.i.d.~sample of the random variables $u$ and $v$, whose joint probability distribution is unknown. Moreover, we suppose to have access to some information regarding such a statistical model, and in particular the observation model $v = g(u) + \varepsilon$, the distribution of the noise variable $\varepsilon$ and the fact that the variable $u$ (whose distribution is in general unknown) is independent of $\varepsilon$. 

The definition of a statistical inverse learning problem lies at the intersection of the two previously introduced ones, and reads as follows:
\[
\textbf{(statistical inverse learning problem)} \quad \textit{given} \quad \{(u_n,v_n)\}_{n=1}^N : v_n = (Af)(u_n)+\varepsilon_n, \quad \textit{recover} \quad f.
\]
This task shows some additional difficulties with respect to both the deterministic inverse problems and the direct statistical learning problem, namely, the random nature of the (possibly scarcely-sampled) data and the ill-posedness of the measurement operator $A$. However, we can take advantage of both the statistical model for the noise and the mathematical model $g = Af$, thus combining strategies from both backgrounds to recover a stable reconstruction of $f$. We immediately underline one relevant difference between the two approaches, and one very important analogy.

The difference lies in the description of the noise. In the proposed formulation of inverse problem, denoted by $f^\dag$ the exact solution, the experimental error on the measurements is modeled by an element $\varepsilon \in Y$ added to the noise-free measurements $g^\dag=A f^\dag$. In order to relate this to a statistical description, $\varepsilon$ can be interpreted as a realization of a random variable, but the requirement that $\varepsilon$ belongs to $Y$ and $\| \varepsilon \|_Y \leq \delta$ is well suited only for some possible models (\eg{}, $\varepsilon$ is a uniform random variable), but not for many other models which are extremely common in application. If, for example, we were to consider the noise as a Gaussian variable, we would not have a theoretical guarantee on the boundedness of any single realization of $\varepsilon$, even though a noise with a very large norm might occur with extremely small probability. The treatment of the so-called \textit{large noise} has been the object of several studies in the last decade in the inverse problem literature. In particular, we take advantage of the approach and results obtained in~\cite{Burger18}, which essentially consists of relying on the theoretical tool defined as the \textit{approximate source conditions}.

On the bright side, both inverse problems and statistical learning problems share an important feature, that is the use of \textit{regularization}. We refer the interested reader to~\cite{Benning18,Engl96,Schuster12} for a general overview on variational regularization for inverse problem and to~\cite{Caponnetto07} for an advanced discussion on regularization in statistical learning, particularly related to the scope of this paper. Generally speaking, this technique consists of replacing the solution of the previously outlined problems by the minimization of a functional defined as the sum of a data mismatch term and a suitable regularization functional $R$. A regularization parameter $\alpha$ controls the balance between the two terms. The solution of such minimization problem is denoted by $f_\alpha^\delta$ in inverse problems (where $\alpha$ stresses the dependence on the choice of the regularization parameter and $\delta$ on the noise level) and by $g_{\alpha,N}$ in the statistical learning context (where $N$ denotes the size of the learning sample). In statistical inverse learning problems, it is therefore natural to make use of regularization as well, and to denote the regularized solution as $\fdan$, as it will be more precisely outlined in section~\ref{sec:SettingStage}.

A compelling matter, which is also the main focus of this paper, is the issue of \textit{convergence}. In inverse problems, this consists of ensuring that, for a suitable choice of $\alpha$, as the noise level $\delta \rightarrow 0$, the regularized solution $f_\alpha^\delta$ convergence to $f^\dag$ in a suitable metric in $X$. In statistical learning, this task is instead related to proving that, for a suitable choice of $\alpha$, the reconstructed function $g_{\alpha,N}$ approaches $g^\dag$ as $N \rightarrow \infty$. Since both quantities are random variables, the convergence is measured in mean, or more precisely by means of the expected value of a suitable metric in $Y$. As a consequence, when considering statistical inverse learning problems we are interested in providing a rule to choose $\alpha$ as a function of $\delta$ and $N$ so that $\fdan$ converges to $f^\dag$ with respect to the expected value of a suitable metric in $X$. Notice that, in this context, we can both consider a scenario in which the noise level $\delta$ reduces (as in inverse problems) or in which the sample size $N$ increases (as in statistical learning), or suitable combinations.

In this context, the most relevant and advanced results are reported in a recent paper by Blanchard and M\"{u}cke~\cite{Blanchard18}, where the authors derive optimal convergence rates for a large class of regularization functionals $R$. More recently, in~\cite{Bubba21} such discussion has been extended to a class of convex, $p$-homogeneous regularizers, which are of great interest for applications. Indeed, 
in the latest decade a common paradigm to solve tomographic inverse problems has been variational regularization with penalties other than the usual quadratic ones, and in particular sparsity-enforcing priors. The results in~\cite{Bubba21} connect the recent developments in convex regularization for inverse problems with the framework of inverse statistical learning. A crucial aspect is, for example, the use of the (expected) Bregman distance as a tool to quantify convergence under some source condition assumption on the solution. The estimates in~\cite{Bubba21}
are not proven to be optimal in a minimax sense and leave out the most interesting $p$-homogenous case, $p = 1$. Nevertheless, such results establish a first step towards the study of sparsity-promoting regularization terms: in particular, by means of techniques based on Besov spaces, they allow to treat the case in which $R(f)$ is the $\ell^p$-norm of the wavelet coefficients of $f$, with $1 < p \leq 2$. Moreover, the theoretical estimates are numerically demonstrated with simulated tomographic data.

In recent years, though, the trend has been to model the task of image reconstruction with more general sparsifying transforms than wavelets and mostly using $\ell^1$-minimization. In particular, \textit{shearlets}~\cite{Kutyniok12}, a representation system specifically designed for multivariate data with anisotropic features, often predominant in images, have already been very successfully applied to various inverse problems, including X-ray tomography.
There is a vast literature on this topic (see subsection~\ref{ssec:LitRev}) for the deterministic framework, but no results that bridge the gap between statistical learning and inverse problems in the general case of frame-based $\ell^p$-minimization, with $1 \leq p < 2$,  possibly constrained to some convex set.

The goal of this paper is, therefore, to extend the results in~\cite{Bubba21} to work with a sparsity-enforcing model which is not limited to working with orthonormal bases for the sparsifying system, and applies in particular to shearlet-based regularization, thanks to powerful theory of shearlet coorbit spaces~\cite{Dahlke12}. In addition, our model allows to work also with randomized tomographic data where generally a non-negativity constraint is included in the minimized functional. 
These results constitute an even further step towards studying state-of-the-art sparsity-promoting regularization.

\subsection{Related research in the literature} 
\label{ssec:LitRev}
Our work lies at the intersection of convex regularization and statistical learning in the context of tomographic inverse problems. Apart from~\cite{Bubba21} which constitutes the foundation for our work (see also references therein), general convex regularization is not considered in statistical learning problems, which focuses more on deriving minimax optimal convergence rates, in a suitable metric, for Lasso-like problems~(see, for example, \cite{Blanchard18,Caponnetto07,Devito06} and references therein or~\cite{Hastie2015} for an overview).

Conversely, the sparsifying effect of Tikhonov regularization in terms of the $\ell^p$-norm, with $0 \leq p \leq 2$ has been an active field of research over the last two decades in the inverse problems field, often leading to state-of-the-art results. In particular, the derivation of convergence rates, in various error metrics, has been the object of study in a number of papers. The most relevant to our framework are \cite{Lorenz08}, which considers $p \in [0,2]$, \cite{Grasmair08} where $p \in [1,2]$,  and~\cite{Burger04,Grasmair11,Grasmair11bis,Haltmeier12}, which instead focus on $p=1$, that is, $\ell^1$ regularization.
An in-depth discussion on these papers is carried out in subsection~\ref{ssec:BoundsError}. Incidentally, notice that in these paper the theoretical estimate are generally not demonstrated with numerical simulations.

Under the label \textit{compressed sensing} many powerful recovery guarantees for subsampled random measurements have been derived~\cite{Candes2006,Foucart13}. 
It is generally well-understood that the theoretical guarantees of compressed sensing do not straightforwardly apply to tomographic imaging~\cite{Kaganovsky14,Petra14}. In an effort to close the gap between theory and practice, J{\o}rgensen, Sydky et al.~systematically studied in a series of works~\cite{Jorgensen15tris,Jorgensen15,Jorgensen15bis,Jorgensen13} how recoverability depends on sparsity and sampling. In addition to stardard $\ell^1$-norm regularization, the authors consider total variation (TV) regularization, which enforces gradient sparsity, and consider also real measured data. Their drive to approach a realistic tomographic setting has been of inspiration also for our work, particularly in section~\ref{sec:conclusions}. 

Sparse regularization methods from fewer tomographic measurements than usually required by standard methods (like filtered-back projection) have been widely used in tomographic reconstructions.
Beside the already mentioned TV, also wavelets~\cite{Klann15,Lassas09,Loris06}, curvelets~\cite{Candes00,Frikel13} and - the most relevant for our work - shearlets~\cite{Bubba18,Bubba20,Colonna10,Riis18} have been successfully applied. All these papers deal with deterministic inverse problems, and the relevance to our work is mostly numerical, for their use of sparse regularization to compensate for the subsampled data regime.

Finally, the theory of shearlet coorbit spaces has been developed in a series of paper~\cite{Dahlke09,Dahlke10,Dahlke11}, and traces back to the fundamental work of Feichtinger and Gr\"{o}chenig~\cite{Feichtinger89,Feichtinger89bis}.

\subsection{Our contribution} 
\label{ssec:OurContr}
Our main contributions are as follows:
\begin{itemize}
    \item In theorem~\ref{thm:general_rate_nonneg}, we derive convergence rates on the symmetric Bregman distance for the regularization term $\wtR(f) = \frac{1}{p} \norm{f}_X^p + \iota_+(f)$ where $1<p<2$, $X$ is the Besov space $B_p^s$ and $\iota_+$ is the indicator function of the non-negativity constraint. Under suitable assumptions, our rates coincide with those of the unconstrained case proved in~\cite{Bubba21}, for both noise regimes.
    \item Theorem~\ref{thm:general_rate_shearlet}, which is our main result, combines some results from the theory of shearlet Banach frames with  theorem~\ref{thm:general_rate_nonneg}. In details, for both noise regimes, we prove the same convergences rates on the symmetric Bregman distance (of theorem~\ref{thm:general_rate_nonneg}) for the regularization term $\wtR(f) = \frac{1}{p} \norm{f}_X^p + \iota_+(f)$, where now $X$ is the shearlet coorbit space $\msSC_{p,m}$ and $1<p<2$. We conjecture that this result could be useful to prove convergence rates also for other frames, under specific assumptions discussed in remark~\ref{rm:BanachFramesExt}.
    \item The convergence rates predicted by our theoretical results are demonstrated by studying the classical inverse problem of X-ray tomography under random sampling of the imaging angles. Our numerical simulations use both Besov and shearlet coorbit space penalties, with $p \in (1,2)$. In sections~\ref{sec:NonnegConstr}, \ref{sec:shearlets} and \ref{sec:Approachingp1} we use simulated data, while in section~\ref{sec:conclusions} we also consider measured data.
    \item In section~\ref{sec:Approachingp1} we leverage numerical intuition, in the case of shearlet coorbit space penalties, to show that the case $p=1$ can be approached  as the limit case $(1,2) \ni p \rightarrow 1$. In particular, in corollary~\ref{cor:minim_converg} we use an argument base on $\Gamma$-convergence to show  the convergence of the $\ell^p$-norm regularized solutions to the $\ell^1$-norm one. This result is especially relevant in the numerical setting, as we outline at the end of subsection~\ref{ssec:Gconv}.
\end{itemize}
Following the path paved by~\cite{Bubba21}, we do not derive any lower bounds and we do not prove any minimax-optimality of our results. As we briefly discuss in section~\ref{sec:conclusions}, in line with~\cite{Bubba21,Burger18} similar techniques can lead to minimax-optimal rates in inverse problems when considered against suitable source conditions.

\subsection{Organization of the paper} 
The remaining of this paper is organized as follows. In section~\ref{sec:SettingStage} we set the notation and  review the main results in~\cite{Bubba21}. The semidiscrete Radon transform is introduced in subsection~\ref{ssec:radon}, and the discrete formulation is described in subsection~\ref{ssec:discretization}.
In section~\ref{sec:NonnegConstr} and \ref{sec:shearlets} we extend the convergence rates proven in~\cite{Bubba21} for $1<p<2$ to the case of constrained shearlet-based regularization, for $1<p<2$. In particular, in section~\ref{sec:NonnegConstr} we prove that the concentrations rates derived in~\cite{Bubba21} for the unconstrained case hold true also when the solution is constrained to the non-negative orthant, in the case of Besov priors.  In section~\ref{sec:shearlets} we extend this result to the case of (constrained) shearlet-based regularization, which is our main result. Then, in section~\ref{sec:Approachingp1}, we move to the case $p=1$: here, some numerical examples are used as stepping stone for further theoretical considerations. Moreover, we make use of $\Gamma$-convergence to show that the regularized solution associated with the
case $p = 1$ can be approached by the regularized solution of $\ell^p$-norm problem when $p \rightarrow 1$. 
Rather than collecting all numerical simulations in a dedicated section, we present them little by little in each section. In particular, numerical simulations for $p \in (1,2)$ and with constrained Besov regularization are in subsection~\ref{ssec:NonnegConstr_numerics}, while those with constrained shearlet-based regularization and $p \in [1,2)$ are collected in subsections~\ref{ssec:ShearletsNumerics} and~\ref{ssec:TMconv}.  Closing remarks and an outlook to the remaining open problems are given in section~\ref{sec:conclusions}.
Finally, in appendix~\ref{app:VMILA}, we explain the practical implementation of VMILA within our numerical framework.

\section{Setting the stage}
\label{sec:SettingStage}
In this section, we set the notation and give a brief overview of some key results from~\cite{Bubba21} which will be useful in the rest of the paper. We also introduce the guiding application for the numerical study: X-ray tomography.
Our assumptions and notations are closely aligned with those of~\cite{Bubba21}.

\subsection{Notations and assumptions}
\label{ssec:ContFramework}
Let us consider a linear inverse problem
\begin{equation}
	\label{eq:IP}
	g = A f,
\end{equation}
where $A: X \rightarrow Y$ is a bounded linear operator between a separable Banach space $X$ and a Hilbert space $Y$.  Consider $Y = L^2(U;V)$, where $U \subset \R^d$ and $V$ is a Hilbert space. Here, we assume that the range of the operator $A$ is contained in a Banach space $Z \subset Y$ such that $Z \subset \mathcal{C}(U;V)$ continuously, and that $A : X \to Z$ is continuous. Notice that $\mathcal{C}(U;V)$ is a Banach space and  $Z \subset Y \subset Z^*$ forms a Gelfand triplet.
The noiseless observation $g^\dag = A f^\dag$ corresponds to our ground truth $f^\dag \in X$. In the following we study properties of a regularized solutions $\fdan$ defined via the variational problem
\begin{equation}
	\label{eq:regularized_sol_R}
	\fdan \in  \argmin_{f\in X} \Jdan(f) := \argmin_{f\in X} \left\{ \frac 12 \norm{\Abu f - \gdan}^2_{V_N} + \alpha R(f) \right\}
\end{equation}
where $V_N$ is the product space $\bigoplus_{n=1}^N V$ equipped with the norm induced by the inner product
\[
\<\mathbf{v},\mathbf{\tilde{v}}\>_{V_N} = \frac{1}{N} \sum_{n=1}^N \< v_n,\tilde{v}_n\>_V
\]
and for any $\mathbf{u} = (u_n)_{n=1}^N \in U_N = \bigoplus_{n=1}^N U$ we define the sampling operator $\Abu \in \LO(Z, V_N)$ such that
\begin{equation}
	\Abu f = (A_{u_n}f)_{n=1}^{N} 
	       = \big( (Af)(u_n) \big)_{n=1}^N
\label{eq:SamplingOperator}	       
\end{equation}
with $A_u$ satisfying the following conditions:
\begin{itemize}
\item[(A1)] there exists $\kappa\leq 1$ such that for all $u\in U$ and for all $f \in X$ we have
$\norm{A_u f}_V \leq \kappa \norm{f}_X$;
\item[(A2)] the mapping	$u \mapsto (Af)(u)$ is measurable for all $f\in X$.
\end{itemize}
In addition, 
\begin{equation}
\label{eq:noisydata}
\gdan := \Abu f^\dag + \delta\epsilon_N    
\end{equation}
where $\delta>0$ is the noise level, $\epsilon_N = (\epsilon_N^n)_{n=1}^N \in V_N$ is a random variable such that $\epsilon_N^n\sim \epsilon$ i.i.d.~with zero-mean and satisfies
\begin{equation}
\label{eq:randomnoise}
\E \big[\norm{\epsilon}_{V}^l \big] < \frac{l!}{2} \, \Sigma^{l-2}
\end{equation}
for all $l\geq 2$ and some constant $\Sigma>0$.
In the following, we will consider the design points $\{u_n\}_{n=1}^N$ as a random sample drawn from a probability distribution $\mu$ on $U$, independent of the noise distribution $\epsilon$. The measure $\mu$ allows to define the space $Y_\mu = L^2(U, \mu; V)$ associated with the inner product
\begin{equation*}
	\< g_1, g_2 \>_{Y_\mu} := \int_U \< g_1(u), g_2(u) \>_V \mu(du).
\end{equation*}
Clearly, $Z \subset {\mathcal C}(U; V) \subset Y_\mu$ and we denote by $A_\mu: X \rightarrow Y_\mu$ the operator such that
\begin{equation*}
	\< A_\mu f, g \>_{Y_\mu} = \int_{U} \< (Af)(u),g(u)\>_V \mu(du) \qquad \forall g \in Y_{\mu}.
\end{equation*}
As a simple example, one can consider $\mu$ as a uniform distribution on a bounded domain $U$: in this case, $\mu$ is a continuous distribution associated with a (constant) density equal to $1/|U|$. Therefore, the inner products of $Y$ and $Y_\mu$ only differ by a multiplicative constant $1/|U|$ and $A_\mu = 1/|U| A$.

The convex functional $R : X \to \R \cup \{ \infty\}$ satisfies the following three condition:
\begin{itemize}
\item[(R1)] the functional $R$ is lower semicontinuous in some topology $\mathscr{T}$ on $X$;
\item[(R2)] the sublevel sets $M_b = \{R \leq b\}$ are sequentially compact in the topology $\mathscr{T}$ on $X$;
\item[(R3)] the convex conjugate $R^\star$ is finite on a ball in $X^*$ centered at zero.
\end{itemize}

Notice that, contrary to~\cite{Bubba21}, we do not require the symmetry condition $R(-f)=R(f)$ for all $f \in X$. As already noted in~\cite{Bubba21}, it is not necessary, and, there, it was only employed to make the results more accessible.

The optimality criterion associated with \eqref{eq:regularized_sol_R} is given by
\begin{equation}
	\label{eq:optimality_criterion}
	\Abu^* ( \Abu \fdan - \gdan) + \alpha \pdan = 0
\end{equation}
for $\pdan \in \partial R(\fdan)$, where $\partial R$ denotes the subdifferential:
\begin{equation*}
	\partial R(f) = \{\sdiff \in X^* \; | \; R(f)-R(\tilde f) \leq \langle \sdiff, f-\tilde f\rangle_{X^* \times X} \; \text{for all} \; \tilde f \in X\}.
\end{equation*}
Moreover, for $\sdiff_f\in \partial R(f)$ and $\sdiff_{\tilde f} \in \partial R(\tilde f)$ we define the symmetric Bregman distance between $f$ and $\tilde f$ as
\begin{equation} \label{eq:BregDist}
	D^{\sdiff_f, \sdiff_{\tilde f}}_R (f,\tilde f) = \langle \sdiff_f-\sdiff_{\tilde f}, f-\tilde f\rangle_{X^*\times X}.
\end{equation}

In the following we consider $R(f) = \frac{1}{p}\norm{f}_X^p$ with $1 < p <2$, which is convex, $p$-homogeneous and differentiable. In addition, in this case there is a unique element $\sdiff_f\in \partial R(f)$: therefore, we drop the dependence on the subgradients in the notation of the symmetric Bregman distance and write simply $D_R(f,\tilde f)$.

\subsection{Compendium of useful preliminary results}
\label{ssec:PrelRes}
We report here results from~\cite{Bubba21} which we will later adjust to work with our different set up. Since we are considering the special case of $R(f) = \frac{1}{p}\norm{f}_X^p$ with $1 < p <2$, assumptions (R1)-(R3) are always satisfied and, therefore, we omit these hypotheses from the statements below.

We start with two results that allow to derive general bounds on the regularization term $R(\fdan)$ and the symmetric Bregman distance between $\fdan$ and $f^\dag$.

\begin{proposition}[proposition 3.1, \cite{Bubba21}]
\label{prop:apriori}
The functional $\Jdan$ has a unique minimizer, $\fdan \in X$, which satisfies
\begin{equation}
\label{eq:apriori2}
	R(\fdan) \leq C\left(R(f^\dag) + \left(\frac{\delta}\alpha\right)^{\frac{p}{p-1}} R^\star( \Abu^* \epsilon_N)\right),
\end{equation}
for some constant $C>0$.
\end{proposition}
Notice that \cite[proposition 3.1]{Bubba21} provides two different \textit{a priori} bounds. We only report the second one, which is the relevant one for our purposes. The next result is instead derived from \cite[proposition 3.2]{Bubba21}, where for simplicity we consider the choice $\Gamma_1 = \gamma_1 I$ and $\Gamma_2 = \gamma_2 I$, where $\gamma_1, \gamma_2 \in \R$ are positive constants and $I$ is the identity operator. 

\begin{proposition}[proposition 3.2, \cite{Bubba21}]
\label{prop:aux_convex2}
The regularized solution $\fdan$ given by \eqref{eq:regularized_sol_R} satisfies
\begin{multline}
	\label{eq:aux_breg2}
	D_R(\fdan, f^\dag) \
	\leq  \inf_{\bar{w} \in V_N} \left(R^\star\left(\frac{1}{\gamma_1}(\sdiff^\dag - \Abu^* \bar w)\right) + \frac \alpha 2 \norm{\bar w}_{V_N}^2\right) + R(\gamma_1(f^\dag - \fdan)) \\
	 + \frac{1}{\alpha} \left(R^\star\left(\frac{\delta} {\gamma_2} \Abu^* \epsilon_N\right)  +  R\left(\gamma_2(f^\dag - \fdan)\right)\right),
\end{multline}
where $\sdiff^\dag = \nabla R(f^\dag)$ and $\gamma_1, \gamma_2 \in \R$ are positive constants.
\end{proposition}

The following lemma, derived from the Xu-Roach's inequalities~\cite{Xu91}, is a key ingredient for the convergence analysis when $R(f) = \frac{1}{p}\norm{f}_X^p$.

\begin{lemma}[lemma 4.1, \cite{Bubba21}]
\label{lem:xuroach}
Let $f,\tilde f\in X$. For $1<p<2$ it holds that
\begin{equation*}
	\gamma^p R(f-\tilde f) \leq C \left(1-\frac p2\right) \gamma^{\frac{2p}{2-p}} \max\left\{R(f), R(\tilde f)\right\} + \frac p2 D_R(f,\tilde f),
\end{equation*}
for some $C>0$ depending on $p$ with any $\gamma>0$.
\end{lemma}

To derive convergence rates, it is common practice in inverse problems to assume some source condition on the solution. As an alternative, we introduce the following object
\begin{equation}
\label{eq:Rbeautiful}
	\msR(\beta, \bu; f^\dag) = \inf_{\bar w\in V_N} 
	\bigg\{
	R^\star\left(\sdiff^\dag - \Abu^* \bar w\right) + \frac \beta 2 \norm{\bar w}_{V_N}^2
	\bigg\}
\end{equation}
where $\sdiff^\dag = \nabla R(f^\dag)$. As we will show later, via $\msR$ we can state a more general requirement to prove convergence rates. 

The following results combines all the previous estimates to provide a deterministic bound on the Bregman distance between $\fdan$ and $f^{\dag}$.

\begin{theorem}[theorem 4.3, \cite{Bubba21}]
\label{thm:bregman_dist_gen}
The regularized solution $\fdan$ given by \eqref{eq:regularized_sol_R} satisfies
the following inequality:
\begin{equation}
\label{eq:bregman_dist_gen2}
	D_R(\fdan, f^\dag) 
	\leq   \widetilde C_p\left[\gamma_1^{-q} \msR(\alpha \gamma_1^q, \bu; f^\dag)  + 
	 H(\alpha, \delta, \gamma_1, \gamma_2) R^\star(\Abu^* \epsilon_N) +
	 \left(\gamma_1^p	+ \frac{\gamma_2^p}{\alpha}\right)^{\frac{2}{2-p}} R(f^\dag)\right]
\end{equation}
for arbitrary $\gamma_1,\gamma_2 >0$, where $\widetilde C_p>0$ is a constant dependent on $p$, $q$ is the H\"{o}lder conjugate of $p$ and 
\begin{equation}
	\label{eq:paramH}
	H(\alpha, \delta, \gamma_1, \gamma_2) = \frac{\delta^q}{\alpha \gamma_2^q} + \left(\gamma_1^p	+ \frac{\gamma_2^p}{\alpha}\right)^{\frac{2}{2-p}} \left(\frac \delta \alpha\right)^{q}.
\end{equation}
\end{theorem}

The estimates below provide the convergence rate, for $N \rightarrow \infty$, of the expected value of the Bregman distance associated with the optimal choice of $\alpha$ and for $1<p<2$. 
In the theorem below, we reported the estimates in theorem 4.11 of \cite{Bubba21} for the choice $Q=q/2$ which is the relevant one for, \eg{}, Besov spaces and the purposes of this paper.

\begin{theorem}[theorem 4.11 with $Q=q/2$, \cite{Bubba21}]
\label{thm:general_rate}
Suppose that assumptions (A1)-(A2) are satisfied and that
\begin{equation}
\label{eq:ass_on_phomog_rate1}
\E \left[ \msR(\beta, \bu; f^\dag) \right] \lesssim \beta + N^{-\frac{q}{2}}
\end{equation}
and 
\begin{equation}
\label{eq:ass_on_phomog_rate2}
\E \left[R^\star(\Abu^* \epsilon_N)\right] \lesssim N^{-\frac q2}.
\end{equation}
Then, as $N \rightarrow \infty$, we have the following convergence rates:
\begin{itemize}
    \item if $\delta N \rightarrow \infty$ (and $\delta^2/N \rightarrow 0$), then
    \begin{equation}
	\label{eq:param_choice_p_fixed_Tapio}
	\E \left[D_R(\fdan, f^\dag) \right] \lesssim  \left( \frac{\delta^2}{N}\right)^{\frac{1}{3}} \quad \text{for} \quad \alpha \simeq  \left( \frac{\delta^2}{N}\right)^{\frac{1}{3}};
    \end{equation}
    \item if $\delta N$ is bounded, then 
    \begin{equation}
	\label{eq:param_choice_p_fixed_Tapio}
	\E \left[D_R(\fdan, f^\dag) \right] \lesssim  N^{-1} \quad \text{for} \quad \alpha \simeq N^{-1}. 
    \end{equation}
\end{itemize}
\end{theorem}

Notice that the case $\delta^2/N \not\rightarrow 0$ (when $N \rightarrow \infty$) is not relevant: indeed, the scenarios we are interested in are when $\delta$ is bounded (this is the framework of statistical learning) and when  $\delta \rightarrow 0$ (this is the framework of inverse problems). 

Finally, we consider the case of sparsity promoting regularization with respect to an orthonormal wavelet basis for $L^2(\R^d)$. This amounts to consider the following regularizer:
\begin{equation}
\label{eq:BesovRegu}
	R(f) := \frac{1}{p} \sum_{\lambda=1}^\infty 2^{\varrho |\lambda|} |\langle f, \psi_\lambda\rangle|^p,
\end{equation}
where $\varrho \in \R$ and $|\lambda|$ denotes the scale of the wavelet basis function $\psi_{\lambda}$.
Notice that this choice fits the proposed framework by choosing $X$ as the Besov space $B_p^s(\R^d)$ for a suitable $s$, equipped with the norm
\begin{equation}
\label{eq:BesovReguEquiv}
\norm{f}_X = \left( \sum_{\lambda=1}^\infty \coef |\langle f, \psi_\lambda\rangle|^p \right)^{1/p}
\qquad \text{with} \qquad 
\coef = 2^{|\lambda| d \big(p(\frac{s}d + \frac{1}{2}) -1 \big)}.
\end{equation}
Indeed, in~\cite{Daubechies04}, the authors show that this provides an equivalent norm on the Besov space, namely, for some positive constants $D,D'$, it holds $D \norm{f}_X \leq \norm{f}_{B_{p}^s} \leq D' \norm{f}_X$. In particular, for the choice $s = \frac{\varrho}{p} + d\left(\frac{1}{p} - \frac{1}{2} \right)$ it holds $\coefb = 2^{\varrho |\lambda|}$ and so the functional $R(f)$ appearing in \eqref{eq:BesovRegu} is equal to $\frac{1}{p}\| f \|_X^{p}$.
In the case of~\eqref{eq:BesovReguEquiv}, propositions 5.4 and 5.5 of \cite{Bubba21} show that the assumption of theorem 4.11 in \cite{Bubba21} are satisfied with $Q=q/2$. Both propositions hold true under the following assumptions:
\begin{itemize}
    \item[(B1)] The ground truth $f^\dag\in X$ satisfies a \emph{classical source condition},
    \begin{equation}
	 \exists \; w \in Z \; \text{ s.t. } \;  r^\dag = A_\mu^* w
	 \qquad \text{ where } \; r^\dag = \partial R(f^\dag).
	\label{eq:SC_Besov}
    \end{equation}
    \item[(B2)] The wavelet basis and the operator $A$ satisfy
    \begin{equation*}
	\sum_{\lambda=1}^\infty \coefb \norm{A \psi_\lambda}_{\infty}^q < \infty,
    \qquad 
    \text{where } \quad 
    \coefb = 2^{|\lambda| d \big(q(-\frac{s}d + \frac{1}{2}) -1 \big)}.
    \end{equation*}
\end{itemize}

\begin{remark}
\label{rm:SourceCondNames}
In the following, condition \eqref{eq:ass_on_phomog_rate1} related to the quantity \eqref{eq:Rbeautiful} will be referred to as \textit{approximate source condition}, as opposed to the \textit{strong source condition} represented by (B1).
\end{remark}

Before moving to extending the results reported in this subsection to our new framework, as a complement to our theoretical analysis, we introduce the guiding application for the numerical experiments, \ie{}, X-ray
tomography.

\subsection{Main application: semidiscrete Radon transform}
\label{ssec:radon}

While the theory we present applies to any operator satisfying assumptions (A1)-(A2) and~\eqref{eq:ass_on_phomog_rate1}-\eqref{eq:ass_on_phomog_rate2}, we are particularly interested in the case of (semidiscrete) Radon transform and its application in 2D tomographic imaging. 
We start by recalling the classical definition of Radon transform $\mathcal{R}$:
\[
\mathcal{R}f (\theta,\tau) = \int_{\R} f(\tau \theta + t \theta^\perp) dt \qquad \theta \in S^1, \tau \in \R.
\]  
We consider the operator $\mathcal{R}$ acting on square-integrable functions $f \in L^2(\Omega)$, with $\Omega =[0,1]^2$. 
In this case, the so-called \textit{sinogram} $\mathcal{R}f$ belongs to the space $L^2([0,2\pi)\times(-\bar{\tau},\bar{\tau}))$ for a suitable $\bar{\tau} > 0$. 
We would like to define the sampling operator as a function associating an angle $\theta \in U = [0,2\pi)$ to the sinogram related to that direction, namely, $\mathcal{R}(\theta) = \mathcal{R}(\theta,\cdot) \in L^2(-\bar{\tau},\bar{\tau})$. Unfortunately, the sinogram space $L^2([0,2\pi)\times(-\bar{\tau},\bar{\tau})) \cong L^2(U; L^2(-\bar{\tau},\bar{\tau}))$ does not show sufficient regularity to perform pointwise evaluations with respect to the angles. 
\par
One way to overcome this difficulty is to rely on a semidiscrete version of the Radon transform. In particular, we set the variable $\tau$ in a discrete space, which corresponds to modeling the X-ray attenuation measurements performed with a finite-accuracy detector, consisting of $\Ndtc$ cells. To this end, we introduce a uniform partition $\{I_1, \ldots, I_{\Ndtc}\}$ of the interval $(-\bar{\tau}, \bar{\tau})$, where we denote by $\tau_i$ the midpoint of each interval $I_i$ and take a smooth positive function $\rho$ of compact support within $(-1,1)$ such that $\int_{-1}^1 \rho(x) dx = 1$. The semidiscrete Radon transform is a function $\radonSD \colon L^1(\Omega) \rightarrow L^2([0,2\pi); \R^{\Ndtc})$ such that, for any $f \in L^1(\Omega)$ and $\theta \in [0,2\pi)$, each component of the vector $\radonSD f(\theta) \in \R^{\Ndtc}$ can be written as
\begin{equation}
	\label{eq:Radon_model_eqL}
	[(\radonSD f)(\theta)]_i = 
	\int_{I_i} \int_{\R} f(\tau \theta + t \theta^\perp) \rho\left(\frac{\tau-\tau_i}{|I_i|}\right) dt d\tau.
\end{equation}
Notice carefully that, according to the formalism of subsection \ref{ssec:ContFramework}, 
$Y = L^2(U;V)$ with $U = S^1\cong [0,2\pi)$ and $V = \R^{\Ndtc}$.
If $f \in L^1(\Omega)$, each component of $\radonSD f(\theta)$ can be interpreted as a suitable average of $\mathcal{R}f (\theta,\tau)$ in a subinterval $I_i$.
By the change of variables $x = \tau \theta + t \theta^\perp$ in equation \eqref{eq:Radon_model_eqL} we can rewrite the previous equation as 
\[
	[(\radonSD f)(\theta)]_i = \int_{\R^2} f(x) \rho_i(x,\theta) dx,
\]
being $\rho_i(x,\theta) = \rho\left(\frac{x\cdot\theta-\tau_i}{|I_i|}\right)$. 
\par 
As a consequence, from the continuity of $\rho$ we can deduce that each component of $\radonSD f(\theta)$ is well defined for any $f \in L^1(\Omega)$, and continuously depend on $\theta$. Hence, we can consider $\radonSD: L^1(\Omega) \rightarrow Z$ being $Z = \mathcal{C}(U;V)$ and the sampled operator $\radonSDth : L^1(\Omega) \rightarrow V$ is well defined for every $\theta \in U$. Moreover, the following bound holds uniformly in $\theta$:
\[
\norm{\radonSDth f}_{V}^2 = \norm{\radonSD f(\theta)}_{V}^2 = \sum_{i = 1}^{\Ndtc} \left|\int_{\Omega} f(x)\rho_i(x,\theta)dx\right|^2 \leq \Ndtc \norm{f}_{L^1(\Omega)}^2 \norm{\rho}_{\infty}^2,
\]
and therefore we conclude that $A$ is a bounded operator from $L^1(\Omega)$ to $Z$. This accounts to say that Assumptions (A1) and (A2) are satisfied for any function space $X$ such that $X$ is continuously embedded into $L^1(\Omega)$.

\subsection{Discrete framework}
\label{ssec:discretization}

In order to perform numerical simulation, we will need a fully discrete counterpart of~\eqref{eq:regularized_sol_R} in the case $A=\radonSD$. To this end, we replace the functional space $X$ with $\R^{\Npxl}$, where $\Npxl$ denotes the total number of pixels involved in the discretization of $\Omega = [0,1]^2$, and consider the following discrete model:
\begin{equation}
\gNd = \g_N^\dag + \delta \vec{\epsilon}_N = \RadonD_{\thetab} \f^\dag + \delta \vec{\epsilon}_N
\label{eq:DiscrRadonSampl}
\end{equation}
where  $\f^\dag \in \R^{\Npxl}$ denotes the (unknown) discrete and vectorized image, $\RadonD_{\thetab} \in \R^{\Ndtc N \times \Npxl}$ represents the sampled version of the Radon operator corresponding to the $N$ randomly sampled angles $\thetab$, $\g_N^\dag \in \R^{\Ndtc N}$ is the subsampled sinogram and $\vec{\epsilon}_N \in \R^{\Ndtc N}$ is the noise. In the implementation, we consider a normal distribution for the noise vector, $\vec{\epsilon} \sim \mathcal{N}(\vec{0},\mathbbm{1}_{\Ndtc N})$, where $\mathbbm{1}_{\Ndtc N}$ is the identity matrix in $\R^{\Ndtc N \times \Ndtc N}$.

In what follows we will only consider regularizers of the form:
\begin{equation}
\vec{R}(\f) = \frac{1}{p} \norm{\Mop \f}_p^p
\label{eq:BesovReguDiscr}
\end{equation}
where $1 < p < 2$ and $\Mop \in \R^{\sigma\Npxl \times \Npxl}$ depends on the sparsifying transform. For example, when we consider wavelet-based regularization, then  $\Mop = \Wop \in \R^{\Npxl \times \Npxl}$ (with $\sigma=1$) is the matrix representation of an orthonormal wavelet transform. 
In particular, according to \eqref{eq:BesovReguEquiv}, $\vec{R}(\f)$ is equivalent to the $B_p^s(\Omega)$ norm, provided that $s = d\left( \frac{1}{p} -\frac{1}{2} \right)$.
With these notations, the discrete counterpart of \eqref{eq:regularized_sol_R} reads as:
\begin{equation}
\faNd = \argmin_{\f \in \R^{\Npxl}} \left\{ \frac{1}{2N} \norm{ \RadonD_{\thetab} \f -\gNd }_2^2 + \alpha \vec{R}(\f)  \right\}.   
\label{eq:minsDiscr}
\end{equation}

To solve~\eqref{eq:minsDiscr}, we use the variable metric inexact line-search algorithm (VMILA)~\cite{Bonettini16} (see, in particular, equations~\eqref{eq:VMILAiter} and \eqref{eq:VMILADualProbl_lp} in appendix~\ref{app:VMILA}).  
Compared to the proximal gradient descent (PGD) algorithm used in~\cite{Bubba21}, VMILA allows for more freedom with respect to the objective function to minimize. Firstly, we can consider any $p \geq 1$, while with PGD we were limited to values of $p$ which allowed for an explicit, analytic formula for the associated proximal operator. Moreover, VMILA admits in the regularization term also other sparsifying transforms than those forming an orthonormal basis. Finally, VMILA can easily handle minimization problems constrained over convex sets. 

\subsubsection{Implementing the source condition}
\label{subsec:SourceCond}
To verify numerically the convergence rates in theorem~\ref{thm:general_rate}, the test image $\f^{\dag}$ should satisfy the source condition (B1). To formulate (B1) in the discrete setting, we use a sufficiently refined discretization of the space $Y$ of full sinograms, namely, we consider the matrix $\RadonD \in \R^{\Ndtc \Nth \times \Npxl}$ representing the Radon transform acting from $\R^{\Npxl}$ to $\R^{\Ndtc \Nth }$, with a fixed number of imaging angles $\Nth \gg N$.
This leads to:
\begin{equation}
\exists  \; \w \in \R^{\Ndtc\Nth} \qquad \text{s.t.} \quad \Wop^{\text{T}} (\Wop \f^{\dag})^{[p-1]} = \RadonD^{\text{T}} \w  
\label{eq:SourceCondDiscr}
\end{equation}
where $\x^{[p]}$ is the component-wise signed $p$-th power, that is, $[\x^{[p]}]_i = \sign (x_i) |x_i|^{p}$.

In practice, a generic phantom of interest $\f_0$ does not necessarily satisfy~\eqref{eq:SourceCondDiscr}. Therefore, in the numerical simulations, to build a phantom that satisfies~\eqref{eq:SourceCondDiscr}, we follow the same strategy of~\cite{Bubba21}: 
\begin{itemize}
    \item[a)] Determine a vector $\w \in  \R^{\Ndtc\Nth}$ solution of the regularized problem
    \begin{equation}
    \w = \argmin_{\widetilde{\w} \in \R^{\Ndtc\Nth}}  \left\{ \frac{1}{2} \norm{\RadonD^{\text{T}} \widetilde{\w} - \Wop^{\text{T}} (\Wop \f_0)^{[p-1]}}_2^2 +  \alpha_{SC}   \norm{\widetilde{\w}}_2^2 \right\},
    \label{eq:SourceCondTik}
    \end{equation}
    for a suitable $\alpha_{SC}>0$.
    \item[b)] Compute $\f^{\dag} = \Wop^{\text{T}}(\Wop \RadonD^{\text{T}} \w)^{[1/(p-1)]}$.
\end{itemize}
As a result, $\f^{\dag}$ satisfies the source condition associated with $\w$, and $\|\f^{\dag} - \f_0\|_2$ is expected to be small. To solve~\eqref{eq:SourceCondTik}, we use the scaled gradient projection (SGP) algorithm~\cite{Bonettini09} (see equations~\eqref{eq:ObjFunctSGP}-\eqref{eq:SGPiter} in appendix~\ref{ssec:VMILAnonneg}).


\section{Warm-up: constraining the problem to the non-negative orthant}
\label{sec:NonnegConstr}
In X-ray tomography applications, it is known a priori that the desired image $f^\dag$ is non-negative and including this information is fundamental to obtain superior reconstruction results. This leads to the following minimization problem:
\begin{equation}
	\label{eq:regularized_nonneg}
	\fdan \in  \argmin_{f\in X} \Jdan(f) := \argmin_{f\in X} \left\{ \frac 12 \norm{\Abu f - \gdan}^2_{V_N} + \alpha R(f) + \iota_{+}(f)\right\}
\end{equation}
where $R(f) = \frac{1}{p} \norm{f}_X^p$ and $\iota_{+}$ is the indicator function of the non-negativity constraint:
\[
\iota_{+}(f) = \begin{cases}
0 &\text{if } \; f \geq 0 \; \text{a.e.,}\\
\infty &\text{elsewhere.}
\end{cases}
\]
This, in particular, implies that $\fdan$ is non-negative. It is moreover natural to introduce the following assumption:
\begin{equation} \label{eq:nonneg_fdaga}
    f^\dag \geq 0 \quad \text{a.e.}
\end{equation}
In the following, we denote by $\wtR(f) = R(f) + \iota_{+}(f)$ the constrained regularization functional. Clearly,  $\wtR(f) = R(f)$ if $f \geq 0$ a.e., and $\wtR(f)=\infty$ otherwise. As in the previous section, the choice $R(f) = \frac{1}{p}\| f \|_X^p$ ensures that $R$ satisfies (R1)-(R3). We do not need to guarantee the same for $\wtR$.

\subsection{Convergence rates for constrained regularization}
\label{ssec:NonnegConstr_theory}

We now characterize the Bregman distance associated with $\wtR$. Notice that this only makes sense when computed between non-negative functions. Also, since $\iota_{+}$ is not differentiable,
$\partial \wtR(g)$ is not single-valued. 
Nevertheless, since $R$ is differentiable we have
\[
\partial \wtR(g) = \{ \nabla R (g)\} \oplus 
    \partial \iota_{+} (g).
\]
For our scopes it is enough to consider the distance $D_{\wtR}$ associated with the choice $0 \in \partial \iota_{+} (g)$, that is,
\begin{equation}
D_{\wtR} (f,\tilde{f}) = 
    \< \nabla R(f) - \nabla R(\tilde{f}), f - \tilde{f} \>, 
    \qquad \forall \; f, \tilde{f} \geq 0 \quad \text{a.e.}
\end{equation}
This amounts to saying that
\begin{equation}
\label{eq:DistBreg_nonneg}
D_{\wtR} (f,\tilde{f}) = D_R (f,\tilde{f}) 
\qquad \forall \; f, \tilde{f} \geq 0 \quad \text{a.e.}    
\end{equation}

\bigskip

We are now ready to show that the results in subsection~\ref{ssec:PrelRes} hold true for the non-negative constrained case. First, we show that, thanks to~\eqref{eq:DistBreg_nonneg}, the general bounds on the regularization term and the symmetric Bregman distance between $\fdan$ and $f^\dag$ are the same as in the unconstrained case.

\begin{proposition}
\label{prop:apriori_nonneg}
The functional $\Jdan$ has a unique minimizer, $\fdan \in X$, which satisfies
\begin{equation}
\label{eq:apriori2_nonneg}
	R(\fdan) \leq C\left(R(f^\dag) + \left(\frac{\delta}\alpha\right)^{\frac{p}{p-1}} R^\star( \Abu^* \epsilon_N)\right)
\end{equation}
for some constant $C>0$.
\end{proposition}
\begin{proof}
The general structure of the proof follows the structure of~\cite[proposition 3.1]{Bubba21}. The only critical step is the following. Consider $b = \Jdan(f^\dag)$ and the sublevel set $M_b = \{f \in X\; | \; \Jdan(f) \leq b \}$. Now, any $f \in M_b$ must be non-negative: therefore, $\wtR = R$ in $M_b$ and the claim follows as in the rest of the proof of \cite[proposition 3.1]{Bubba21} by using the $p$-homogeneity of $R$.
\end{proof}

\begin{proposition}
\label{prop:aux_convex2_nonneg}
The regularized solution $\fdan$ given by \eqref{eq:regularized_nonneg} satisfies
\begin{multline}
	\label{eq:aux_breg2}
	D_R(\fdan, f^\dag) \
	\leq  \inf_{\bar{w} \in V_N} \left\{ R^\star\left(\frac{1}{\gamma_1}(\sdiff^\dag - \Abu^* \bar w)\right) + \frac \alpha 2 \norm{\bar w}_{V_N}^2\right\} + R(\gamma_1(f^\dag - \fdan)) \\
	 + \frac{1}{\alpha} \left(R^\star\left(\frac{\delta} {\gamma_2} \Abu^* \epsilon_N\right)  +  R\left(\gamma_2(f^\dag - \fdan)\right)\right),
\end{multline}
where $\sdiff^\dag = \nabla R(f^\dag)$ and $\gamma_1, \gamma_2 \in \R$ are positive constants.
\end{proposition}
\begin{proof}
The general structure of the proof follows the structure of~\cite[proposition 3.2]{Bubba21}. We only need to modify some critical steps of the original arguments. Since the subgradient of $\wtR$ is not single-valued, the optimality criterion~\eqref{eq:optimality_criterion} in this case is given by
\[
0 \in \partial \Jdan(\fdan) 
\qquad \Longleftrightarrow \qquad 
0 \in \{ \Abu^* (\Abu \fdan - \gdan) + \alpha \nabla R(\fdan)\}
\oplus \partial \iota_{+}(\fdan).
\]
This is equivalent to
\[
- \Abu^* (\Abu \fdan - \gdan) - \alpha \nabla R(\fdan) 
    \in \partial \iota_{+}(\fdan),
\]
which in turn yields
\begin{equation}
\label{eq:optimality_criterion_nn}
0 \leq \< - \Abu^* (\Abu \fdan - \gdan) - \alpha \nabla R(\fdan), \fdan - f \> 
\qquad \forall \, f \geq 0.
\end{equation}
By choosing $f = f^{\dag}$ (which is non-negative by assumption), applying the definition of $\gdan$ given in \eqref{eq:noisydata} to the optimality criterion
\eqref{eq:optimality_criterion_nn} and by subtracting on both sides $\alpha \< \nabla R(f^\dag), \fdan - f^\dag \>$, we obtain
\begin{equation}
\norm{\Abu(\fdan - f^\dag)}_{V_N}^2 
    + \alpha D_R(\fdan, f^\dag) \; \leq \;
	\alpha \langle \sdiff^\dag, f^\dag - \fdan\rangle_{X^*\times X} + \delta \langle \Abu^* \epsilon_N, \fdan - f^\dag  \rangle_{X^* \times X}.  
\end{equation}
The remainder of the proof is identical to~\cite[proposition 3.2]{Bubba21}. Notice, in particular, that we use the Fenchel-Young's inequality employing (the convex conjugate of) $R$, not $\wtR$.
\end{proof}

Given that proposition~\ref{prop:apriori_nonneg} and~\ref{prop:aux_convex2_nonneg} provide estimates depending on $R$ (and not $\wtR$), and $D_{\wtR} (\fdan,f^{\dag}) = D_R (\fdan,f^{\dag})$, we can directly use lemma~\ref{lem:xuroach} to prove theorems~\ref{thm:bregman_dist_gen} and \ref{thm:general_rate}. We then summarize the main result for the constrained case in the theorem below.  Since we are particularly interested in providing a theoretical backbone for the wavelet-based regularization with non-negativity constraint, we consider the case $X = B_p^s$ and assume that conditions (B1)-(B2) hold true. A more general result can be obtained whenever it is possible to guarantee that \eqref{eq:ass_on_phomog_rate1}-\eqref{eq:ass_on_phomog_rate2} are satisfied.

\begin{theorem}
\label{thm:general_rate_nonneg}
Suppose assumptions (A1)-(A2) are verified. Let $\wtR(f) = R(f) + \iota_{+}(f)$ and assume that $f^\dag \geq 0$ a.e. Let $X = B_p^s$ and suppose that (B1)-(B2) hold true.
Then, we have the following convergence rates, as $N \rightarrow \infty$:
\begin{itemize}
    \item if $\delta N \rightarrow \infty$ (and $\delta^2/N \rightarrow 0$), then
    \begin{equation}
	\label{eq:param_choice_p_fixed}
	\E \left[D_{\wtR}(\fdan, f^\dag) \right] \lesssim  \left( \frac{\delta^2}{N}\right)^{\frac{1}{3}} \quad \text{for} \quad \alpha \simeq  \left( \frac{\delta^2}{N}\right)^{\frac{1}{3}};
    \end{equation}
    \item if $\delta N$ is bounded, then 
    \begin{equation}
	\label{eq:param_choice_p_satur}
	\E \left[D_{\wtR}(\fdan, f^\dag) \right] \lesssim  N^{-1} \quad \text{for} \quad \alpha \simeq N^{-1}.
    \end{equation}
\end{itemize}
\end{theorem}

\subsection{Experiments and results}
\label{ssec:NonnegConstr_numerics}
In this subsection, we verify the expected convergence rates proven in theorem \ref{thm:general_rate_nonneg} for the previously introduced tomographic application (see subsections~\ref{ssec:radon} and \ref{ssec:discretization}), considering the following two scenarios:
\begin{itemize}
\item Fixed noise, \ie{}, $\delta >0$ constant. Since $\delta N \rightarrow \infty$, according to \eqref{eq:param_choice_p_fixed}, we take $\alpha \simeq N^{-1/3}$. In particular, we choose $\delta = c_{\delta} = 0.01 \| \RadonD \f^\dag \|_{\infty}$, namely, the $1\%$ of the peak value of the sinogram of the ground truth computed with a sufficiently fine angle discretization. Then, we let $\alpha = c_{\alpha}N^{-1/3}$, where $c_{\alpha}$ is heuristically determined in each experiment;
\item Decreasing noise, for example, $\delta \simeq N^{-1}$. In this case, the optimal parameter choice is $\alpha \simeq N^{-1}$ (see \eqref{eq:param_choice_p_satur}). In particular, 
we choose $\delta = c_{\delta} N^{-1}$, with $c_{\delta} = 0.02 N_{\min} \| \RadonD \f^\dag \|_{\infty}$, where $N_{\min}$ is the minimum value of $N$ considered for the numerical experiments. Notice that if  $N \in [N_{\min}, N_{\max}]$, the corresponding value of $\delta$ will range from $0.02 \| \RadonD \f^\dag \|_{\infty}$ to $0.02 N_{\min}/N_{\max} \| \RadonD \f^\dag \|_{\infty}$. Then, we set $\alpha = c_{\alpha} N^{-1}$, where $c_{\alpha}$ is heuristically determined in each experiment.
\end{itemize}
In all our experiments, we choose $N_{\min} = 36$ and $N_{\max} = 162$ in the interval $[0,\pi)$. In each noise scenario, $N$ random angles are sampled using Matlab's \texttt{rand} and the Gaussian noise $\vec{\epsilon}_N$ is created by the command \texttt{randn}. 
The forward operators $\RadonD_{\thetab}$ and its adjoint (as well as $\RadonD$ and its adjoint) are implemented using Matlab's \texttt{radon} and \texttt{iradon} routines, with suitable normalization. 
Reconstructions are computed using VMILA as described in appendix~\ref{ssec:VMILAnonneg} (see, in particular,  equations~\eqref{eq:VMILAiter} and \eqref{eq:VMILADualObjConstrained}), where we set $\Mop = \Wop \in \R^{\Npxl \times \Npxl}$.
The operator $\Wop$ is implemented using
SPOT's \texttt{opWavelet2}~\cite{VandenBerg2014}, with Haar filters and five scales.
The expected values appearing in~\eqref{eq:param_choice_p_fixed} and~\eqref{eq:param_choice_p_satur} are approximated by sample averages, computed using $30$ random realizations. This means that, for each number of angles $N$, the reconstruction is performed $30$ times, each time with a different set of $N$ drawn angles and noise vector.

Finally, in all tests we employ a phantom sized $128 \times 128$, hence $\Npxl = 128^2$. Notice that, to verify the converge rates in theorem~\ref{thm:general_rate_nonneg}, 
the phantom should satisfy the source condition. This is generally not a trivial task and we discuss in details two different strategies in the subsection below.

\subsubsection{Source conditions}
In subsection~\ref{subsec:SourceCond}, we described a technique to generate an image $\f^\dag$ satisfying the source condition (B1), starting from a phantom of interest $\f_0$. Unfortunately, even though the original phantom $\f_0$ is non-negative, the resulting $\f^\dag$ may have some negative components, thus violating \eqref{eq:nonneg_fdaga}. A possible solution to this issue is to post-process the image $\f^\dag$ by rescaling it so that its components range between $0$ and $1$. We do so by dividing it by its maximum component and then setting to zero the negative entries, that is
\[
[\f^\dag]_i \leftarrow \max\left\{\frac{[\f^\dag]_i}{\max_j([\f^\dag]_j)} , 0\right\}.
\]
The operation of rescaling is linear, which means that the rescaled image satisfies the source condition (associated with the vector $(\max_j([\f^\dag]_j))^{[1-p]} \w$); nonetheless, taking the positive part of a vector is not linear. As a result, we cannot expect that the post-processed image $\f^\dag$ satisfies the source condition exactly; nevertheless, we expect the quantity $\| \RadonD^T \w - \nabla \vec{R}(\f^\dag)\|_2$ to be small. 
\par
As an alternative, we can verify if the original phantom $\f^\dag = \f_0$ satisfies the approximate source conditions associated with the operator $\vec{R}$ (namely, a discrete version of \eqref{eq:ass_on_phomog_rate1}, see also remark~\ref{rm:SourceCondNames}). In particular, we can fix a sufficiently small value of $\beta$ and study the decay of the quantity $\E[\msR(\beta, \bu; \f^\dag)]$ with respect to $N$ as described in algorithm~\ref{algo:approxSC}.
\begin{algo}[]
(Verification of the approximate source conditions)
\begin{enumerate}
    \item Fix a phantom $\f^\dag$, $1<p<2$ and $\beta>0$;
    \item Compute $\vsdiff^\dag = \Wop^T(\Wop\f^\dag)^{[p-1]}$;
    \item For $N = N_{\min}, \ldots, N_{\max}$
    \begin{enumerate}
        \item For $k = 1,\ldots,K$
        \begin{itemize}
            \item[(a$_1$)] Randomly sample $N$ angles;
            \item[(a$_2$)] Compute $ \msR(\beta,\bu;\f^\dag)$ by solving $\displaystyle \inf_{\bar{\w} \in \R^{\Ndtc N}} 
	\bigg\{
	\frac{1}{q}\norm{\vsdiff^\dag - \RadonD_{\thetab}^T \bar{\w}}_q^q + \frac{\beta}{2N} \norm{\bar{\w}}_{2}^2 \bigg\}$
        \end{itemize}
    \item Approximate $\E[\msR(\beta, \bu; \f^\dag)]$ by the sample average on the $K$ repetitions.
    \end{enumerate}
    \item Verify that $\E[\msR(\beta, \bu; \f^\dag)]$ decays as $N^{-q/2}$.
\end{enumerate}
\label{algo:approxSC}
\end{algo}

\begin{remark}
Notice that the latter strategy (based on verifying the approximate source condition) yields a phantom that does not depend anymore on a particular $p$. Indeed, for the strong source condition (B1) we build a  different phantom $\f^\dag$, for a particular choice of $p$, starting from a phantom image $\f_0$. Instead, for the approximate source condition we study the decay of the quantity $\E[\msR(\beta, \bu; \f^\dag)]$ for a phantom image $\f_0$ and, if verified, we set $\f^\dag = \f_0$ independently of $p$. 

On the other hand, to verify the strong source condition one only needs the operator $\RadonD$, \ie{}, always the same operator regardless of the sampling operation. Instead, the approximate source condition requires the sampled operator $\RadonD_{\thetab}$.
\end{remark}

In all tests, we chose $\f_0$ to be the plant phantom, available on GitHub~\cite{PlantPhantom2020}. For such $\f_0$ we verified both the approximate source condition (with $K=30$) and the strong source condition (with different values of $p$), for which we always have that  $\| \RadonD^T \w - \nabla \vec{R}(\f^\dag)\|_2$ is below the $4\%$ of $\| \RadonD^T \w\|_2$. In this context, we used an operator $\RadonD$ on a refined angle grid with $\Nth = 500$. Similarly to~\eqref{eq:SourceCondTik}, the solution of the minimization problem in (a$_2$) is done by using the SGP algorithm (see equations~\eqref{eq:ObjFunctSGP}-\eqref{eq:SGPiter} in appendix~\ref{ssec:VMILAnonneg}).

\subsubsection{Discussion}
\label{sssec:ResultsNonneg}

\begin{figure}
    \centering
    \begin{tabular}{@{}c@{\;}c@{\;}c@{\;}c@{}}
        \multicolumn{2}{c}{fixed noise}
        & \multicolumn{2}{c}{decreasing noise} \\
        $p=3/2$ & $p=4/3$
        & $p=3/2$ & $p=4/3$ \\ 
        \includegraphics[width=0.25\textwidth]{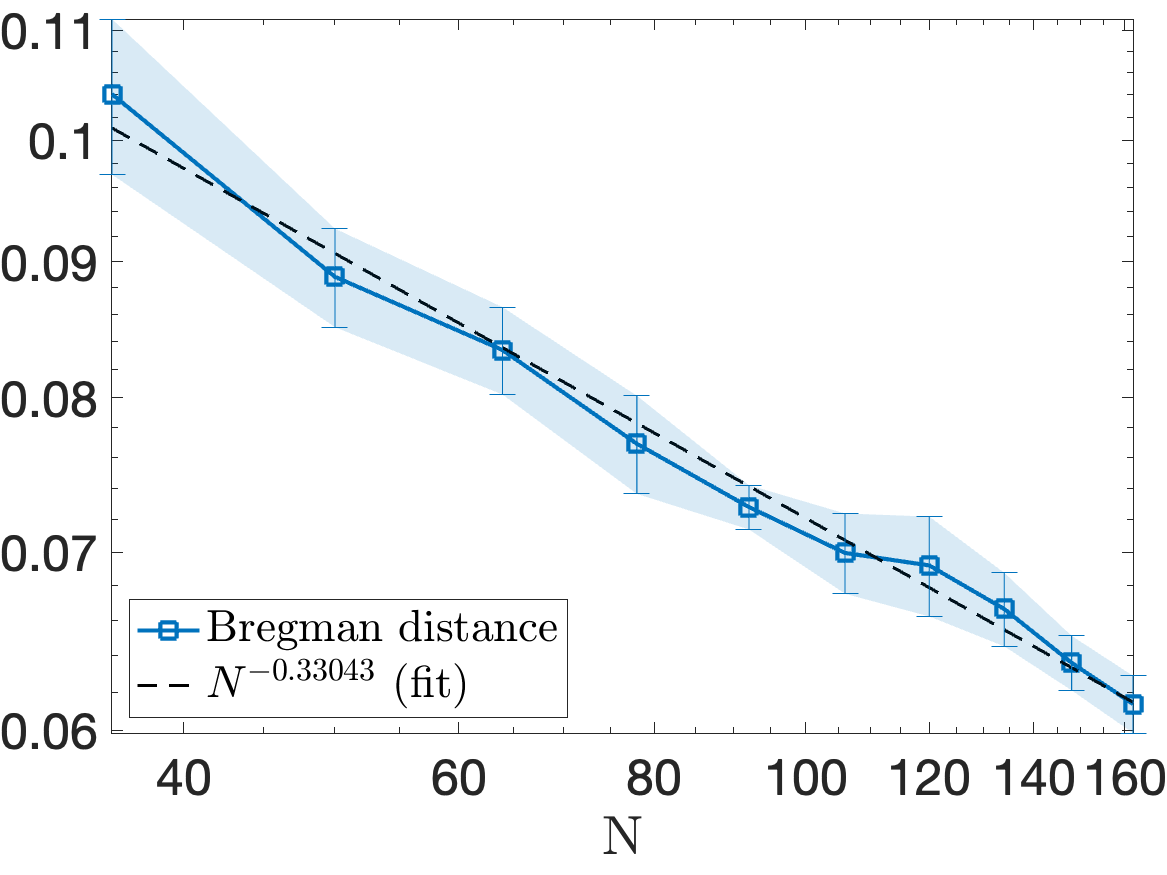} 
        & \includegraphics[width=0.25\textwidth]{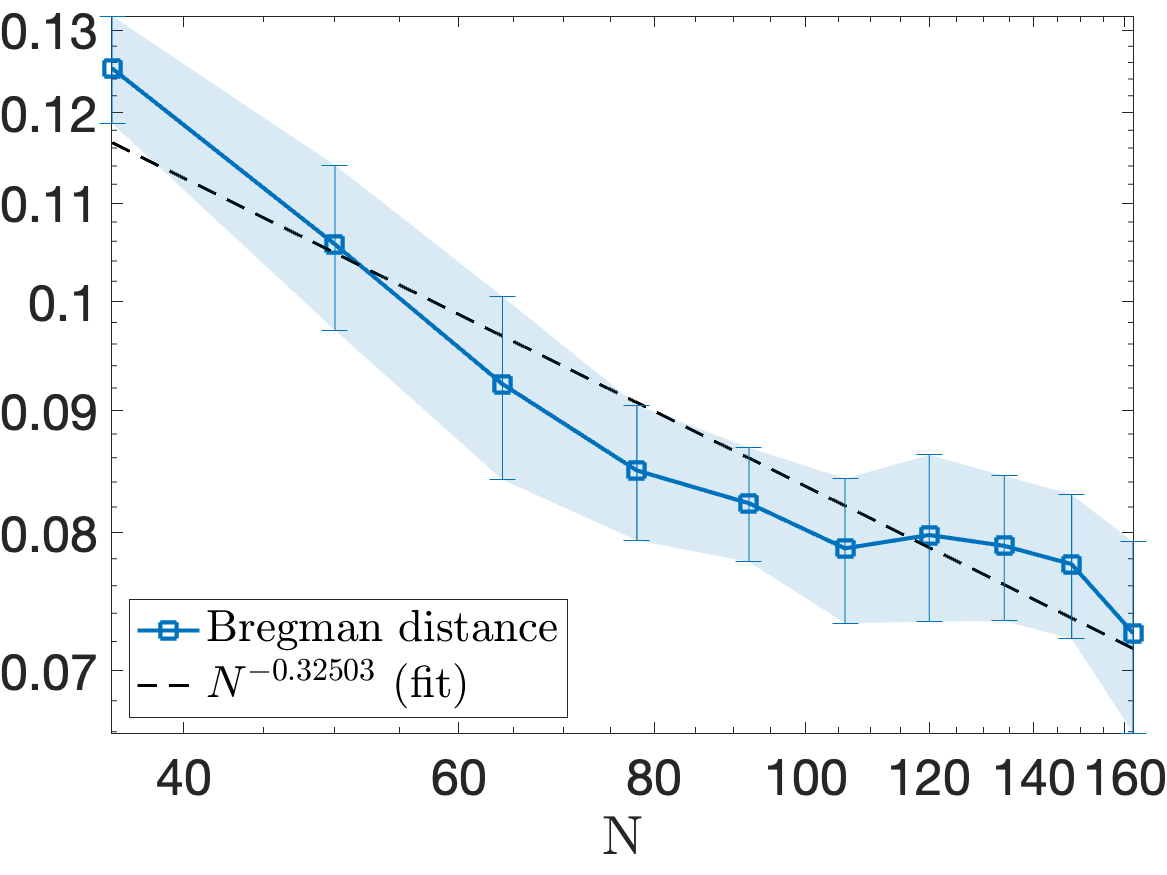} 
        & \includegraphics[width=0.25\textwidth]{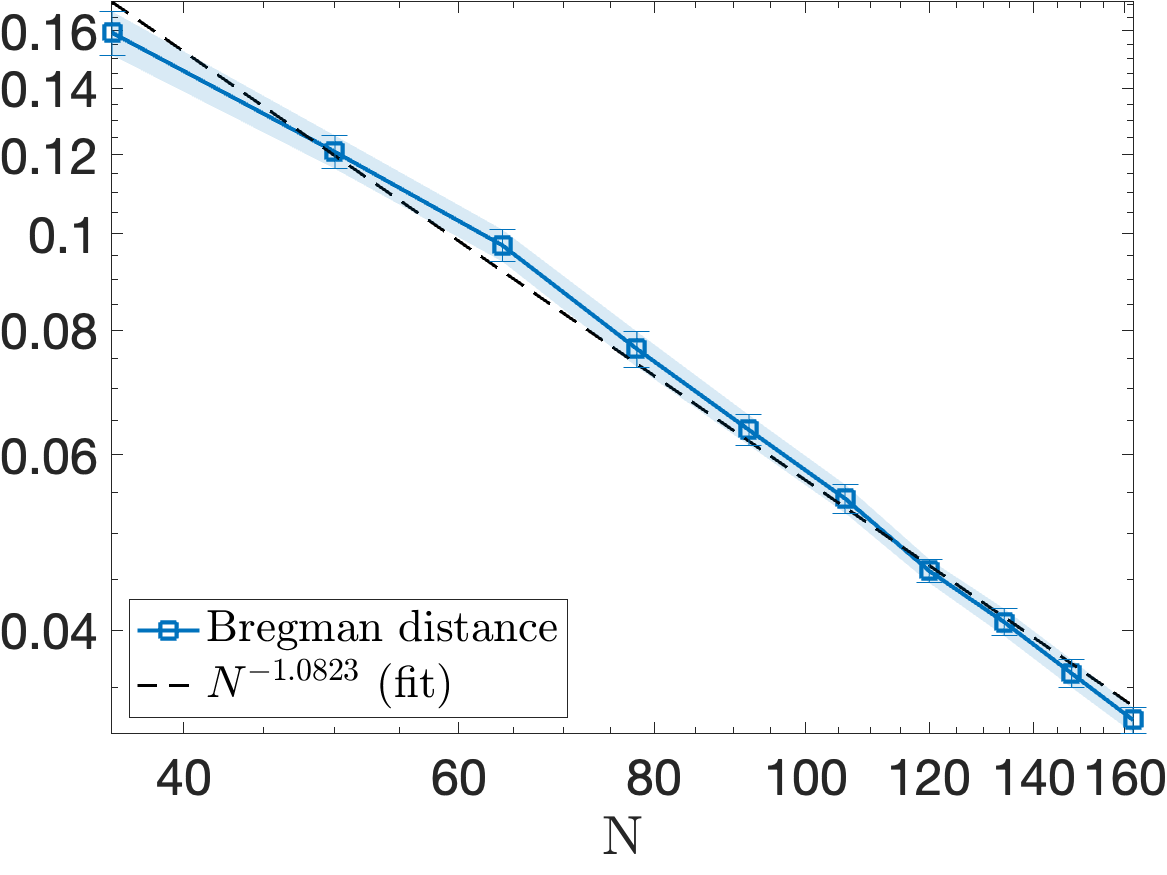}
        & \includegraphics[width=0.25\textwidth]{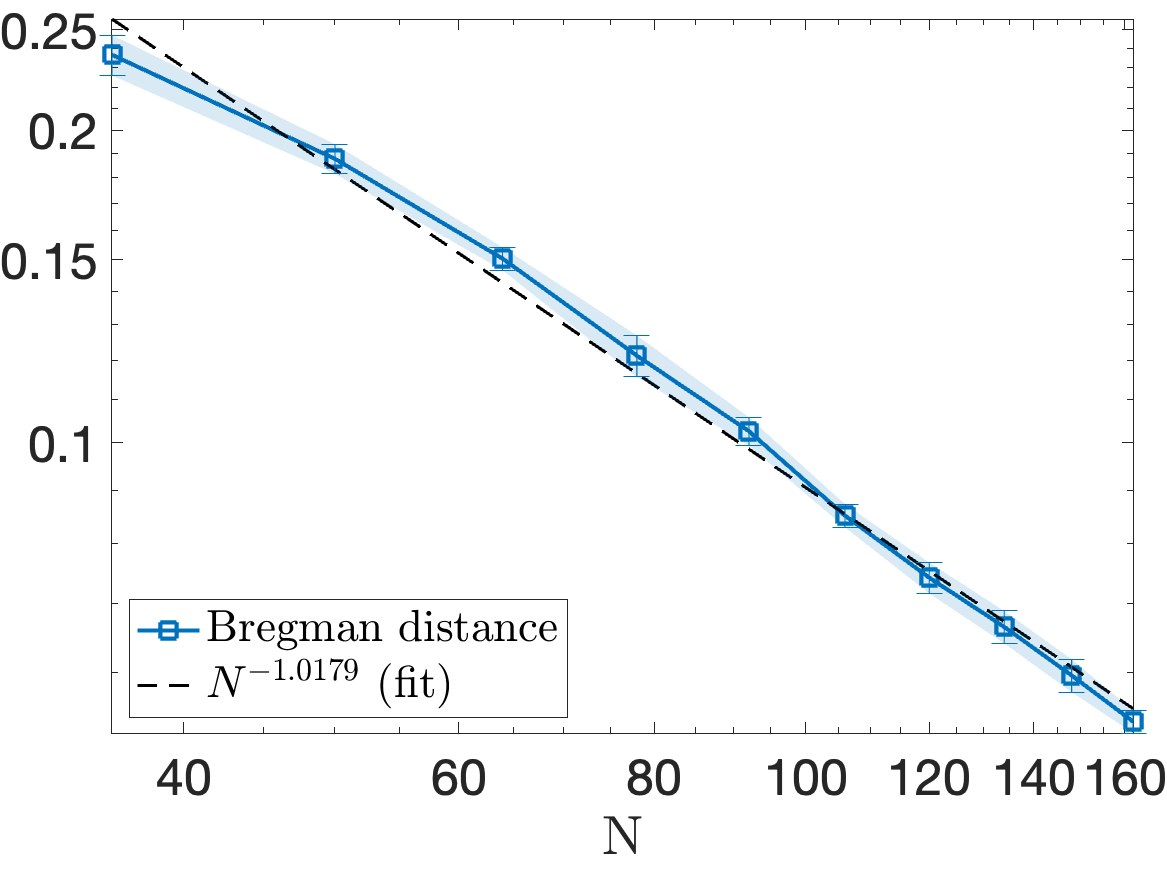} \\
        (a) & (b) & (c) & (d)
    \end{tabular}
    \caption{Approximate decay of the expected value of the Bregman distance, with wavelets-based regularization, for $p=3/2$ ((a) and (c)) and $p=4/3$ ((b) and (d)). The phantom satisfies the strong source conditions. (a)\&(b): Fixed noise regime. (c)\&(d): Decreasing noise regime.}
    \label{fig:Plant_wavelets_strSC}
\end{figure}

\begin{figure}
    \centering
    \begin{tabular}{@{}c@{\;}c@{\;}c@{\;}c@{}}
        \multicolumn{2}{c}{fixed noise}
        & \multicolumn{2}{c}{decreasing noise} \\
        $p=3/2$ & $p=4/3$
        & $p=3/2$ & $p=4/3$ \\
        \includegraphics[width=0.25\textwidth]{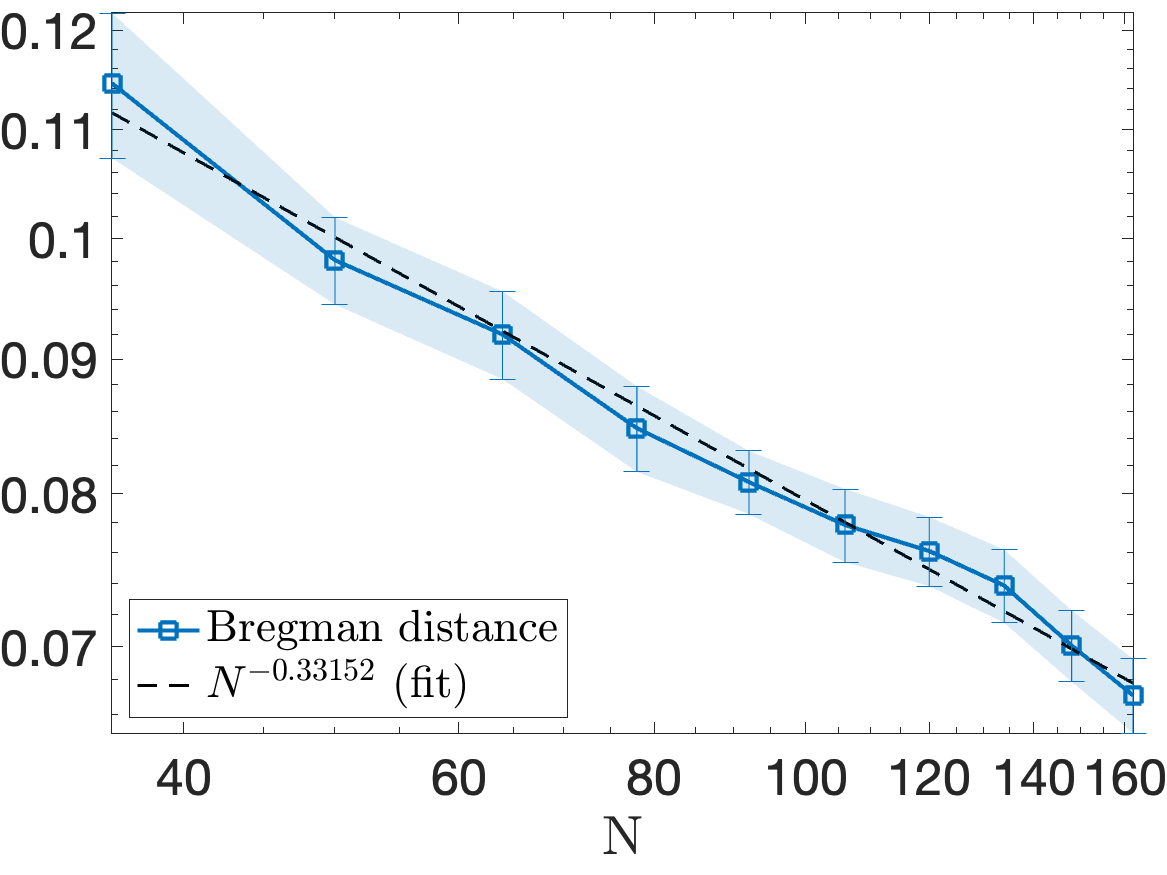} 
        & \includegraphics[width=0.25\textwidth]{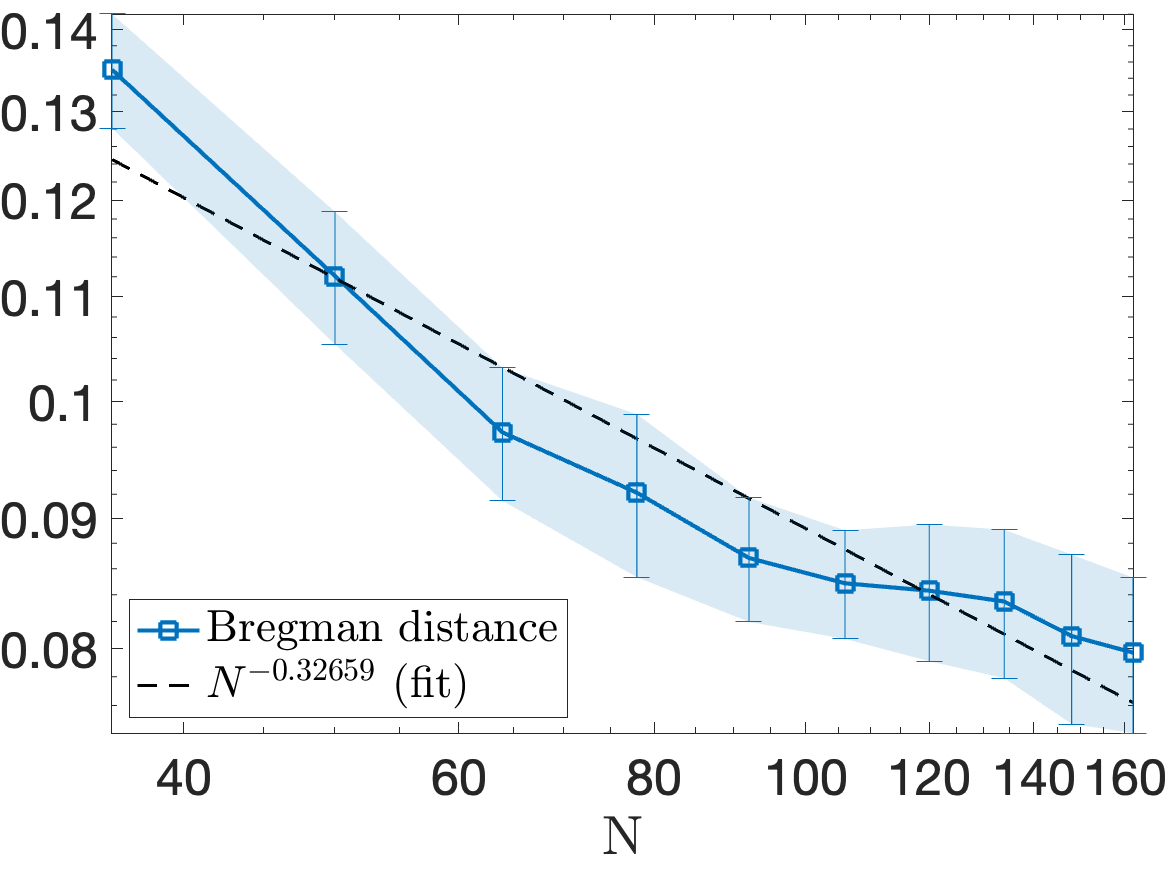} 
        & \includegraphics[width=0.25\textwidth]{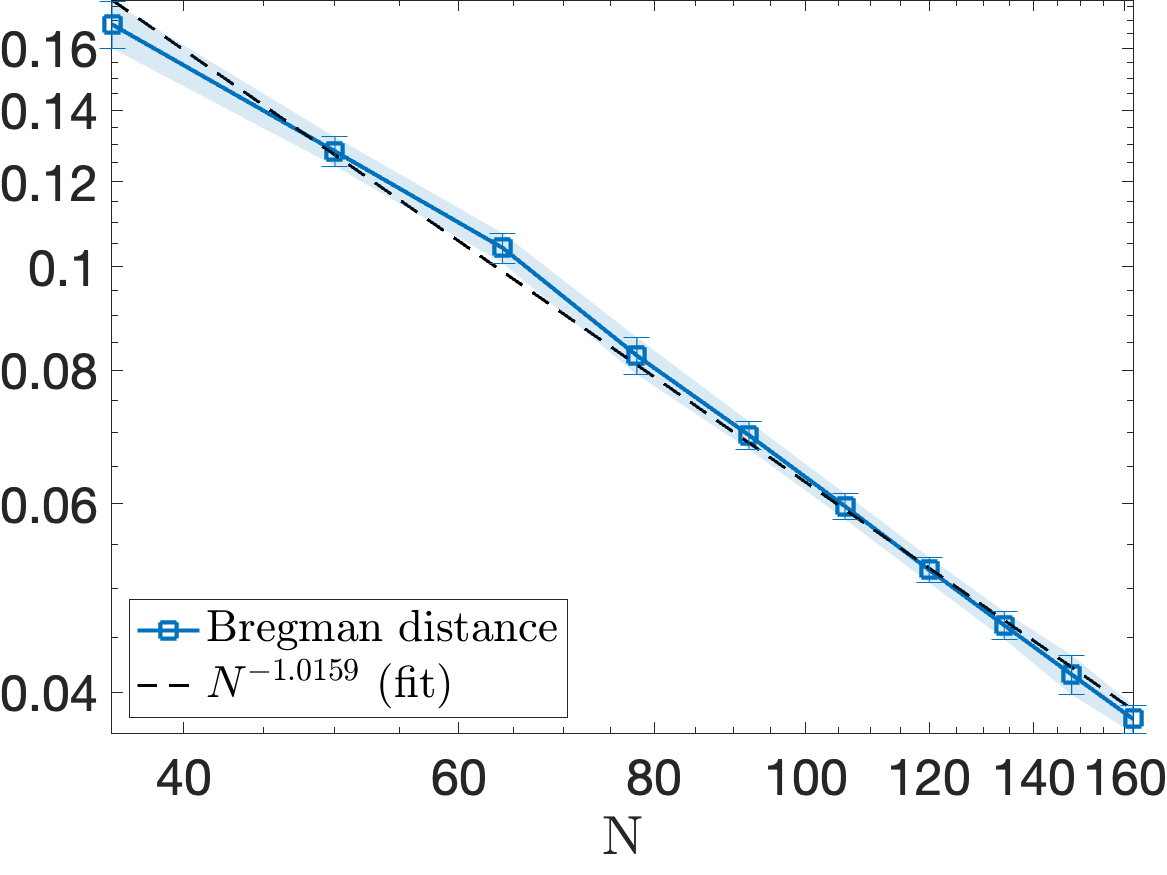}
        & \includegraphics[width=0.25\textwidth]{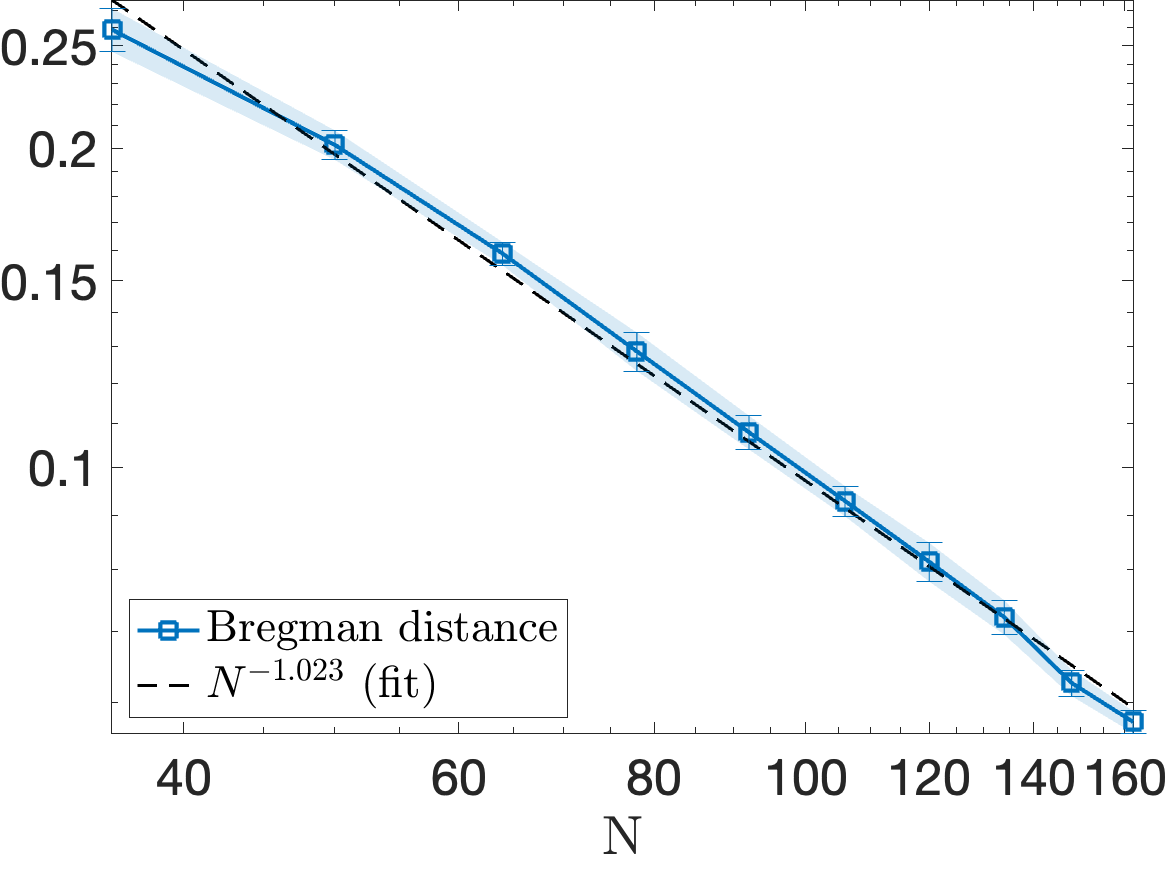} \\
        (a) & (b) & (c) & (d)
    \end{tabular}
    \caption{Approximate decay of the expected value of the Bregman distance, with wavelets-based regularization, for $p=3/2$ ((a) and (c)) and $p=4/3$ ((b) and (d)). The phantom verifies the approximate source conditions. (a)\&(b): Fixed noise regime. (c)\&(d): Decreasing noise regime.}
    \label{fig:Plant_wavelets_appSC}
\end{figure}

In figures \ref{fig:Plant_wavelets_strSC} and \ref{fig:Plant_wavelets_appSC} we report the value of the expected Bregman distance $\E[ D_{\vec{\wtR}}(\f^\delta_{\alpha,N}, \f^\dag)]$ as a function of $N$, with respect to a phantom $\f^\dag$ satisfying either the strong source condition (see figure \ref{fig:Plant_wavelets_strSC}) or the approximate source condition (see figure \ref{fig:Plant_wavelets_appSC}).
In each figure, we consider both decreasing and fixed noise regimes and two different choices for $p$, namely, $p=3/2$ and $p=4/3$.

In order to provide a quantitative assessment, we compare the theoretically predicted decay with the experimental one, which is obtained by computing the best monomial approximation $c N^{\beta}$ of the reported curves. In each plot, the value of the expected Bregman distance is indicated by a blue solid line and its monomial approximation by a black dashed line. We also report the standard deviation error bars (that is, the shaded region in each plots).

According to theorem \ref{thm:general_rate_nonneg}, we expect the same decay of $\E [D_{\vec{\wtR}}(\f^\delta_{\alpha,N}, \f^\dag)]$, independently of $p$: as $N^{-1/3}$ in the fixed noise one, and as $N^{-1}$ in the decreasing noise scenario. We can see in table \ref{tab:decaysWavelets} that the theoretical behaviour is numerically verified. This is also confirmed in figures \ref{fig:Plant_wavelets_strSC} and \ref{fig:Plant_wavelets_appSC}, where we can also gather some further insight on the behaviour of the random variables $D_{\vec{\wtR}}(\f^\delta_{\alpha,N}, \f^\dag)$. In particular, observe that in the fixed noise regime the oscillations around the mean are larger, especially for $p=4/3$. This does not seem to be related to the size of the samples, since even using 100 random realizations delivers a very similar plot. 
Finally, observe that, as expected, the choice of $\f^\dag$ with respect to the source condition (strong or approximate) does not affect at all the verification of the estimates. Actually, once fixed $p$ and the noise regime, the reconstructions providing the estimates in figures \ref{fig:Plant_wavelets_strSC} and \ref{fig:Plant_wavelets_appSC} use very similar (if not the same) values of the regularization parameter $\alpha$.

To summarize, the results reported in this section allow to conclude that the decay of the expected Bregman distance reported in theorem~\ref{thm:general_rate_nonneg} can be verified in the case of the discrete Radon transform, also for the constrained formulation.

\begin{table}[t]
\centering
\caption{Approximate decay of the expected value of the Bregman distance, with wavelets-based regularization, for $p=3/2$ and $p=4/3$.}
\begin{tabular}{c|c|cc|cc}
\multicolumn{1}{c}{} &  & \multicolumn{2}{c|}{strong source condition} 
& \multicolumn{2}{c}{approx.~source condition} \\
scenario & theoretical & $p = 3/2$ & $p=4/3$ & $p = 3/2$ & $p=4/3$  \\
\hline
decreasing noise & $-1$ & $-1.0823$ & $-1.0179$ & $-1.0159$ & $-1.023$ \\
fixed noise & $-1/3$ & $-0.33043$ & $-0.32503$ & $-0.33152$ & $-0.32659$
\end{tabular}
\label{tab:decaysWavelets}
\end{table}

\section{Shearlets come into play}
\label{sec:shearlets}
We now want to extend theorem~\ref{thm:general_rate_nonneg} so that it can be applied to the more general framework of frames, rather than bases, which would allow in particular to consider shearlet-based regularization. This amounts to consider the following regularizer:
\begin{equation}
\label{eq:ShearRegu}
	R(f) := \frac{1}{p} \sum_{\lambda \in \Lambda} m_{\lambda}^p |\langle f, \psi_\lambda\rangle|^p,
\end{equation}
where $\{\psi_{\lambda}\}_{\lambda \in \Lambda}$ is now a shearlet frame and $m_{\lambda}$ plays the role of  weight and will be formally introduced below.
Similarly to the wavelet case, the space $X$ must be chosen carefully: in this case, we consider \textit{shearlet coorbit spaces}, which are generally associated with decay properties of shearlet coefficients of discrete shearlet frames and related to deriving embedding results into Besov spaces~\cite{Dahlke12}.

\subsection{A splash of shearlet coorbit theory}
\label{ssec:shearletstheory}
Shearlets are representation systems specifically designed to provide optimally sparse approximations of a special class of signals, called cartoon-like images~\cite{Kutyniok12introduction}. Here, we give a concise overview of their main properties and the reader is referred to the cited literature for more details.

We introduce the shearlet transform on $L^2(\R^d)$, with $d \geq 2$ even though in the following we will apply it only in the $d=2$ setting.
For $f \in L^2(\R^d)$ we define 
\begin{equation}
    \pi(a,\vec{s},\vec{t})f(x) = f_{a,\vec{s},\vec{t}}(x) = |\det(A_a)|^{-\frac{1}{2}} f(A_a^{-1} S_{\vec{s}}^{-1}(x-t)) 
\end{equation}
where $(a,\vec{s}) \in \R \backslash \{0\} \times \R^{d-1}$ controls the parabolic scaling matrix
$A_a$ and the shearing matrix $S_{\vec{s}}$:
\[
A_a = \begin{pmatrix}
a & \vec{0}^{\text{T}}_{d-1} \\[0.25em]
\vec{0}_{d-1} & \sign(a)|a|^{\frac{1}{d}} \mathbbm{1}_{d-1}
\end{pmatrix}
\qquad \text{and} \qquad
S_{\vec{s}} = \begin{pmatrix}
1 & \vec{s}^{\text{T}} \\[0.25em]
\vec{0}_{d-1} & \mathbbm{1}_{d-1}
\end{pmatrix}
\]
and $\vec{t} \in \R^d$ encodes translations. 
Here, $\pi(a,\vec{s},\vec{t})$ is a unitary representation of the so-called \textit{shearlet group} $\bbS$ defined as the semi-direct product $(\R \backslash \{0\} \times \R^{d-1}) \ltimes \R^d$ with an opportune group operation~\cite{Dahlke12}.

For a well chosen generator function $\psi \in L^2(\R^d)$ (\ie{}, satisfying the admissibility condition of theorem 1 in~\cite{Dahlke12}), the \textit{continuous shearlet transform} is given by:
\begin{equation}
\label{eq:shearCoeff}
    \msSH_{\psi}f(a,\vec{s},\vec{t}) := \<f, \psi_{a,\vec{s},\vec{t}} \> 
    \qquad \text{ for } \; f \in L^2(\R^d)
\end{equation}
where $\psi_{a,\vec{s},\vec{t}} = \pi(a,\vec{s},\vec{t}) \psi$ and $\<\cdot,\cdot\>$ denotes the inner product of $L^2(\R^d)$.
Thus, $\msSH_{\psi}$ analyzes the function $f$ around the location $\vec{t}$ at different resolutions and orientations encoded by the scale and shearing parameters $a$ and $\vec{s}$, respectively. Throughout this work, it is assumed that $\psi$ has compact support.

Shearlet frames have become a well studied research object in the last decade.
Of particular importance for our work are the results in \cite{Dahlke09}, in which shearlet coorbit spaces are introduced. We report here some key definitions and results on shearlet coorbit spaces, which we will later use to prove our main result. Our notation is closely alligned with that of~\cite{Dahlke09}.

Let $\omega(a,\vec{s},\vec{t}) = \omega(a,\vec{s}) \in L^1_{\text{loc}}(\R^d)$ be weight functions and consider the space
\[
    \msH_{1,\omega} = \{ f \in L^2(\R^d): \msSH_{\psi}(f) \in L^1_{\omega}(\bbS) \},
\]
and its anti-dual $\msH_{1,\omega}^{\sim}$. 
Notice that $\msH_{1,\omega} \hookrightarrow L^2(\R^d) \hookrightarrow \msH_{1,\omega}^{\sim}$ forms a Gelfand triplet. In particular, this means that the inner product
$\<\cdot,\cdot\>$ can be extended on the space  $\msH_{1,\omega}^{\sim} \times \msH_{1,\omega}$. Therefore, if $\psi \in \msH_{1,\omega}$ we can extend the definition~\eqref{eq:shearCoeff} of the shearlet transform also to $f \in \msH_{1,\omega}^{\sim}$.

Let 
\[
L^p_m(\bbS) := \{ 
    F \; \text{measurable} \; : \; Fm \in L^p(\bbS)
    \}
\]
where $m$ is a $\omega$-moderate weight on $\bbS$, \ie{}, $m(xyz) \leq \omega(x)m(y)\omega(z)$ for all $x,y,z \in \bbS$, and $1 \leq p \leq \infty$.
For $1 \leq p \leq \infty$ we can introduce the 
shearlet coorbit spaces as Banach spaces defined by
\begin{equation}
    \msSC_{p,m} := \{ 
    f \in \msH_{1,\omega}^{\sim} : \msSH_{\psi}(f) \in L^p_m(\bbS)
    \}
    \label{eq:shearlet_coorbit}
\end{equation}
with norms $\norm{f}_{\msSC_{p,m}} := \norm{\msSH_{\psi}(f)}_{L^p_m(\bbS)}$. 

Starting from the continuous transform, it is possible to derive a discrete shearlet system obtained by sampling the parameter space $\R \backslash \{0\} \times \R^{d-1} \times \R^d$ on a discrete subset $\Lambda$. We denote this system by $\{ \psi_{\lambda} = \pi((a,\vec{s},\vec{t})_{\lambda})\psi \; : \; \lambda \in \Lambda\}$ and introduce the following notation for the \textit{discrete shearlet transform} $\sh_{\psi}: X \rightarrow \ell^p_m$:
    \begin{equation}
    \label{eq:shearCoeffDiscr}
    \sh_{\psi}f := 
    \{ \<f, \psi_{\lambda} \>_{\msH_{1,\omega}^{\sim} \times \msH_{1,\omega}}\}_{\lambda \in \Lambda}.
    \end{equation}

Under suitable (and rather technical, see~\cite[theorem 4]{Dahlke12}) assumptions on $\psi$ and $\{\psi_{\lambda}\}_{\lambda \in \Lambda}$, it is possible to show that $\{\psi_{\lambda}\}_{\lambda \in \Lambda}$ provides a Banach frame for $\msSC_{p,m}$, namely:
\begin{itemize}
    \item[(i)] (\textbf{representation}) $f \in \msSC_{p,m}$  if and only if $\sh_{\psi}f \in \ell^p_m$. 
    \item[(ii)] (\textbf{norm equivalence})There exists two constants $0 < D \leq D' < \infty$ such that
    \[
    D \norm{f}_{\msSC_{p,m}} \leq 
        \norm{\sh_{\psi}f}_{\ell^p_m}
        \leq 
        D' \norm{f}_{\msSC_{p,m}}.
    \]
    \item[(iii)] (\textbf{reconstruction}) There exists a bounded, linear reconstruction operator $\sh_{\psi}^{\dag}$ from $\ell^p_m$ to $\msSC_{p,m}$ such that $\sh_{\psi}^{\dag} \left( \sh_{\psi}f \right) = f$.
\end{itemize}
In view of this characterization, if we now choose $X = \{ f \in \msSC_{p,m}: \supp(f) \subset \Omega\}$, being $\Omega = [0,1]^2$, equipped with the $\ell^p_m$-norm we have
\begin{equation} \label{eq:Brilliant_norm}
    \frac{1}{p} \norm{f}_X^p = 
    \frac{1}{p} \norm{\sh_{\psi}f}_{\ell^p_m}^p = 
    \frac{1}{p} \sum_{\lambda \in \Lambda} m_{\lambda}^p |\< f, \psi_\lambda\>_{\msH_{1,\omega}^{\sim} \times \msH_{1,\omega}}|^p
\end{equation}
which means that the regularization functional~\eqref{eq:ShearRegu} can be written as needed for our scopes. Notice that, with a slight abuse of notation, we use $m$ to denote the weight both in the continuous setting and as the sequence obtained by collecting $m_\lambda = m(a_\lambda,\vec{s}_\lambda,\vec{t}_\lambda)$ over the discrete set $\Lambda$.

\begin{remark}
\label{rm:BanachFramesExt}
The properties (i)-(iii) ensured by~\cite[theorem 4]{Dahlke12} allow us to consider $\{\psi_\lambda\}_{\lambda \in \Lambda}$ as a Banach frame for the Banach space $X = \msSC_{p,m}$, according to the definition provided in the seminal work~\cite{casazza1999frames}. The results contained in this section can be generalized from the shearlet case to any transform associated with the functions $\{\psi_\lambda\}_{\lambda \in \Lambda}$ for which it is possible to find a Banach space $X$ such that $\{\psi_\lambda\}_{\lambda \in \Lambda}$ is a Banach frame for it, in relation to a sequence space $\ell^p$ with $1< p \leq 2$. 
\end{remark}

\subsection{Convergence rates for shearlet regularization} 
\label{ssec:ConvRatesShearlets}
We are now ready to extend the results in~\cite{Bubba21} to shearlet-based regularization, possibly constrained to the non-negative orthant. From now on, we fix $X =\{ f \in \msSC_{p,m} :  \supp(f)\subset \Omega \}$, $\Omega = [0,1]^2$, equipped with the $\ell^p_m$-norm. 
Let $R(f) = \frac{1}{p}\norm{\sh_{\psi} f }_{\ell^p_m}^p$ and, for any $c \in \ell^p_m$, $\whR(c) = \frac{1}{p}\norm{c}_{\ell^p_m}^p$ so that $R(f) = \whR(\sh_{\psi}f)$. 

In view of \eqref{eq:Brilliant_norm}, this choice of $X$ ensures that the machinery of section \ref{sec:SettingStage} can be extended to the shearlet case. In particular, propositions \ref{prop:apriori} and \ref{prop:aux_convex2}, lemma \ref{lem:xuroach} and theorem \ref{thm:bregman_dist_gen} are immediately verified, and in order to prove theorem \ref{thm:general_rate} we only need to ensure that conditions \eqref{eq:ass_on_phomog_rate1}-\eqref{eq:ass_on_phomog_rate2} are verified also in the shearlet coorbit spaces.
We start by characterizing the convex conjugate $R^\star$ of $R$.

\begin{proposition}
Consider $p>1$. Then, for $y \in X^*$, we have
\[
R^{\star}(y) \leq \frac{1}{q} \| (\sh_{\psi}^{\dag})^* y \|_{\ell_{1/m}^q}^q.
\]
\end{proposition}
\begin{proof}
Recall the definition of the convex conjugate operator: for $y \in X^*$,
\begin{align*}
R^\star(y) &= 
\sup_{f \in X} \{ \< y,f\>_{X^*\times X} - \whR(\sh_{\psi} f)\} \\  
&= \sup_{c \in \imag(\sh_{\psi})} \left\{ \< y, \sh_{\psi}^{\dag} c \>_{X^*\times X} - \whR \left(\sh_{\psi} (\sh_{\psi}^{\dag} c) \right) \right\} 
\end{align*}
where $\imag(\sh_{\psi})\subset \ell^p_m$ denotes the range of the operator $\sh_{\psi}$ which does not necessarily coincide with $\ell^p_m$ since $\sh_{\psi}$ is not surjective. Nevertheless, if $c \in \imag(\sh_\psi)$, then $\sh_\psi \sh_\psi^\dag c = c$, hence
\[
R^\star(y) =
\sup_{c \in \imag(\sh_{\psi})} 
    \left\{ \< (\sh_{\psi}^{\dag})^* y, c\>_{(\ell^p_m)^*\times \ell^p_m} - 
    \whR \left( c \right) \right\},
\]
where $(\sh_{\psi}^{\dag})^* \colon X^* \rightarrow (\ell^p_m)^*$ is the (Banach) adjoint of the reconstruction operator. Notice also that $(\ell^p_m)^* = \ell^q_{1/m}$ and the pairing product $\< \cdot, \cdot \>_{(\ell^p_m)^*\times \ell^p_m}$ is the extension of the $\ell^2$ inner product. Finally,
\[
R^\star(y) \leq
\sup_{c \in \ell^p_m} 
    \left\{ \< (\sh_{\psi}^{\dag})^* y, c \>_{(\ell^p_m)^*\times \ell^p_m} - 
    \whR \left( c \right) \right\} = \whR^\star((\sh_\psi^\dag)^*y).
\]
For the sake of completeness, we also derive the expression of the convex dual $\whR^\star \colon \ell^q_{1/m} \rightarrow \R$. Consider a sequence $d \in \ell^q_{1/m} = (\ell^p_m)^*$: we have
\begin{align*}
\whR^\star(d) &= \sup_{c \in \ell^p_m} \left\{ 
\< d, c \>_{(\ell^p_m)^*\times \ell^p_m} - \whR(c) \right\} \\
& = \sup_{c \in \ell^p_m} \left\{ 
\sum_{\lambda \in \Lambda} \left( d_\lambda c_\lambda - \frac{1}{p} m_\lambda^p |c_\lambda|^p \right)
\right\}.
\end{align*}
Since the functional $J(c) = \< d, c \>_{(\ell^p_m)^*\times \ell^p_m} - \whR(c)$ is concave and differentiable, its maximum is achieved by the sequence $c^*$ which satisfies $\nabla J(c^*) = 0$. Each component of the gradient of $J$ can be computed as follows: 
\[ 
[\nabla J(c)]_\lambda = \partial_{c_\lambda} J(c) = d_\lambda - m_\lambda^p \sign(c_\lambda)|c_\lambda|^{p-1}.
\]
Therefore,  $\nabla J(c^*) = 0$ can be solved component-wise to get
\[
c^*_\lambda = m_\lambda^{-\frac{p}{p-1}} \sign(d_\lambda)|d_\lambda|^{\frac{1}{p-1}}.
\]
As a consequence 
\[
\begin{aligned}
\whR^\star(d) = J(c^*) &= \sum_{\lambda \in \Lambda} \left(  m_\lambda^{-\frac{p}{p-1}} \sign(d_\lambda)|d_\lambda|^{\frac{1}{p-1}} d_\lambda - \frac{1}{p} m_\lambda^p  m_\lambda^{-\frac{p^2}{p-1}} |d_\lambda|^{\frac{p}{p-1}} \right) \\
& = \sum_{\lambda \in \Lambda} \left(  m_\lambda^{-\frac{p}{p-1}} |d_\lambda|^{\frac{p}{p-1}} - \frac{1}{p} m_\lambda^{-\frac{p}{p-1}} |d_\lambda|^{\frac{p}{p-1}} \right) = \frac{1}{q}  \sum_{\lambda \in \Lambda} m_\lambda^{-q} |d_\lambda|^{q}
\end{aligned}
\]
In conclusion, for any sequence $d \in (\ell^p_m)^*$, $\whR^\star(d) = \frac{1}{q} \norm{d}_{\ell^q_{1/m}}^q$, which allows to conclude the thesis.
\end{proof}

In order to state our main result, we need to make the following assumptions.
\begin{itemize}
    \item[(S1)] The ground truth $f^\dag\in X$ satisfies a \emph{classical source condition},
    \begin{equation}
	 \exists \; w \in Z \; \text{ s.t. } \;  r^\dag = A_\mu^* w
	 \qquad \text{ where } \; r^\dag = \partial R(f^\dag).
	\label{eq:SC_Coorbit}
    \end{equation}
    \item[(S2)] Let $\{e_j\}_{j \in \N}$ be the canonical basis of $\ell^p$, \ie{}, the sequences defined by $[e_j]_i = \delta_{ij}$.   Then, the operator $A$ satisfies
    \begin{equation*}
	\sum_{j=1}^\infty m_j^{-q} \norm{A \sh_{\psi}^\dag e_j}_{\infty}^q < \infty.
    \end{equation*}
\end{itemize}
Notice that the set $\{ \sh_{\psi}^\dag e_j\}_j$ can be interpreted as a dual frame of $\{\psi_\lambda\}_\lambda$ on $X$. Indeed, for every $f \in X$, it holds $f = \sh_\psi^\dag \sh_\psi f$ and, by representing $\sh_\psi f$ in the orthogonal basis $\{e_j\}_j$, we have
\[
f = \sh_\psi^\dag \bigg( \sum_{j =1}^\infty [\sh_{\psi} f]_j e_j \bigg) = \sum_{j =1}^\infty \< f,\psi_j \>_{\msH_{1,\omega}^{\sim} \times \msH_{1,\omega}} \sh_\psi^\dag e_j,
\]
which amounts to saying that $\{\sh_\psi^\dag e_j\}_j \subset X$ is a dual frame of the shearlet frame. We nevertheless point out that $\{\sh_\psi^\dag e_j\}_j $ is not a shearlet frame.
    
We can now replicate the result \cite[theorem 5.4]{Bubba21} regarding the decay of the quantity
\[
\mathscr{R}(\beta,\bu;f^\dag) = \inf_{\bar{w} \in V_N} \left\{ R^\star(r^\dag - \Abu^*\bar{w}) +\frac{\beta}{2}\|\bar{w}\|_{V_N}^2\right\},
\]

\begin{proposition}
\label{prop:shearlet_rate1}
Under assumptions (S1)-(S2), we have
\[
\mathbb{E}[\mathscr{R}(\beta,\bu,f^\dag)] \lesssim N^{-\frac{q}{2}} + \beta.
\]
\end{proposition}
\begin{proof}

Thanks to the source condition (S1), $r^\dag = A_\mu^* w$ for some $w \in Z$: hence, we can consider $\bar{w} = \Sbu w$ in~\eqref{eq:Rbeautiful} and deduce
\begin{equation*}
	\mathscr{R}(\beta, \bu; f^\dag) \leq 
	\frac{1}{q} \norm{(\sh_\psi^\dag)^*(A_\mu^* - \Abu^* \Sbu)w}_{\ell^q_{1/m}}^q + \frac{\beta}{2} \norm{\Sbu w}_{V_N}^2.
\end{equation*}
The expectation of the second term coincides with $\frac{\beta}{2} \norm{w}_{Y_\mu}^2$ and can be bounded by (a constant times) $\beta$ due to the continuous embedding of $Z$ into $Y_\mu$.
For the first term, we can write each component of the sequence as
\begin{align*}
\< (\sh_{\psi}^\dag)^* 
    (A_\mu^*-\Abu^* \Sbu)w, e_j \>_{(\ell_m^p)^*\times \ell_m^p} &= 
    \< (A_\mu^*-\Abu^* \Sbu)w, \sh_{\psi}^\dag e_j \>_{X^*\times X} \\
    &= \< w, A_\mu \sh_{\psi}^\dag e_j \>_{Y_{\mu}} 
    - \< \Sbu w, \Abu \sh_{\psi}^\dag e_j \>_{V_N} \\
    &= \frac{1}{N} \sum_{n=1}^N \left( \< w, A_\mu \sh_{\psi}^\dag e_j\>_{Y_{\mu}} - 
    \< w(u_n), (A \sh_{\psi}^\dag e_j)(u_n) \>_{V_N}
    \right) 
\end{align*}
We now set $\xi^{j}_n = \< w, A_\mu \sh_{\psi}^\dag e_j\>_{Y_{\mu}} - 
    \< w(u_n), (A \sh_{\psi}^\dag e_j)(u_n) \>_{V_N}$: then, the random variables $\xi^{j}_n$ are zero-mean and i.i.d.
Furthermore, by applying the Cauchy-Schwarz's inequality we have that
\begin{equation*}
 \left|	\< w(u), (A \sh_{\psi}^\dag e_j)(u)\>_V \right| \leq \norm{w}_\infty \norm{A \sh_{\psi}^\dag e_j}_{\infty}
\end{equation*}
for any $u\in U$. Therefore, for each $j$ the random variables $\xi^j_n$, $n=1,...,N$, are bounded uniformly according to
\begin{equation*}
	\< w, A \sh_{\psi}^\dag e_j \>_{Y_\mu} - \norm{w}_\infty \norm{A \sh_{\psi}^\dag e_j}_{\infty} \leq \xi^{j}_n \leq \langle w, A \sh_{\psi}^\dag e_j \rangle_{Y_\mu} + \norm{w}_\infty \norm{A \sh_{\psi}^\dag e_j}_{\infty},
\end{equation*}
and the terms $\norm{A \sh_\psi^\dag e_j}_{\infty}$ are uniformly bounded thanks to assumption (S2).
As in \cite[theorem 5.4]{Bubba21}, the proof is ended by the Hoeffding's inequality for bounded random variables,
\[
\begin{aligned}
	\E \left[ \frac{1}{q}\norm{(\sh_\psi^\dag)^* 
	    (A_\mu^* - \Abu^* \Sbu)w)}^q_{\ell^q_{1/m}} \right] & = 
	\frac{1}{q} \sum_{j=1}^\infty N^{-q}m_j^{-q} \mathbb{E} \left[ \left|\sum_{n=1}^N \xi^{j}_n\right|^q \right] \\
	&= \frac{1}{q} \sum_{j=1}^\infty N^{-q}m_j^{-q} \int_0^\infty t^{q-1} \ \mathbb{P}\left(\left|\sum_{n=1}^N \xi^{j}_n\right| > t\right) dt \\
	& \leq  \frac{2}{q} \sum_{j=1}^\infty  N^{-q}m_j^{-q} \int_0^\infty t^{q-1} \exp\left( -\frac{t^2}{2N \norm{w}_\infty^2 \norm{A \sh_\psi^\dag e_j}_{\infty}^2 } \right) dt	\\
	& = \frac{2}{q} \sum_{j=1}^\infty  N^{-\frac{q}{2}}m_j^{-q} \norm{w}_\infty^q \norm{A \sh_\psi^\dag e_j}_{\infty}^q \int_0^\infty s^{q-1}\exp \left(-\frac{1}{2} s^2 \right) ds \\
    & \lesssim N^{-\frac{q}{2}}  \sum_{j=1}^\infty  m_j^{-q} \norm{A \sh_\psi^\dag e_j}_{\infty}^q.
\end{aligned}
\]
\end{proof}

Analogously, we can extend \cite[proposition 5.5]{Bubba21} to the current setting as follows.
\begin{proposition}
\label{prop:shearlet_rate2}
Under assumptions (S1)-(S2), we have that
\begin{equation*}
	\E [R^\star (\Abu^* \epsilon_N)] \lesssim N^{-q/2}.
\end{equation*}
\end{proposition}

\begin{proof}
Proceeding as above, $R^\star(\Abu^*\epsilon_N) = \frac{1}{q}\norm{(\sh_\psi^\dag)^*\Abu^*\epsilon_N}^q_{\ell^q_{1/m}}$ and we can write each component of the sequence $(\sh_\psi^\dag)^*\Abu^*\epsilon_N$ as follows:
\begin{equation*}
	\< (\sh_\psi^\dag)^* \Abu^* \epsilon_N, e_j\>_{(\ell^p_m)^*\times \ell^p_m} = \< \Abu^* \epsilon_N, \sh_\psi^\dag e_j\>_{X^*\times X}  = 
	\frac{1}{N} \sum_{n=1}^N \< \epsilon_N^n, A_{u_n} \sh_\psi^\dag e_j \>_V =: 
	\frac{1}{N} \sum_{n=1}^N \tilde \xi_n^j.
\end{equation*}
The random variables $\tilde \xi_n^j$ are independent and zero-mean, since $\epsilon_N^n$ is zero-mean and independent of $u_n$. By assumption (A1) and \eqref{eq:randomnoise}, we have that $\tilde \xi_n^j$ are also sub-Gaussian random variables and
\begin{equation*}
	\nsG{\tilde \xi_n^j} := \inf \left\{t>0 \; \big| \; \E \left[ \exp\left(\frac{(\tilde \xi_n^j)^2}{t^2}\right) \right] \leq 2 \right\} \\
	\leq \norm{A \sh_\psi^\dag e_j}_{\infty} \nsG{\epsilon_N^n}.
\end{equation*}
Notice that the terms $\norm{A \sh_\psi^\dag e_j}_{\infty}$ are uniformly bounded thanks to assumption (S2).
By applying the Hoeffding's inequality for sub-Gaussian random variables, we get
\[
\begin{aligned}
	\E [R^\star(\Abu^* \epsilon_N)] &= \frac{1}{q}\sum_{j=1}^\infty N^{-q} m_j^{-q} \, \E \left[ \left|\sum_{n=1}^N \tilde \xi^j_n\right|^q \right] \\
	&\leq  \frac{2}{q} \sum_{j=1}^\infty N^{-q} m_j^{-q} \int_0^\infty t^{q-1} \exp\left(-\frac{C t^2}{N \norm{A \sh_\psi^\dag e_j}_{\infty}^2} \right) dt \\
	&\leq  \frac{2}{q} \sum_{j=1}^\infty C^{-\frac q2} N^{-\frac{q}{2}}m_j^{-q} \norm{A \sh_\psi^\dag  e_j}_{\infty}^q \int_0^\infty s^{q-1}\exp\left(-\frac{1}{2} s^2\right) ds \\ 
	&\lesssim N^{-\frac{q}{2}} \sum_{j=1}^\infty m_j^{-q} \norm{A \sh_\psi^\dag e_j}_{\infty}^q.
\end{aligned}
\]
\end{proof}

Combining propositions \ref{prop:shearlet_rate1} and \ref{prop:shearlet_rate2}, we can finally obtain a version of theorem~\ref{thm:general_rate} suited for the case of shearlet-based regularization. With an argument analogous to the one leading to theorem~\ref{thm:general_rate_nonneg} (see, in particular, \eqref{eq:DistBreg_nonneg}), we directly formulate the convergence rate for the case constrained to the non-negative orthant, that is, $\wtR(f) = R(f) + \iota_{+}(f)$.
\begin{theorem} \label{thm:general_rate_shearlet}
Suppose assumptions (A1)-(A2) are verified. Let $\wtR(f) = R(f) + \iota_{+}(f)$ and assume that $f^\dag \geq 0$ a.e. Let $X = \msSC_{p,m}$ and suppose that (S1)-(S2) hold true.
Then, we have the following convergence rates, as $N \rightarrow \infty$:
\begin{itemize}
    \item if $\delta N \rightarrow \infty$ (and $\delta^2/N \rightarrow 0$), then
    \begin{equation}
	\label{eq:param_choice_p_fixed_s}
	\E \left[D_{\wtR}(\fdan, f^\dag) \right] \lesssim  \left( \frac{\delta^2}{N}\right)^{\frac{1}{3}} \quad \text{for} \quad \alpha \simeq  \left( \frac{\delta^2}{N}\right)^{\frac{1}{3}};
    \end{equation}
    \item if $\delta N$ is bounded, then 
    \begin{equation}
	\label{eq:param_choice_p_satur_s}
	\E \left[D_{\wtR}(\fdan, f^\dag) \right] \lesssim  N^{-1} \quad \text{for} \quad \alpha \simeq N^{-1}.
    \end{equation}
\end{itemize}
\end{theorem}

To conclude this subsection, we show that, by choosing $A$ as the semidiscrete Radon transform $A = \radonSD$, the assumptions of theorem~\ref{thm:general_rate_shearlet} are verified. In particular, according to the discussion in subsection \ref{ssec:radon}, assumption (A1) is verified since $X = \msSC_{p,m} \subset L^1(\Omega)$. Assumption (S2) is verified in the following lemma. We recall that in this application we set $Z = \mathcal{C}([0,2\pi);\R^{\Ndtc})$, and therefore $\| \cdot \|_Z = \| \cdot \|_{\infty}$.

\begin{lemma}
\label{lem:Apsi}
The dual frame $\{\sh_\psi^\dag e_j\}_j$ and the  operator $\radonSD$ satisfy
    \begin{equation}
    \sum_{j=1}^{\infty} m_j^{-q} \| \radonSD \sh_\psi^\dag e_j \|_\infty^q < \infty.
        \label{eq:Aframe_inf}
    \end{equation}
\end{lemma}
\begin{proof}
For any $f \in X$, we have $\radonSD f \in Z = \mathcal{C}(U,V) = \mathcal{C}([0,2\pi);\R^{\Ndtc})$. Therefore, we use~\eqref{eq:Radon_model_eqL} and the Sobolev embedding of $H^1([0,2\pi);\R^{\Ndtc})$ into $\mathcal{C}([0,2\pi);\R^{\Ndtc})$ to conclude that
\begin{equation}
  \| \radonSD f\|_\infty \leq C_{S} \| \radonSD f \|_{H^1((0,2\pi);\R^{\Ndtc})} = C_{S} \left( \sum_{i=1}^{\Ndtc} \left\| \int_{\Omega} f(x) \rho_i(x,\cdot)dx \right\|_{H^1(0,2\pi)}^2 \right)^{1/2}.
\label{eq:aux_ineq}  
\end{equation}
Denoting by $h_i(\theta) = \int_{\Omega} f(x) \rho_i(x,\theta)dx$ and by $\langle \cdot, \cdot \rangle$ the scalar product in $L^2(\Omega)$ we have that
\begin{equation}
\| h_i \|_{H^1(0,2\pi)}^2 =\| h_i \|_{L^2(0,2\pi)}^2 +\| \nabla_\theta h_i \|_{L^2(0,2\pi)}^2 = \int_0^{2\pi} \langle f, \rho_i(\cdot,\theta) \rangle^2 d\theta + \int_0^{2\pi} \langle f, \nabla_\theta \rho_i(\cdot,\theta) \rangle^2 d\theta.
\label{eq:aux_eq}    
\end{equation}
Now, apply \eqref{eq:aux_ineq} and \eqref{eq:aux_eq} in the following sum, each time substituting $f$ with $\sh_\psi^\dag e_j$: as a result,
\begin{align*}
\sum_{j=1}^\infty m_j^{-q}\| \radonSD \sh_\psi^\dag e_j \|_\infty^q &\lesssim \sum_{j=1}^\infty m_j^{-q} \left( \sum_{i=1}^{\Ndtc} \left( \| h_i \|_{L^2(0,2\pi)}^2 +\| \nabla_\theta h_i \|_{L^2(0,2\pi)}^2 \right) \right)^{q/2} \\
&\lesssim \sum_{j=1}^\infty \sum_{i=1}^{\Ndtc} m_j^{-q} \left(
\| h_i \|_{L^2(0,2\pi)}^q +\| \nabla_\theta h_i \|_{L^2(0,2\pi)}^q \right) \\
&= \sum_{j=1}^\infty \sum_{i=1}^{\Ndtc} m_j^{-q} \left( \left( \int_0^{2\pi} \langle \rho_i(\cdot,\theta), \sh_\psi^\dag e_j \rangle^2 d\theta \right)^{q/2} + \left( \int_0^{2\pi} \langle \nabla_\theta \rho_i(\cdot,\theta), \sh_\psi^\dag e_j \rangle^2 d\theta \right)^{q/2} \right) \\
&\lesssim \sum_{j=1}^\infty \sum_{i=1}^{\Ndtc} \left( m_j^{-q} \int_0^{2\pi} |\langle \rho_i(\cdot,\theta), \sh_\psi^\dag e_j \rangle|^q d\theta + m_j^{-q} \int_0^{2\pi} |\langle \nabla_\theta \rho_i(\cdot,\theta), \sh_\psi^\dag e_j \rangle|^q d\theta \right) \\
&= \sum_{i=1}^{\Ndtc} \int_0^{2\pi} \left( \sum_{j=1}^\infty m_j^{-q} |\langle \rho_i(\cdot,\theta), \sh_\psi^\dag e_j \rangle|^q + \sum_{j=1}^\infty m_j^{-q} |\langle \nabla_\theta \rho_i (\cdot,\theta), \sh_\psi^\dag e_j \rangle|^q \right)d\theta 
\end{align*}
where we used Jensen's inequality both for finite sums and for the definite integral, since by assumption $q/2> 1$. Both sums inside the integral can be treated as follows: for a generic element $y \in X^*$,
\[
\begin{aligned}
    \sum_{j=1}^\infty m_j^{-q} |\< y, \sh_\psi^\dag e_j \>_{X^* \times X}|^q &= \sum_{j=1}^\infty m_j^{-q} |\< (\sh_\psi^\dag)^* y, e_j \>_{(\ell^p_m)^* \times \ell^p_m}|^q = \norm{(\sh_\psi^\dag)^*y}_{(\ell^p_m)^*}^q \\
    & \leq \norm{(\sh_\psi^\dag)^*}_{X^* \rightarrow (\ell^p_m)^*}^q \norm{y}_{X^*}^q = \norm{\sh_\psi^\dag}_{\ell^p_m \rightarrow X}^q \norm{y}_{X^*}^q.
\end{aligned}
\]
Finally, it is possible to consider both $\rho_i(\cdot,\theta)$ and $\nabla_\theta \rho_i(\cdot,\theta) \in X$ as elements of $X^*$. Indeed, by the characterization of the dual of coorbit spaces provided in \cite[theorem 4.9]{Feichtinger89}, we have
\[
\msSC_{p,m} = \{ 
    f \in \msH_{1,\omega}^{\sim} : \msSH_{\psi}(f) \in L^p_m(\bbS)\} \quad \Rightarrow \quad
(\msSC_{p,m})^* = \{ 
    f \in \msH_{1,\omega}^{\sim} : \msSH_{\psi}(f) \in L^q_{1/m}(\bbS)\},
\]
and by the smoothness of $\rho$ we deduce that $\| \rho_i(\cdot,\theta) \|_{X^*},\| \nabla_\theta \rho_i(\cdot,\theta) \|_{X^*} \lesssim \| \rho \|_{C^1}$. As a result,
\begin{align*}
\sum_{j=1}^\infty m_j^{-q} \| \radonSD \sh_\psi^\dag e_j \|_\infty^q
& \lesssim \sum_{i=1}^{\Ndtc} \int_0^{2\pi} \left( \| \rho_i(\cdot, \theta) \|_{X^*}^q + \| \nabla_\theta\rho_i (\cdot, \theta) \|_{X^*}^q\right) d\theta\\ 
& \leq C(\Omega, \| \rho\|_{C^1},\{|I_i|\},N_{dtc}) < \infty.
\end{align*}
\end{proof}

\subsection{Experiments and results}
\label{ssec:ShearletsNumerics}
Similarly to subsection~\ref{ssec:NonnegConstr_numerics}, we now verify numerically the expected convergence rates proven in theorem \ref{thm:general_rate_shearlet} for the two noise scenarios of interest: decreasing and fixed noise. Notice that in this case we have $X = \{f \in \msSC_{p,m}: \supp(f) \in \Omega \}$, which we can again discretize with $\R^{\Npxl}$. The shearlet operator $\shD \in \R^{\Nsh \times \Npxl}$, with $\Nsh = \sigma \Npxl  > \Npxl$, is implemented with ShearLab~\cite{Kutyniok16}, using a compactly supported generator function, with three scales, $\sigma = 33$ subbands and weights $m \equiv 1$. This choice of the shearlet parameters is compliant with the theoretical framework used to prove the estimates in theorem \ref{thm:general_rate_shearlet}.
All the other settings (namely, the parameters $\Npxl$, $N_{\min}$, $N_{\max}$, $\delta$ and $c_{\delta}$ and the operators $\RadonD$, $\RadonD_{\thetab}$ with their adjoints) are as in subsection~\ref{ssec:NonnegConstr_numerics}. The regularization parameter $\alpha$ is heuristically determined in each experiment, by means of $c_{\alpha}$. As in the wavelet case, the sample averages are computed using 30 random realizations. 
Reconstructions are computed using VMILA (see, in particular, equations~\eqref{eq:VMILAiter} and \eqref{eq:VMILADualObjConstrained} in appendix~\ref{ssec:VMILAnonneg}), where we set $\Mop = \shD \in \R^{\sigma \Npxl \times \Npxl}$.

Finally, similarly to the wavelet regularization case, the phantom should satisfy the source condition (S1), associated with shearlets. This is an even more challenging task, as the discussion below shows.

\subsubsection{Source conditions} 
\label{ssec:SC_shearlets}
In the case of shearlet regularization, generating a phantom satisfying the source condition (S1) is a complicated task. Indeed, for $p>1$, \eqref{eq:SC_Coorbit} can be formulated in the discrete setting as follows:
\begin{equation}
\exists  \; \w \in \R^{\Ndtc\Nth} \qquad \text{s.t.} \quad \shD^{\text{T}} (\shD \f^{\dag})^{[p-1]} = \RadonD^{\text{T}} \w. 
\label{eq:SourceCondDiscr_shearlets}
\end{equation}
As a result, starting from a phantom of interest $\f_0$, it is still possible to compute an element $\w$ which solves a regularized version of \eqref{eq:SourceCondDiscr_shearlets}, as in \eqref{eq:SourceCondTik}. Nevertheless, it is not possible to straightforwardly invert \eqref{eq:SourceCondDiscr_shearlets} to recover a phantom $\f^\dag$ which would be exactly associated with $\w$. Indeed, we can still split this problem into two linear systems $\shD^{\text{T}}\y = \RadonD^{\text{T}}\w$ and $\shD \f^\dag = \y^{\left[ \frac{1}{p-1}\right]}$, but since the adjoint matrix $\shD^{\text{T}}$ is not left-invertible, the first subproblem can only be solved if $\RadonD^{\text{T}}\w \in \imag(\shD^{\text{T}})$, which is not true in general. Therefore, computing $\f^\dag$ from $\w$ in \eqref{eq:SourceCondDiscr_shearlets} should be addressed as a nonlinear system of equations, whose well-posedness is unclear and whose solution could only be approximated numerically. As a side note, notice that it is actually possible to solve such problem in the case $p=2$, since the system reduces to $\shD^{\text{T}} \shD \f^\dag = \RadonD^{\text{T}}$ which has the unique solution $\f^\dag = (\shD^{\text{T}}\shD)^{-1}\RadonD^{\text{T}}\w = \shD^\dag (\shD^\dag)^{\text{T}} \RadonD^{\text{T}}\w$.

As a consequence, in the case of shearlet regularization, we can only rely on the approximate source condition: namely, for a prescribed phantom $\f^\dag = \f_0$ we should check if a discrete version of condition \eqref{eq:ass_on_phomog_rate1} is satisfied. By means of algorithm \ref{algo:approxSC} (with $K=30$), it is easy to assess that the plant phantom satisfies the approximate source condition for the values $p=3/2$ and $p=4/3$, which we used in the numerical experiments.

\subsubsection{Discussion}

\begin{figure}
    \centering
    \begin{tabular}{@{}c@{\;}c@{\;}c@{\;}c@{}}
        \multicolumn{2}{c}{fixed noise}
        & \multicolumn{2}{c}{decreasing noise} \\
        $p=3/2$ & $p=4/3$
        & $p=3/2$ & $p=4/3$ \\
        \includegraphics[width=0.25\textwidth]{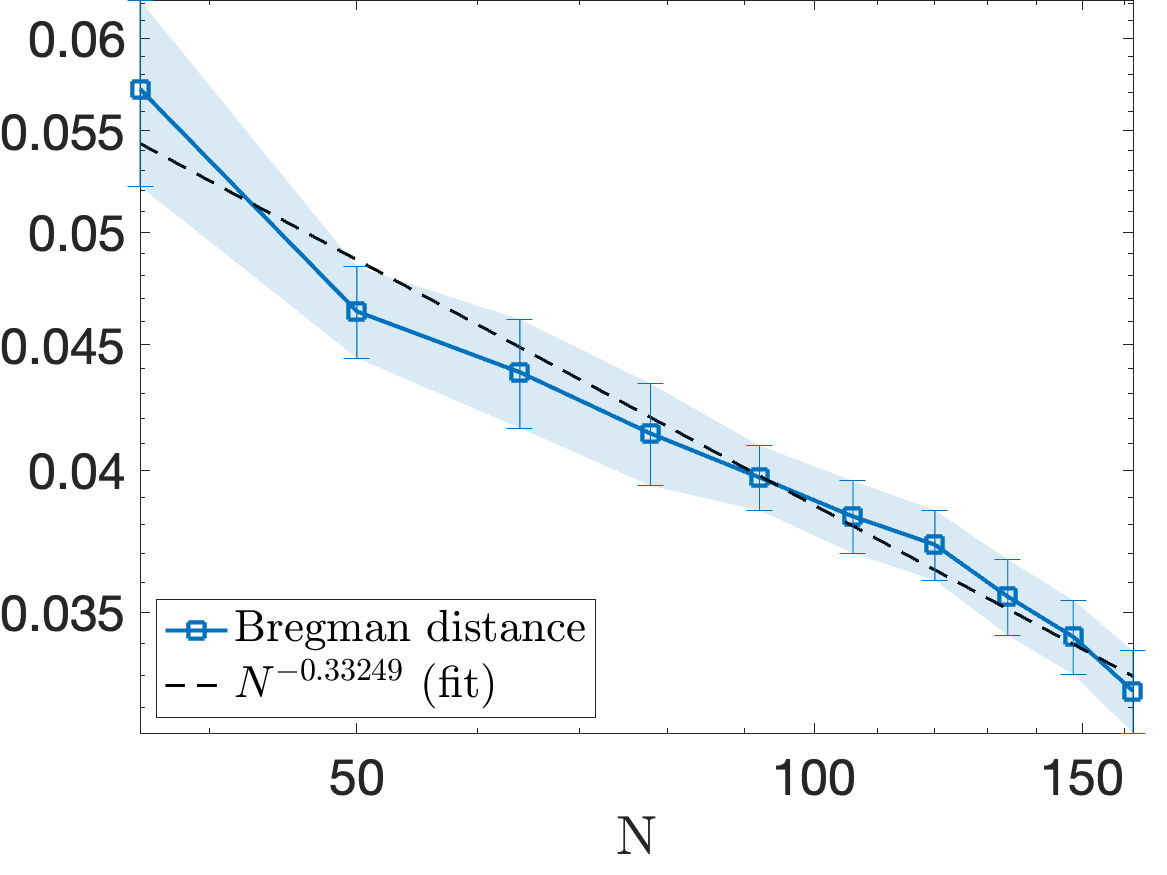} 
        & \includegraphics[width=0.25\textwidth]{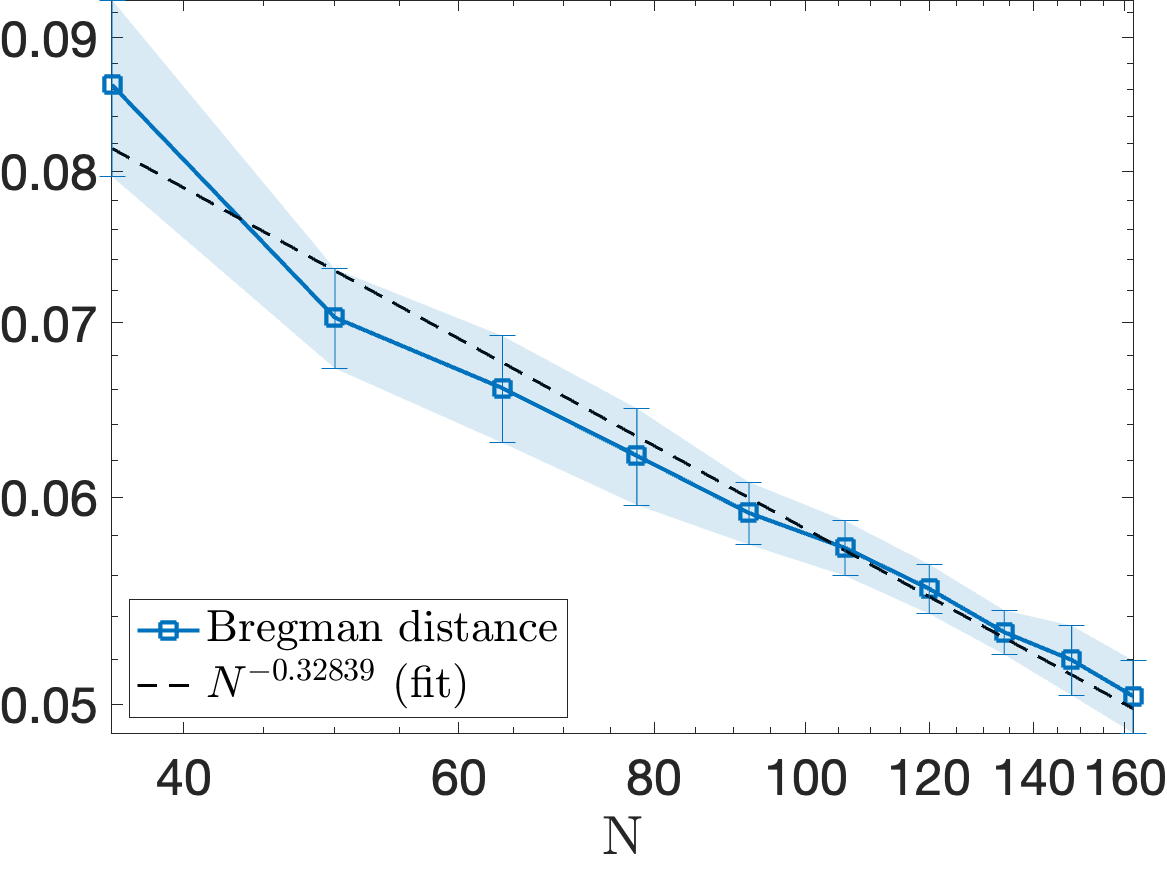} 
        & \includegraphics[width=0.25\textwidth]{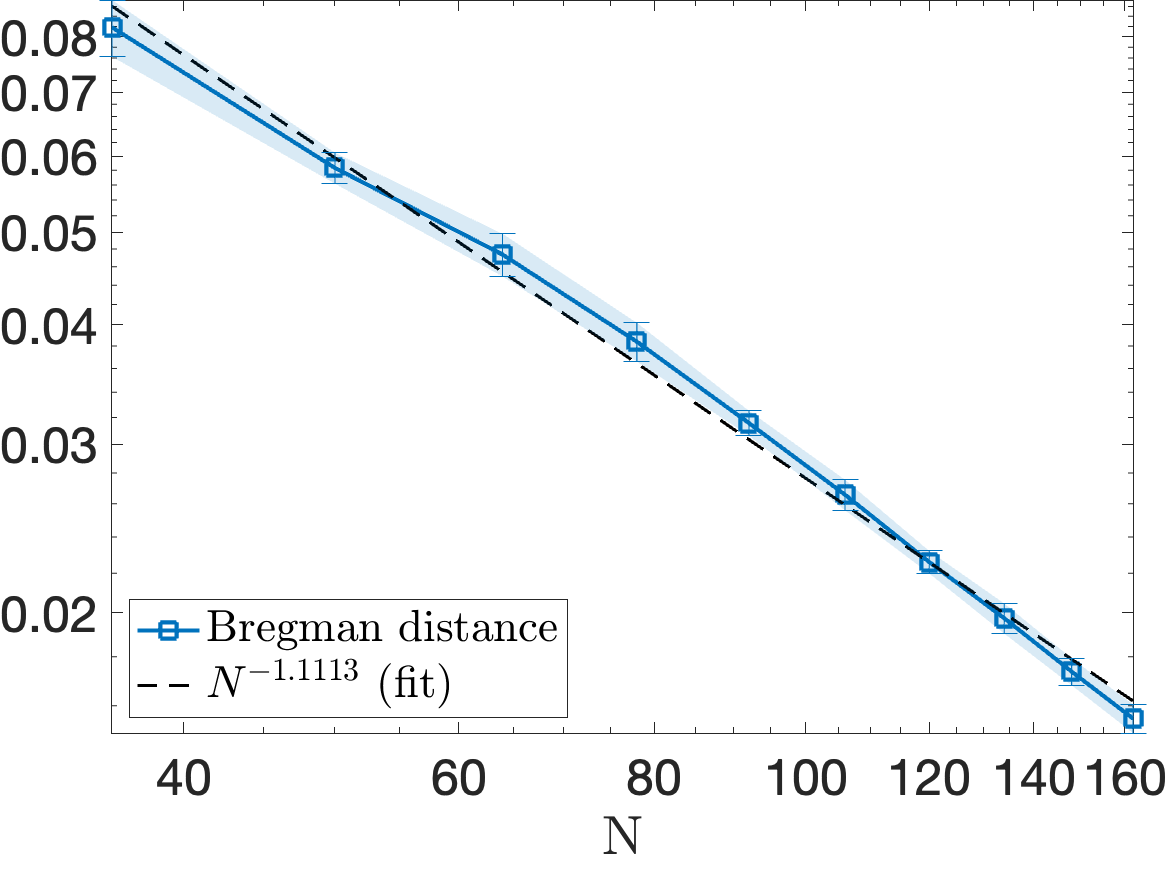}
        & \includegraphics[width=0.25\textwidth]{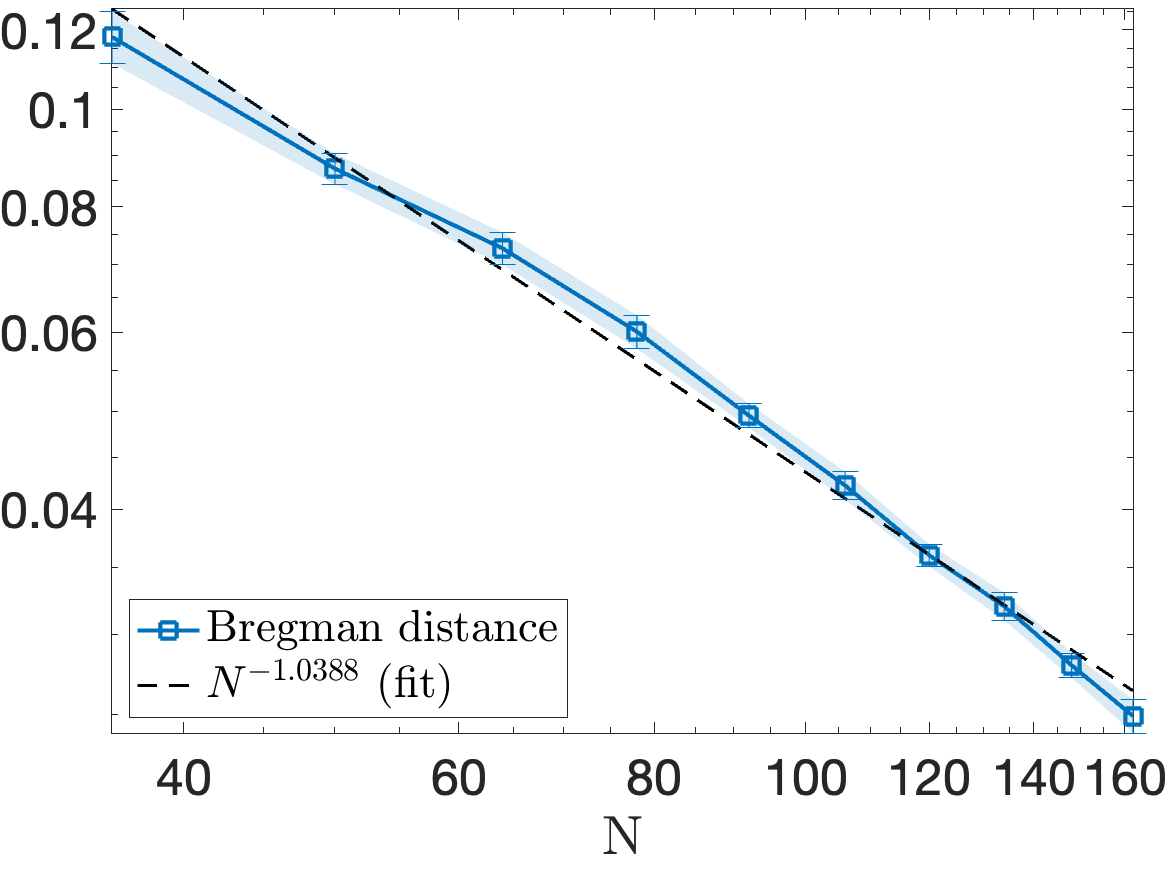} \\
        (a) & (b) & (c) & (d) 
    \end{tabular}
    \caption{Approximate decay of the expected value of the Bregman distance, with shearlets-based regularization, for $p=3/2$ ((a) and (c)) and $p=4/3$ ((b) and (d)). The phantom verifies the approximate source conditions. (a)\&(b): Fixed noise regime. (c)\&(d): Decreasing noise regime.}
    \label{fig:Plant_shearlets_appSC_NoWeights}
\end{figure}
\begin{table}[t]
\centering
\caption{Approximate decay of the expected value of the Bregman distance, with shearlets-based regularization, for $p=3/2$ and $p=4/3$.}
\begin{tabular}{c|c|cc}
scenario & theoretical & $p = 3/2$ & $p=4/3$   \\
\hline
decreasing noise & $-1$ & $-1.1113$ & $-1.0388$  \\
fixed noise & $-1/3$ & $-0.33249$ & $-0.32839$ 
\end{tabular}
\label{tab:decaysShearlets}
\end{table}

In figure \ref{fig:Plant_shearlets_appSC_NoWeights}  we report the value of the expected Bregman distance $\E[ D_{\vec{\wtR}}(\f^\delta_{\alpha,N}, \f^\dag)]$ (blue solid line) as a function of $N$, for the plant phantom $\f^\dag$ satisfying the  approximate source condition. Similarly to the wavelet case, the shaded region in figure \ref{fig:Plant_shearlets_appSC_NoWeights} encompasses the standard deviation error bars, while the black dashed line is the best monomial approximation $c N^{\beta}$ to the expected Bregman distance decay determined numerically.
We consider both decreasing and fixed noise regimes, with both $p=3/2$ and $p=4/3$. 

Analogously to the wavelet case, the theoretically predicted decays~\eqref{eq:param_choice_p_fixed_s} and \eqref{eq:param_choice_p_satur_s} are numerically verified, as the results in table~\ref{tab:decaysShearlets} and figure \ref{fig:Plant_shearlets_appSC_NoWeights} show. For the fixed noise regime, unlike in the wavelet case,  we notice less oscillations around the mean, especially for $p=4/3$.

\medskip

As an incidental remark, notice that the numerical tests in subsections~\ref{ssec:NonnegConstr_numerics} and \ref{ssec:ShearletsNumerics} use $p=3/2$ and $p=4/3$ to allow a comparison with the numerical study in~\cite{Bubba21}. However, with VMILA is rather straightforward to consider \textit{any} other $p \in (1,2)$. This observation was the starting point for the experiments in section~\ref{sec:Approachingp1}.

\section{Approaching the case $p = 1$}
\label{sec:Approachingp1}

Over the last two decades, a common paradigm to solve tomographic inverse problems has been to consider  sparsity-enforcing penalties. In particular, there has been widespread interest in regularization by $\ell^1$-norm of wavelet or shearlet coefficients. 
This corresponds to our setup by setting $p=1$ and 
$X=B_1^s$ for wavelet-based regularization or $X=\msSC_{1,m}$ for shearlet-based regularization.

Notice, however, that in this case a series of complications arise that hinder the straightforward application of the theory developed so far. For example, the Bregman distance is no longer uniquely defined and, therefore, in some cases it might not be the ideal metric to derive concentration rates. Also, despite the fact that our main result is independent of $p$ (and its H\"{o}lder conjugate $q$), all the lemmata and propositions used to prove theorems~\ref{thm:general_rate_nonneg} and~\ref{thm:general_rate_shearlet} depend on $p$ and $q$.

As a consequence, we need \textit{ad hoc} strategies, and possibly a different perspective, to deal with this case. Therefore, in this section, we propose some strategies that numerically show that we can expect the same convergence rates as in the $1 < p < 2$ case and provide a partial theoretical analysis which nicely complement the numerical study.

\subsection{Numerical intuition}
\label{ssec:TMconv}
\begin{figure}
    \centering
    \begin{tabular}{@{}c@{\;}c@{\;}c@{\;}c@{}}
        \multicolumn{2}{c}{fixed noise}
        & \multicolumn{2}{c}{decreasing noise} \\
        $p=1.1$ & $p=1.01$
        & $p=1.1$ & $p=1.01$ \\
        \includegraphics[width=0.25\textwidth]{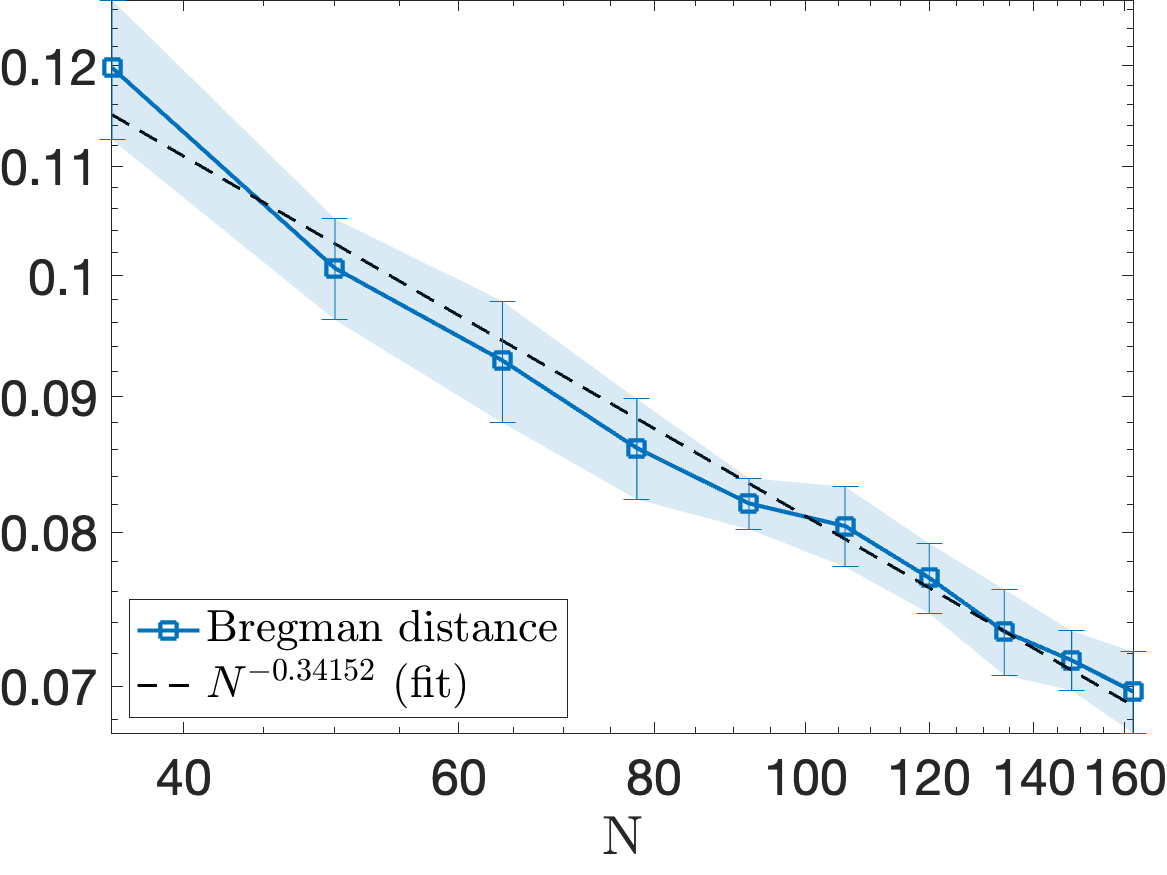} 
        & \includegraphics[width=0.25\textwidth]{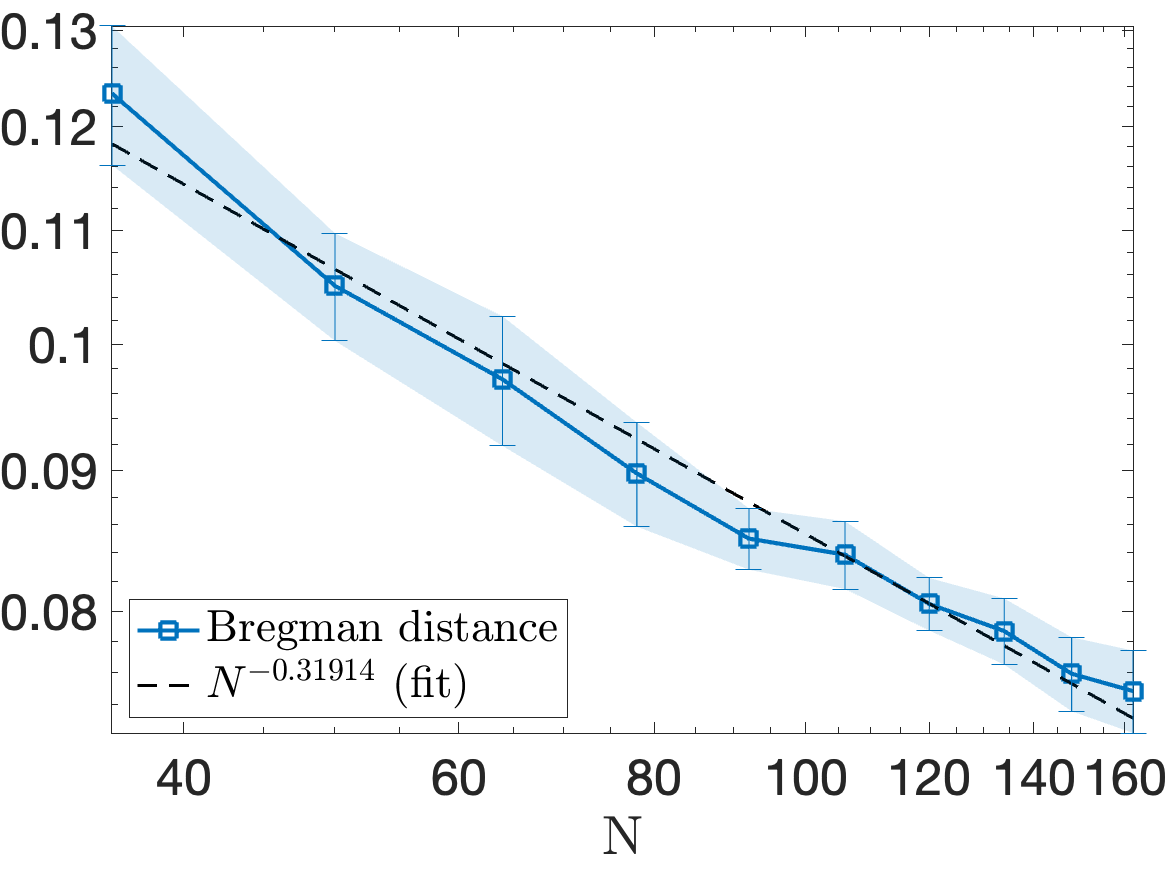} 
        & \includegraphics[width=0.25\textwidth]{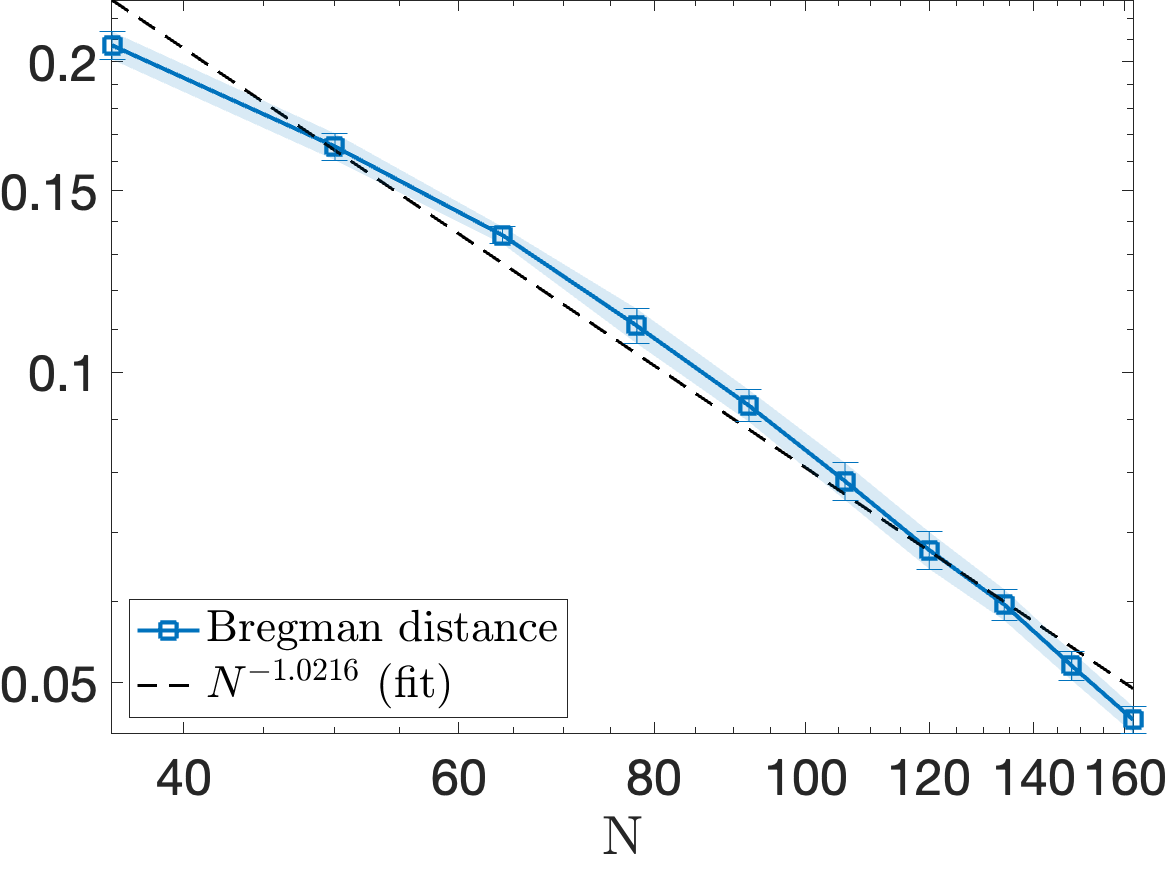}
        & \includegraphics[width=0.25\textwidth]{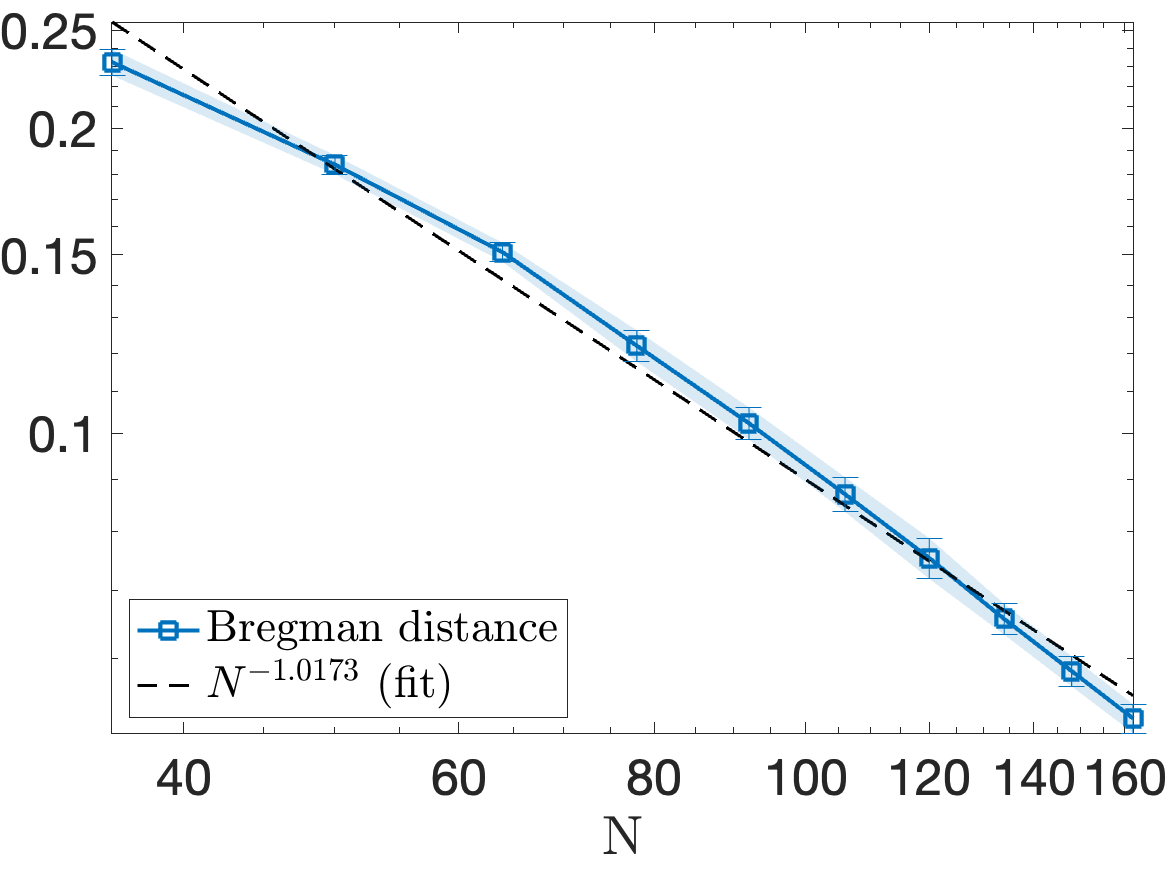} \\
        (a) & (b) & (c) & (d)
    \end{tabular}
    \caption{Approximate decay of the expected value of the Bregman distance, with shearlets-based regularization, for $p=1.1$ ((a) and (c)) and $p=1.01$ ((b) and (d)). (a)\&(b): Fixed noise regime. (c)\&(d): Decreasing noise regime.}
    \label{fig:Plant_shearlets_pTendsTo1_appSC_NoWeights}
\end{figure}
\begin{table}[t]
\centering
\caption{Approximate decay of the expected value of the Bregman distance, with shearlets-based regularization, for $p=1.1$ and $p=1.01$.}
\begin{tabular}{c|c|cc}
scenario & theoretical & $p = 1.1$ & $p=1.01$   \\
\hline
decreasing noise & $-1$ & $-1.0216$ & $-1.0173$  \\
fixed noise & $-1/3$ & $-0.34152$ & $-0.31914$ 
\end{tabular}
\label{tab:decaysShearlets_pTendsTo1}
\end{table}

To build intuition for the $p=1$ case, we start by gathering some evidence of the numerical behaviour of the expected value of the Bregman distance when we let $p$ get arbitrarely close to $1$. As we remarked already, even though the Bregman distance depends on $p$, the concentration estimates do not. Therefore we can expect to observe the same decay of the expected value of the Bregman distance, regardless of the value of $p \in (1,2)$. 

To this end we consider $p=1.1$ and $p=1.01$ as representative values of $p \rightarrow 1$ and repeat the tests carried out in subsection~\ref{ssec:ShearletsNumerics}. We maintain the same numerical set up as in subsection~\ref{ssec:ShearletsNumerics}, except for the values of $p$ (and $c_\alpha$, which is heuristically determined in each experiment). As the results displayed in table~\ref{tab:decaysShearlets_pTendsTo1} and figure~\ref{fig:Plant_shearlets_pTendsTo1_appSC_NoWeights} show, we can draw the same conclusions reached for the $p=3/2$ and $p=4/3$ cases: $\E [D_{\vec{\wtR}}(\f^\delta_{\alpha,N}, \f^\dag)]$ decays as $N^{-1/3}$ in the fixed noise scenario, and as $N^{-1}$ in the decreasing noise one. Furthermore, the behaviour of $D_{\vec{\wtR}}(\f^\delta_{\alpha,N}, \f^\dag)$ as random variable is perfectly in line with what observed so far. 
The only significant difference with respect to the tests run in subsection~\ref{ssec:ShearletsNumerics} concerns verifying the source condition. From a theoretical perspective, we can use the same strategy outlined in subsection~\ref{ssec:SC_shearlets}. 
However, the numerical assessment of the approximate source conditions, \eg{}, by mean of algorithm~\ref{algo:approxSC} is impractical. Indeed, if $p$ is close to $1$, its conjugate exponent $q$ becomes extremely large, and a decay of the orded $N^{-q}$ is difficult to observe due to finite-precision arithmetic and numerical errors.
In practice, already for $p=1.1$ (and $q=11$) it is impossible to verify the source condition. Nonetheless, the numerical evidence gathered with $p=3/2$ and $p=4/3$, with both wavelet-based and shearlet-based regularization, is that the observed decay is not hindered by the (verification of the) source condition. Therefore, we can rely on this to conclude that we can trust the results in table~\ref{tab:decaysShearlets_pTendsTo1} and figure~\ref{fig:Plant_shearlets_pTendsTo1_appSC_NoWeights} even though it has not been possible to verify the source condition.

\begin{remark} \label{rem:constants}
Notice that even though the main results (theorems~\ref{thm:general_rate_nonneg} and~\ref{thm:general_rate_shearlet}) do not depend on $p$, some of the constants
that have been omitted there or in the preliminary results depend on $p$, and their behaviour could in principle be critical as $p$ approaches $1$ (\ie{}, $q$ tends to $\infty$). In particular, the constants omitted from lemma~\ref{lem:xuroach} and from propositions~\ref{prop:shearlet_rate1} and~\ref{prop:shearlet_rate2} (and equivalent results in the wavelet case) show an undesired blow-up as $p\rightarrow 1$. As suggested by the stability of the numerical results, it should be possible to circumvent such theoretical shortcomings, for example introducing alternative assumptions with respect to (S1) or (B1).
\end{remark}

Given that the machinery developed in the previous sections upholds even when we let $p$ tend to $1$, one can take a leap of faith and repeat the tests once more, this time setting $p=1$ and maintaining the same numerical set up of subsection~\ref{ssec:ShearletsNumerics}. Also in this case we can use VMILA (see, in particular, equations~\eqref{eq:VMILAiter} and \eqref{eq:VMILADualProbl_l1constrained} in appendix~\ref{ssec:VMILAnonneg}). Now, the theory from section~\ref{sec:shearlets} does no longer apply: nonetheless, the results in table~\ref{tab:decaysShearlets_p1} and figure~\ref{fig:Plant_shearlets_p1_NoWeights} suggest that, also in this case, we can expect the same decay for the expected value of the Bregman distance. 
Notice that, compared to the $1<p<2$ case, producing the plots in figure~\ref{fig:Plant_shearlets_p1_NoWeights} requires, first and foremost, to choose a representative of the subdifferential $\partial R$ to define the Bregman distance.  Indeed, $R(f) = \norm{f}_X$, with $X$ either $B_1^s$ or $\msSC_{1,m}$, is not differentiable and therefore the subdifferential $\partial R$ is not single-valued. As a consequence, the Bregman distance is not uniquely defined and depending on the choice of an element in the subdifferential, the definition of the Bregman distance changes. 
Starting from \eqref{eq:BregDist}, when $p=1$ the definition of the subdifferential reads as:
\begin{equation*}
\sdiff_f \in \partial R(f) 
\qquad \Leftrightarrow \qquad \sdiff_f = M^* \zeta 
\qquad \text{where} \quad 
\zeta_{l} = \begin{cases}
1 &\text{if } \; [M f]_{l} >0 \\
-1 &\text{if } \; [M f]_{l} <0 \\
\end{cases}
\end{equation*}
and 
\begin{equation}
\zeta_l \in [-1,1] 
\qquad \text{if} \quad [M f]_{l} =0, 
\label{eq:BregDistp1}    
\end{equation}
where $M$ denotes either the wavelet or the shearlet transform. Any different choice of $\zeta_l$ in \eqref{eq:BregDistp1} leads to a different element $r_f \in \partial R(f)$, hence to a different Bregman distance.
In the case of table~\ref{tab:decaysShearlets_p1} and figure~\ref{fig:Plant_shearlets_p1_NoWeights}, we consider the same Bregman distance $D_R(f,\tilde{f})$ as in~\cite{Burger07}, where the element $r_f$ of the subgradient of $R$ in $f$ is selected also according to $\tilde{f}$, and in particular the value of $\zeta_l$ in \eqref{eq:BregDistp1} is chosen as
\begin{equation}
\zeta_l = \sign([M \tilde{f}]_{l}) 
\qquad \text{if} \quad 
[M f]_{l} =0. 
\label{eq:BregDistp1_Burger}      
\end{equation}

\begin{figure}
    \centering
    \begin{tabular}{@{}c@{\quad}c@{}}
        fixed noise 
        & decreasing noise  \\ 
        \includegraphics[width=0.25\textwidth]{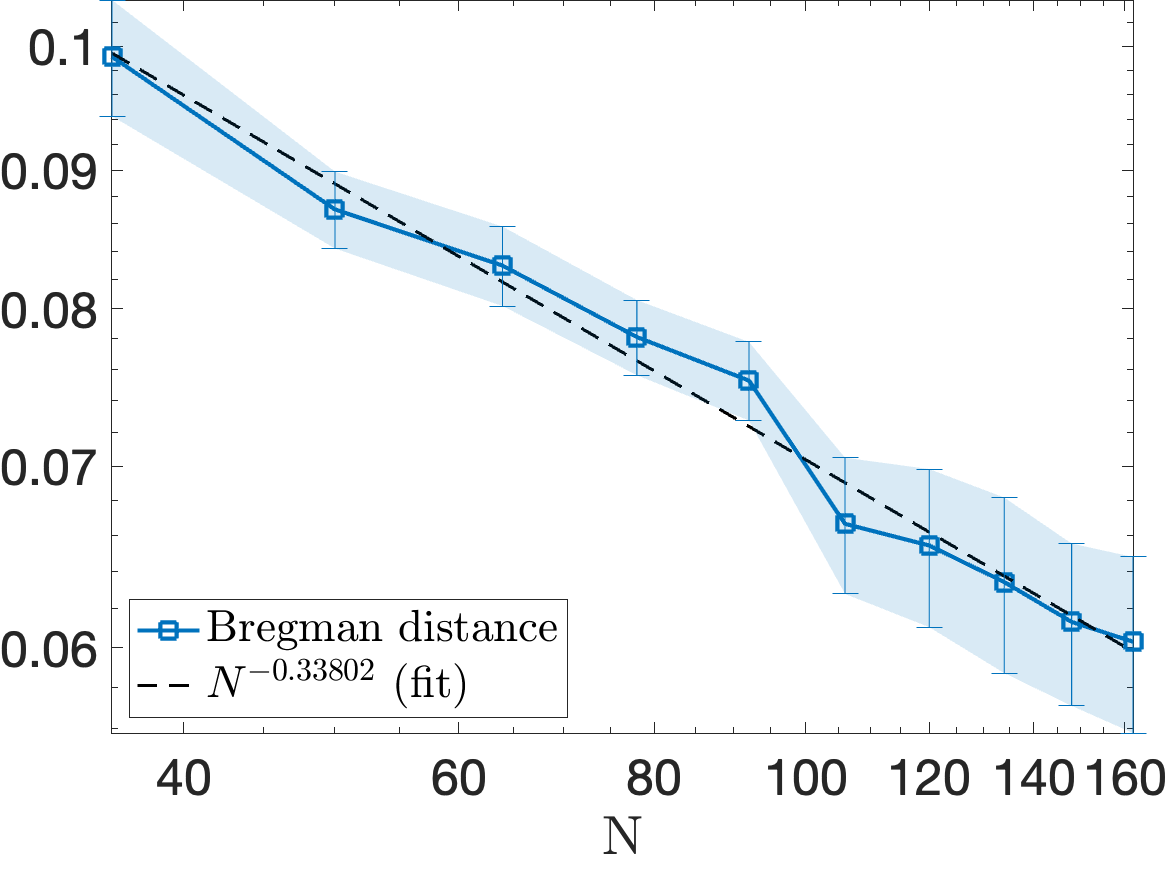} 
        & \includegraphics[width=0.25\textwidth]{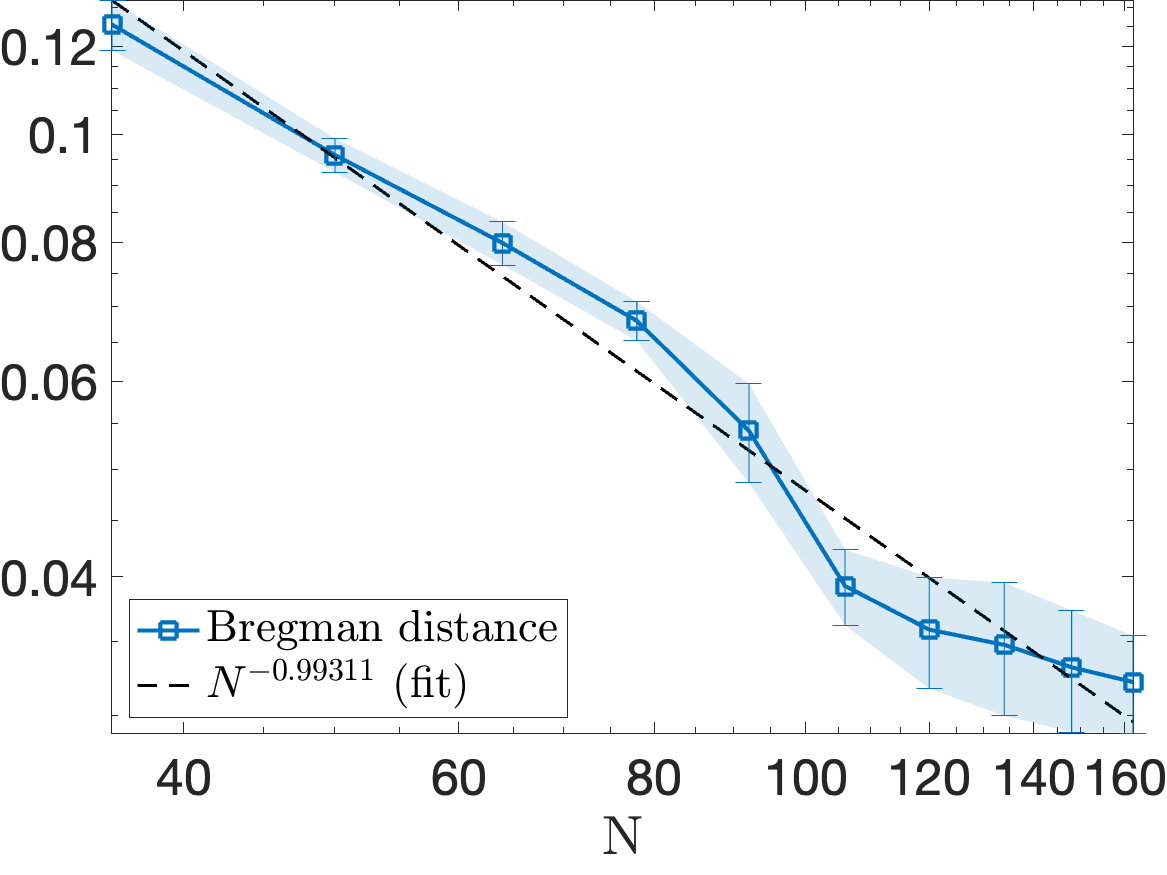} \\
        (a) & (b)
    \end{tabular}
    \caption{Approximate decay of the expected value of the Bregman distance, with $p=1$ and shearlets-based regularization. (a) Fixed noise regime. (b) Decreasing noise regime.}
    \label{fig:Plant_shearlets_p1_NoWeights}
\end{figure}
\begin{table}[t]
\centering
\caption{Approximate decay of the expected value of the Bregman distance, with $p=1$ and shearlets-based regularization.}
\begin{tabular}{c|c|c}
scenario & theoretical & $p = 1$   \\
\hline
decreasing noise & $-1$ & $-0.99311$   \\
fixed noise & $-1/3$ & $-0.33802$  
\end{tabular}
\label{tab:decaysShearlets_p1}
\end{table}

With respect to different alternatives, this choice of $r_f$ (and of $r_{\tilde{f}}$) allows $D_R(f,\tilde{f})$ to be a slightly more informative indicator of the difference between $f$ and $\tilde{f}$. We nevertheless point out that in the case $p=1$, whatever the choice of $r_f$, $D_R$ is actually not a metric: indeed, $D_R(f,\tilde{f})=0$ whenever the corresponding components of $Mf$ and $M\tilde{f}$ have the same signs, regardless of their magnitude. This implies that, depending on the application, $D_R$ can be a non-informative tool to quantify the difference between elements of $X$ (see~\cite[figure 1]{Molinari20}). In figure~\ref{fig:Plant_shearlets_p1_NoWeights}, though, it seems that our choice of $D_R(\fdan,f^\dag)$ is able to capture the desired convergence properties.

\subsection{Theoretical considerations}
\label{ssec:BoundsError}
The non-differentiability of $R$, and the resulting non-unique definition of the Bregman distance, is not the only issue we have to deal with when considering the case $p=1$. We now point out the main critical aspects preventing the extension of our results in theorems~\ref{thm:general_rate_nonneg} and~\ref{thm:general_rate_shearlet}, and discuss some connections with alternative approaches which are already present in the literature.

As already pointed out in remark~\ref{rem:constants}, one first critical issue is the uncontrolled behaviour of the constants omitted in our main results when $p \rightarrow 1$. In addition to this, the terms $\mathbb{E}[\mathscr{R}(\beta,\bu,f^\dag)]$ and $\E [R^\star (\Abu^* \epsilon_N)]$, which are extensively used in the proof of theorems~\ref{thm:general_rate_nonneg} and~\ref{thm:general_rate_shearlet}, must be carefully handled in this context. Indeed, the carachterization of the dual space $X^*$ is more problematic: despite $B_1^s$ and $\msSC_{1,m}$ being still separable Banach spaces, it is impossible to identify $X^*$ with $\ell^\infty$. Moreover, since the function $R$ is now $1$-homogeneous, its conjugate $R^\star$ corresponds to the indicator function of the unit ball in the dual space $X^*$, which may entail that $\E [R^\star (\Abu^* \epsilon_N)]$ diverges. Notice finally that, in the case $p=1$, the subdifferential is not invertible even in the case of the wavelet transform. For this reason, it is impossible to generate a phantom $\f^\dag$ satisfying (B1) or (S1) along the line of subsection \ref{subsec:SourceCond}. In line of principle, one could apply again~\eqref{eq:BregDistp1_Burger} and determine an optimal element $\w$ associated with a desired phantom $\f_0$, but it would be impossible to generate a new phantom $\f^\dag$ from such $\w$. On the other hand, we are not aware of any alternative or approximate source conditions, which could be in spirit equivalent to the one involving~\eqref{eq:Rbeautiful}. Therefore, we cannot follow a strategy resembling algorithm~\ref{algo:approxSC}.

The derivation of convergence rate for Tikhonov regularization in the presence of $\ell^1$-norm has been the object of extended studies in the last two decades, and we can relate and take advantage of a wide varieties of approaches and techniques in the literature. We here discuss the connections with a subset of the most relevant related papers: in all cases, the authors consider the Bregman distance together with other error metrics, such as the $\ell^2$ and $\ell^1$ errors, and of course the variational formulation does not include the sampling operator~\eqref{eq:SamplingOperator}.
Starting from the seminal work from Burger and Osher \cite{Burger04}, the Bregman distance has been considered the natural metric in which to derive convergence rates for convex regularization terms. In this context, under the assumption of a source condition, it is natural to derive convergence rates, as the noise level vanishes, of the Bregman distance associated with a specific choice of the subdifferential element, namely, the one involved in the source conditions. If we were to directly translate this ideas in our framework, nevertheless, due to the presence of the sampled operator, the choice of the source condition element (and of the Bregman distance itself) would vary at each sample. Notice that this was not a problem in our previous results, where we could formulate the source conditions (B1) or (S1) for the non-sampled operator.
Along this line, an interesting direction of investigation is represented by the fundamental work of Grasmair, Scherzer, and Haltmeier\cite{Grasmair11bis}, which showed that combining the source condition with some additional assumptions drawn from compressed sensing, such as the well-known restricted isometry property, allows to recover a result on the convergence rate of the error norm from the one on the Bregman distance. 
Among the results stemmed from this approach (see also \cite{Grasmair11}), the most relevant one to our framework is~\cite{Haltmeier12}, which considers redundant, non-tight frames, that is, the class to which shearlet systems belong to. In~\cite{Haltmeier12}, the author derives an error estimate on the $\ell^2$ error by using upper bounds on the discrepancy error and the Bregman distance. Deriving this result requires several assumptions (see \cite[Assumption III.6]{Haltmeier12}), some of which are rather technical and quite difficult to interpret in light of a practical application such as tomography. In particular, the author requires the injectivity of the forward operator restricted to the linear span of certain elements of the dual frame (whose definition depends also on the choice of the source condition element). This assumption, in addition to suffering from the same limitation of the choice of the source condition element, is also very difficult to interpret when we pick a specific forward operator, such as the the X-ray transform.
Therefore, even though the machinery in~\cite{Haltmeier12} could be very promising, it is arguably inapplicable to our randomized, sampled setting.
Another promising direction is to consider the argument in \cite{Lorenz08}, where a convergence rate with respect to the $\ell^1$-norm is deduced from the one in Bregman distance, introducing the finite basis injectivity property.
Finally, the most relevant and promising approach for our purposes is the one derived in
the paper from Burger, Helin and Kekkonen~\cite{Burger18}. This work, which has already been a source of inspiration to treat the case $1<p\leq 2$ in \cite{Bubba21}, considers also 
the $p=1$ case, and shows a possible strategy to deal with the presence of large noise. Adapting the techniques in~\cite{Burger18} to the case of random sampling seems therefore the most encouraging perspective in this direction. Clearly, this requires to re-build from scratch the theoretical framework and we leave this line of investigation to future work.

\subsection{$\Gamma$-converging to $p=1$}
\label{ssec:Gconv}

To conclude the discussion about $\ell^1$ regularization, we here propose a slightly different strategy.
Starting from the numerical evidence described in subsection~\ref{ssec:TMconv}, we want to provide a theoretical justification of why the solution of the $\ell^p$ regularized problem stably depends on $p$, even in the case $p \rightarrow 1$. 
In sections \ref{sec:NonnegConstr} and \ref{sec:shearlets} we characterized the regularized solutions in the presence of a regularization term $R$ equal to the $\ell^p$-norm of the wavelet or shearlet transform coefficients as the minimizer of \eqref{eq:regularized_nonneg}. To do so, the choice of the Banach space $X$ on which the minimization is carried out is crucial, and corresponds to Besov spaces (\ie{}, $X = B_p^s(\Omega)$) for wavelet regularization and to coorbit spaces (\ie{}, $X = B_p^s(\Omega)$) for shearlet regularization. Since this characterization is valid also for $p=1$, we can study the behaviour of the regularized solutions as $p\rightarrow 1$ taking advantage of several results regarding the convergence of minimizers of functionals. To this end, a preliminary difficulty is represented by the fact that each functional is minimized over a different space $X$ (which in fact depends on $p$). This can be fruitfully handled by means of standard technique based on $\Gamma$-convergence~\cite{Braides02,Dalmaso12}.

In order to avoid ambiguity in the notation, we start by  explicitly remarking the dependence on $p$ of the space we consider. Hence, we let $X_p$ be either $B_p^s(\Omega)$ or $\msSC_{p,m}$ and denote by $M$ the transform operator (either shearlet or wavelet). For simplicity, in this section we consider uniform weights, which amounts to say $s = d \left(\frac{1}{p}-\frac{1}{2}\right)$ or $m \equiv 1$.
Notice that, via $M$, we can identify each space $X_p$ with the sequence space $\ell^p$, and use the norm $\norm{f}_{X_p} = \norm{M f}_{\ell^p}$ on $X_p$. In particular, for $1\leq p \leq 2$, it holds that $X_1 \subset X_p \subset X_2$, and $\norm{f}_{X_{p'}} \leq \norm{f}_{X_{p}}$ for $p'\geq p$ holds for any $f \in X_p$.

For each different $p$, we introduce a different regularization term denoted as follows:
\begin{equation}
\label{eq:Regu_depend_p}
R_p(f)  = \frac{1}{p} \norm{M f}_{\ell_p}^p 
    = \whR_p(M f),
\end{equation}
where $\whR_p = \frac{1}{p} \norm{\cdot}_{\ell_p}^p$.
Finally, we define the functional
\begin{equation}
\label{eq:Funct_depend_p}
\Jdanp(f)  =  \frac{1}{2} 
    \norm{\Abu f - \gdan}^2_{V_N} + \alpha R_p(f) 
    + \iota_{+}(f),
\end{equation}
whose minimizer in $X_p$ is denoted by $\fdanp$. Since the only term depending on $p$ in $\Jdanp$ is the regularization term $R_p$, we start by proving its $\Gamma$-convergence to $R_1$, or, equivalently, the $\Gamma$-convergence of $\whR_p$ to $\whR_1$.
Notice carefully that each $\whR_p$ is defined - and minimized - on different spaces $\ell^p$. Nevertheless, since we are interested in $1 \leq p < 2$, all such spaces are embedded in $\ell^2$: therefore, we can extend all $\whR_p$ to functionals defined on the largest space $\ell^2$ as follows:
\begin{equation}
    \label{eq:extended_whR}
\bR_p(c) = \begin{cases} 
    \frac{1}{p} \norm{c}_{\ell^p}^p &\quad  c \in \ell^p \\[0.25em]
    \infty &\quad c \in \ell^2 \setminus \ell^p.
\end{cases}
\end{equation}
We are interested in proving that the functionals $\bR_p(c)$ converge to $\bR_1(c)$ in the sense of the $\Gamma$-convergence with respect to the $\ell^2$ topology. As it will be shown later, this entails that the regularized functionals $\Jdanp$ $\Gamma$-converge to $\Jdano$ (suitably extended to functionals in $X_2$) with respect to the $X_2$ topology, and ultimately the convergence of their minimizers.
To do so, we first need to prove the following ancillary result.
\begin{lemma}
\label{lem:GammaConvLem2}
Let $p_n \rightarrow 1$. Then, for all $c \in \ell^1$ we have  $\bR_{p_n}(c) \rightarrow \bR_1(c)$, namely,
\[
\frac{1}{p_n} \norm{c}_{\ell^{p_n}}^{p_n} \rightarrow \norm{c}_{\ell^1}.
\]
\end{lemma}
\begin{proof}
Since $c \in \ell^1$ we have $\sum_i |c_i| < \infty$. Consider $I = \{ i \, : \, |c_i|>1\}$: it is clear that card$(I) < \infty$. 

Observe that 
\begin{equation}
\frac{1}{p_n} \norm{c}_{\ell^{p_n}}^{p_n} = 
    \sum_{i \in \N} \frac{1}{p_n} |c_i|^{p_n} = 
    \sum_{i \in I} \frac{1}{p_n} |c_i|^{p_n} + 
         \sum_{i \in \N \setminus I} \frac{1}{p_n} |c_i|^{p_n}.
\end{equation}
Since $\frac{1}{p_n} |c_i|^{p_n} \rightarrow |c_i|$ for all $i$, this yields the convergence of the finite sum
\[
\sum_{i \in I} \frac{1}{p_n} |c_i|^{p_n} \rightarrow 
    \sum_{i \in I} |c_i|.
\]
Finally, by using dominated convergence, 
since for $i \in \N \setminus I$ it holds $\frac{1}{p_n} |c_i|^{p_n} \leq |c_i|$, and $\sum_{i \in \N \setminus I} |c_i| \leq \| c \|_{\ell^1} < \infty$, we recover 
\[
\sum_{i \in \N \setminus I} \frac{1}{p_n} |c_i|^{p_n} \rightarrow 
    \sum_{i \in \N \setminus I} |c_i|.
\]
\end{proof}

The previous lemma allows us to prove the $\Gamma$-convergence. 
\begin{theorem} \label{thm:Gamma_conv_p}
Let $\{p_n\}_{n \in \N} \subset (1,2)$ be such that $p_n \searrow 1$ and let $\bR_n = \bR_{p_n}$. Then, the functionals $\bR_n \rightarrow \bR = \bR_1$ in the sense of the $\Gamma$-convergence, with respect to the $\ell^2$ topology: namely, 
\begin{itemize}
    \item[(i)] For all $\{c_n\}_{n \in \N} \subset \ell_2$ and $c \in \ell^2$ such that $c_n \rightarrow c$ in $\ell_2$ we have
    \[
     \bR(c) \leq \liminf_{n \rightarrow \infty} \bR_n(c_n).
    \]
    \item[(ii)] For all $c \in \ell_2$, there exists  $\{c_n\}_{n \in \N} \subset \ell_2$ such that $c_n \rightarrow c$ in $\ell_2$ and 
    \[
    \limsup_{n \rightarrow \infty} \bR_n(c_n) \leq \bR(c).
    \]
\end{itemize}
\end{theorem}
\begin{proof}
To prove (i), we first exclude the case $\liminf_{n \rightarrow \infty} \bR_n(c_n) = \infty$, which would trivially satisfy the inequality. If instead $\liminf_{n \rightarrow \infty} \bR_n(c_n) < \infty$, then there exists a subsequence (still denoted by $\{c_n\}$) such that $\lim_{n \rightarrow \infty} \bR_n(c_n) = \liminf_{n \rightarrow \infty} \bR_n(c_n) < \infty$. As a consequence, such sequence is bounded: let $b>0$ be such that $\bR_n(c_n) < b$. This also implies, by the definition of $\bR_n$, that $c_n \in \ell^{p_n}$. Fix now an index $\overline{n}>0$ and the corresponding $\overline{p} = p_{\overline{n}}$: since $p_n$ is monotonically decreasing, $1< p_n < \overline{p}$ for all $n>\overline{n}$, hence
\[
\| c_n \|_{\ell^{\overline{p}}} \leq \| c_n \|_{\ell^{p_n}} \leq \bR_n(c_n) \leq b.
\]
Therefore, again up to a subsequence, by the weak compactness of $\ell^{\overline{p}}$, there exists $\widetilde{c} \in \ell^{\overline{p}}$ such that $c_n \rightharpoonup \widetilde{c}$ in $\ell^{\overline{p}}$. Since $\overline{p} < 2$, the weak $\ell^{\overline{p}}$ convergence implies the weak $\ell^2$ one, and by the uniqueness of the weak limit we conclude that $\widetilde{c} = c$. As a results, we have that $c \in \ell^{\overline{p}}$ and, by the weak lower semicontinuity of the $\ell^{\overline{p}}$-norm,
\[
\bR_{\overline{n}}(c) \leq \liminf_{n \rightarrow \infty} \bR_{\overline{n}}(c_n) \leq \liminf_{n \rightarrow \infty} \bR_{n}(c_n).
\]
Since this argument is valid for any $\overline{n}$, we conclude that $c \in \ell^1$ and, by applying lemma \ref{lem:GammaConvLem2}, 
\[
\bR(c) = \lim_{\overline{n} \rightarrow \infty} \bR_{\overline{n}}(c) \leq \liminf_{n \rightarrow \infty} \bR_{n}(c_n).
\]
To prove (ii), instead, we first exclude the trivial case $\bR(c) = \infty$. As a consequence, we can assume that $c \in \ell^1$, and therefore $c \in \ell^{p_n}$ for all $n$. Then, it is possible to consider a constant recovery sequence $c_n = c$ $\forall n$ to conclude that, via lemma \ref{lem:GammaConvLem2},
\[
\lim_{n \rightarrow \infty} \bR_n(c_n) = \lim_{n \rightarrow \infty} \bR_n(c) = \bR(c).
\]
\end{proof}

The most relevant consequence of the $\Gamma$-convergence of functionals is the convergence of the respective minimizers. We therefore use the previous result to show the convergence of the $\ell^p$-norm regularized solutions to the $\ell^1$-norm one.
\begin{corollary} \label{cor:minim_converg}
Let $\{p_n\}_{n \in \N} \subset (1,2)$ be such that $p_n \searrow 1$. Then, denoted by $\fdanp$ the solution of \eqref{eq:Funct_depend_p}, we have that
\begin{equation}
\norm{ \fdanp - \fdano}_{X_2} \rightarrow 0.
    \label{eq:Gamma_minim_converg}
\end{equation}
\end{corollary}
\begin{proof}
We first notice that the functionals $\Jdanp$ (suitably extended to $\infty$ outside $X_p$) $\Gamma$-converge to $\Jdano$ (suitably extended to $\infty$ outside $X_1$) with respect to the $X_2$ topology. This is a direct consequence of the $\Gamma$-convergence of the functionals $\bR_p$ proved in theorem \ref{thm:Gamma_conv_p}, and of the fact that the remaining terms in $\Jdanp$ are independent on $p$ and continuous with respect to $f$ (see \cite[remark 1.7]{Braides02}). Then, the convergence of the minimizer with respect to the $X_2$ norm is a direct consequence of the $\Gamma$-convergence of the functionals $\Jdanp$ (see \cite[theorem 7.4]{Dalmaso12}).
\end{proof}

Corollary~\ref{cor:minim_converg} colud be used to deduce theoretical convergence rates of $\fdano$ to $f^\dag$, using as a metric the (expected) error in $\ell^2$-norm. Nevertheless, this would require to quantify \eqref{eq:Gamma_minim_converg}, providing a convergence rate and an explicit dependence of the constants with respect to $N$, as well as an explicit expression of the omitted constants in theorems~\ref{thm:general_rate_nonneg} and~\ref{thm:general_rate_shearlet} as a function of $p$. This is beyond the scope of this work, and is left for future investigation. 

On the other hand, the result in corollary~\ref{cor:minim_converg} can be considered from a purely practical perspective. Indeed, according to \eqref{eq:Gamma_minim_converg}, in any numerical setup (with a prescribed maximum sample size $N_{\max}$), it is possible to find a value of $p$ such that the regularized solutions $\fdanp$ are equal to $\fdano$ up to a desired tolerance, also inheriting the expected decay of the Bregman distance associated with the $p$ case.

\section{Discussion and outlook}
\label{sec:conclusions}

In this paper, we extended the theoretical and numerical results of~\cite{Bubba21} to the framework of shearlet-based regularization with $p \in (1,2)$, possibly constrained to the non-negative orthant. In particular, we proved that also in this case we obtain the same convergence rates on the Bregman distance, that is, $\E [D_{\wtR}(f^\delta_{\alpha,N}, f^\dag)]$ decays as $N^{-1/3}$ in the fixed noise scenario, and as $N^{-1}$ in the decreasing noise one. Also, we verified numerically the expected decays in the case of simulated tomographic data with randomly sampled imaging angles. For the $p=1$ case, we carried out a numerical study, gathering empirical evidence that we can expect the same theoretical bounds on the Bregman distance as in the $p \in (1,2)$ case. Finally, we used the tools of $\Gamma$-convergence to gain further insight in the $p=1$ case by approaching it as the limit case $(1,2) \ni p \rightarrow 1$.

The rigorous theoretical analysis of the $p=1$ case is left as future work. Clearly, proving convergence rates for sparsity enforcing regularization with respect to non-tight frames entails some theoretical burden, since we loose nice properties such as the orthonormality of the sparsity transform and the smoothness of the penalty term. Therefore, this raises the question whether this is a worthy problem to investigate.  

\begin{figure}
    \centering
    \begin{tabular}{@{}c@{\qquad}c@{}}
        Plant
        & Lotus root \\ 
        \includegraphics[width=0.4\textwidth]{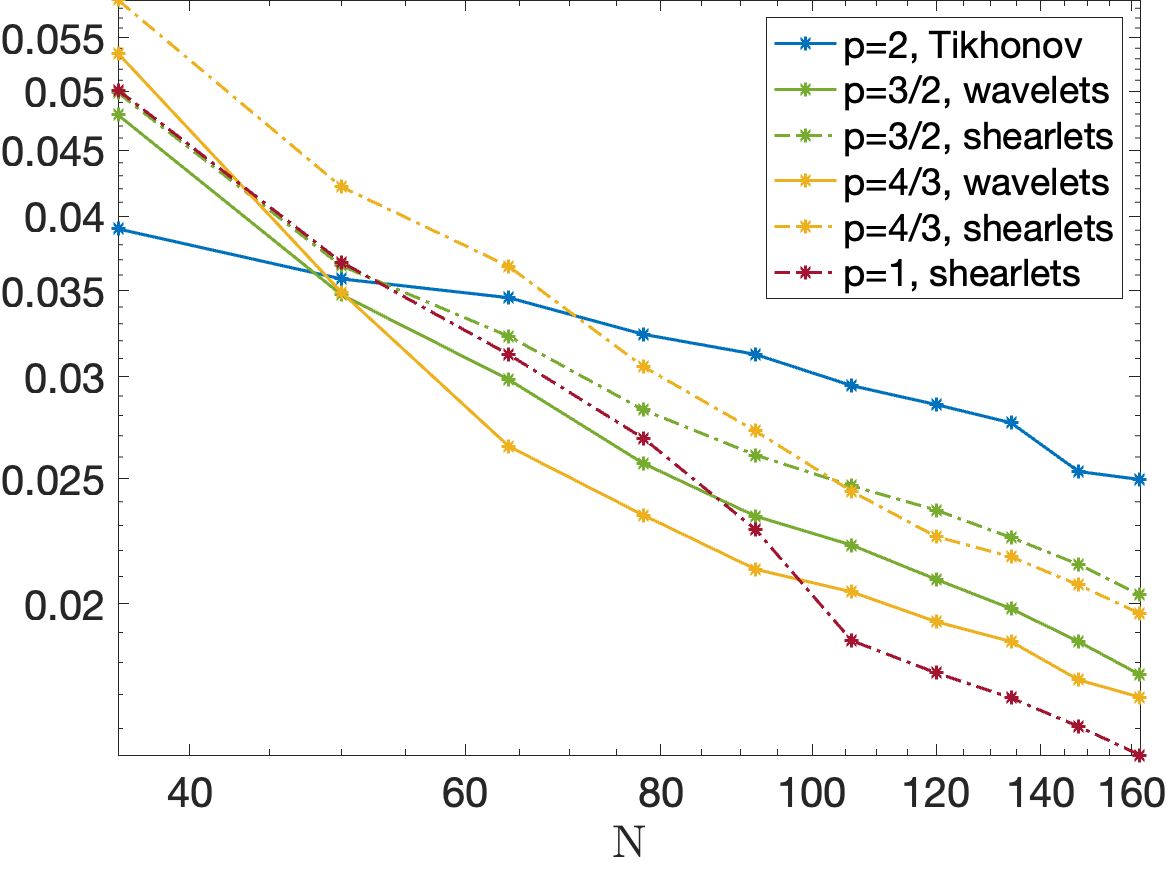} 
        & \includegraphics[width=0.4\textwidth]{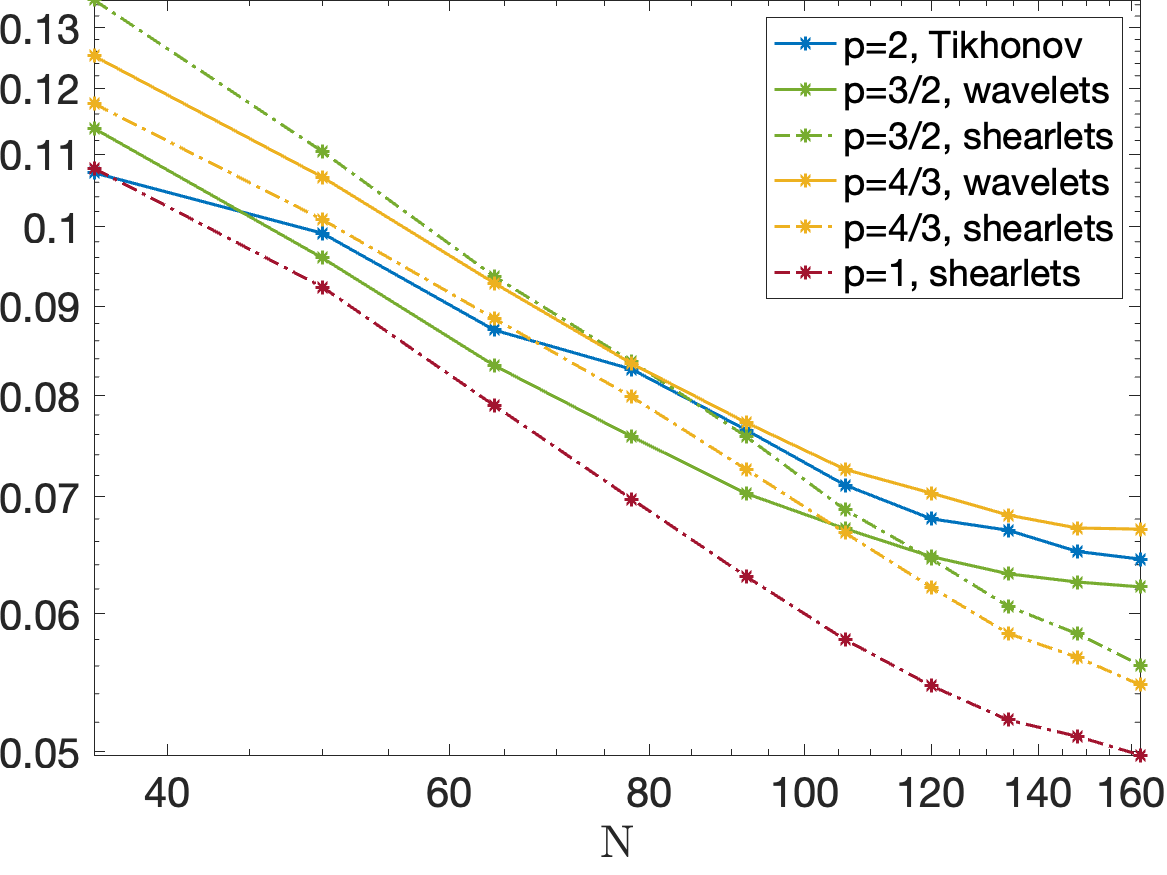} \\
        (a) & (b)
    \end{tabular}
    \caption{Approximate decay of the expected value of the squared relative error, in the fixed noise regime and for different regularization strategies (\ie{}, different values of $p$). (a) Simulated data of a plant phantom. (b) Measured data of a lotus root.}
    \label{fig:Fixed_Comparison}
\end{figure}

To this end, in figure~\ref{fig:Fixed_Comparison} we want to compare the decay of the expected value of the (squared) relative error, that is, $\E \big[\|\f^\dag - \f^\delta_{\alpha,N}\|^2_2/\|\f^\dag\|^2_2 \big]$, for  different regularization strategies. The goal is to show that for $p \rightarrow 1$, and for multiresolution systems which provably yield sparser approximations, the expected value of the relative error gets appreciably smaller, especially when the number of angles increases. To showcase that the plots do not depend on the particular choice of the data, in addition to the simulated plant data used throughout this work, we furthermore make use of real data from a scan of a lotus root~\cite{Bubba16}. In this case, the measured data are corrupted with additional Gaussian noise, using the same strategy adopted for corrupting the simulated data and a noise level $c_\delta$ three times higher than the one described in subsection~\ref{ssec:NonnegConstr_numerics} for simulated data. As usual, the sample averages are computed using 30 random realizations.
For the error computations we use as ground truth a reconstruction of the target using dense angular sampling and $\Npxl = 128^2$; the operator $\RadonD_{\thetab}$ and its adjoint are computed with ASTRA~\cite{VanAarle16,VanAarle15}, which allows to draw random imaging angles within the fan-beam geometry setting. All the reconstructions are computed using either VMILA, for $p=3/2,4/3$ (see equations~\eqref{eq:VMILAiter} and~\eqref{eq:VMILADualObjConstrained} in Appendix~\ref{app:VMILA}) and $p=1$ (see equations~\eqref{eq:VMILAiter} and~\eqref{eq:VMILADualProbl_l1constrained} in Appendix~\ref{app:VMILA}), or SGP, for $p=2$ (see equations~\eqref{eq:ObjFunctSGP}-\eqref{eq:SGPiter} in appendix~\ref{ssec:VMILAnonneg}). 
In particular, we focus on the fixed noise regime and different values of $p \in [1,2]$, considering both wavelets and shearlets for the sparsifying transform: the error decay for the simulated case are reported in figure~\ref{fig:Fixed_Comparison}(a), while for the measured case in figure~\ref{fig:Fixed_Comparison}(b).
The case $p=2$ corresponds to classical Tikhonov regularization, namely, in this case $\Mop = \mathbbm{1}_{\Npxl}$. 
Notice that, in figure~\ref{fig:Fixed_Comparison}(a), for $p=3/2, \, 4/3, \, 1$, these are the relative errors corresponding to the Bregman distances reported in figures~\ref{fig:Plant_wavelets_appSC}, \ref{fig:Plant_shearlets_appSC_NoWeights} and  \ref{fig:Plant_shearlets_p1_NoWeights}, respectively. 
In both cases we can see that, as $N$ increases, the strategy yielding the lowest reconstruction error is shearlet-based regularization with $p=1$, while the highest is essentially given by Tikhonov regularization with $p=2$. The decay rates for $p=3/2$ and $p=4/3$, and with both wavelets and shearlets, fall within the region delimited by the $p=1$ and $p=2$ curves. 
Similar plots can be obtained for the decreasing noise regime, even though the difference between regularization strategies is less noticeable since the noise level is reducing.  These plots are very informative also in view of practical applications, and provide some quantitative insight to relations between sparsity and sufficient (under)sampling. Indeed, depending on the chosen regularization strategy (\ie{}, the value of $p$) one needs a different amount of imaging angles to achieve a fixed value of the error. 
In particular, the plots suggest that with $p=1$ one needs less data (that is, fewer imaging angles) to achieve a much better reconstruction, according to the error metric.
Ultimately, this supports the much advocated paradigm that sparsity-enforcing penalties provide a viable alternative to the usual quadratic ones for the regularization of undersampled ill-posed problems, and gives motivation to this work.

\section*{Acknowledgments}
The authors would like to thank the Isaac Newton Institute for Mathematical Sciences, Cambridge, for support and hospitality during the programme ``Mathematics of Deep Learning'' where work on this paper was undertaken. This work was supported by EPSRC grant no EP/R014604/1. TAB is supported by the Royal Society through the Newton International Fellowship grant n. NIF\textbackslash R1\textbackslash 201695 and was partially supported by the Academy of Finland through the postdoctoral grant, decision number 330522. LR is supported by the Air Force Office of Scientific Research under award number FA8655-20-1-7027.
The authors would also like to thank Tapio Helin and Martin Burger for introducing them to the fascinating field of statistical inverse learning problems.

\begin{appendix}
\section{Implementing VMILA with $1 \leq p < 2$}
\label{app:VMILA}

To solve the minimization problem 
\begin{equation}
\argmin_{\f \in \R^{\Npxl}} \bigg\{ \frac{1}{2}
    \norm{\RadonD_{\thetab} \f - \gNd}_2^2 + 
    \frac{\alpha}{p} \norm{\Mop \f}_p^p \bigg\}
\label{eq:MinProblDiscr}    
\end{equation}
we use the variable metric inexact line-search algorithm (VMILA)~\cite{Bonettini16}. 
Here, $\Mop$ represents either the wavelet transform operator, that is, $\Mop = \Wop \in \R^{\Npxl \times \Npxl}$ with $\sigma=1$, or the shearlet transform operator, that is, $\Mop = \shD \in \R^{\sigma\Npxl \times \Npxl}$ where $\sigma > 1$ is the number of subbands (\ie{}, $\sigma\Npxl=\Nsh$). Now, denoted by  
$\Gamma_0(\f) = \frac{1}{2} \norm{\RadonD_{\thetab} \f - \gNd}_2^2$
and $\Gamma_1(\f) = \frac{\alpha}{p} \norm{\Mop \f}^p_p$, the $(k+1)$-th iteration of VMILA for the minimization of~\eqref{eq:MinProblDiscr} is given by:
\begin{equation}
\f^{(k+1)} = \f^{(k)} + \mu_k ( \vi^{(k)} - \f^{(k)})
\label{eq:VMILAiter}
\end{equation}
where the steplength $\mu_k$ is determined by means of a backtracking loop until a modified Armijo inequality is satisfied and
\begin{equation}
\vi^{(k)} = \prox^{D_k}_{\lambda_k\Gamma_1} (\z^{(k)}) 
\label{eq:VMILAprox}    
\end{equation}
with $\z^{(k)} = \f^{(k)} - \lambda_k D_k^{-1} \nabla \Gamma_0(\f^{(k)})$. Here, $\lambda_k$ is a steplength and $D_k$ a scaling matrix at iteration $k$. In particular, $\lambda_k$ is chosen in the closed interval $[\lambda_{\min}, \lambda_{\max}] \subset \R_+$ according to an adaptive Barzilai-Borwein rule~\cite{Barzilai88,Frassoldati08}. The scaling matrix $D_k$ is chosen in the compact set $\mathcal{D}_L$, where $\mathcal{D}_L$ is the set of the symmetric positive definite matrices $D$ such that $\norm{D} \leq L$ and $\norm{D^{-1}} \leq L$, for a given threshold $L > 1$. The  entries of the diagonal scaling matrix $D = \diag(d_i)$ are given by
\[
d_i^{(k)} = \min \bigg\{L, \max \bigg\{ \frac{1}{L}, \frac{\f_i^{(k)}}{\RadonD_{\thetab}^T \big(\RadonD_{\thetab} \f_i^{(k)} \big)} \bigg\} \bigg\}
\]
following the strategy in~\cite{Bonettini16}.

When $\Gamma_1$ is given by the composition of a $p$-norm with a linear operator, as it is in our case, it is not possible to compute the proximity operator in~\eqref{eq:VMILAprox} in a closed-form. However, it is possible to compute an approximation $\widetilde{\vi}^{(k)}$ of $\vi^{(k)}$: we follow the strategy of the $\eta$-approximation introduced in~\cite{Bonettini16}. There, the authors show that 
\begin{equation}
\vi^{(k)} = \prox^{D_k}_{\lambda_k\Gamma_1} (\z^{(k)}) 
    = \max_{\nu \in \R^{\sigma \Npxl}} H(\nu, \f^{(k)})
\label{eq:VMILAprox2}    
\end{equation}
where 
\begin{equation}
H(\nu, \f^{(k)}) = 
    - \frac{1}{2\lambda_k} \norm{\lambda_k D_k^{-1} \Mop^T \nu - \z^{(k)}}^2_{D_k} - \widehat{\Gamma}_1^\star(\nu) - \Gamma_1(\f^{(k)})
    - \frac{\lambda_k}{2} \norm{\nabla \Gamma_0(\f^{(k)})}^2_{D_k^{-1}} + \frac{1}{2\lambda_k}
    \norm{\z^{(k)}}^2_{D_k}
\label{eq:VMILADualObj}    
\end{equation}
with $\Gamma_1(\f) = \widehat{\Gamma}_1(\Mop \f)$ \ie{}, $\widehat{\Gamma}_1 = \frac{\alpha}{p} \norm{\cdot}_p^p$. Here, $\nu \in \R^{\sigma \Npxl}$ is the dual variable whose primal is $\vi \in \R^{\Npxl}$ and $\widehat{\Gamma}_1^\star$ is the conjugate function of $\widehat{\Gamma}_1$. In particular, we have:
\begin{itemize}
    \item $1 < p < 2$: $\widehat{\Gamma}_1^\star$ is the $q$-th power of the dual norm, \ie{}, $\widehat{\Gamma}_1^\star(\nu) = \frac{\alpha^{1-q}}{q} \norm{\nu}_q^q$, with $q$ H\"{o}lder conjugate of $p$ (for example, $q=3$ when $p=\frac{3}{2}$, $q=4$ when $p=\frac{4}{2}$, $q=11$ when $p=1.1$ and $q=101$ when $p=1.01$). In this case, the problem in~\eqref{eq:VMILAprox2}-\eqref{eq:VMILADualObj} reads as:
    \begin{equation}
    \begin{split}
    \vi^{(k)} = \max_{\nu \in \R^{\sigma \Npxl}} \; &H(\nu, \f^{(k)})    
    \qquad \text{with} \\[0.25em] 
    H(\nu, \f^{(k)}) = 
    - \frac{1}{2\lambda_k} \norm{\lambda_k D_k^{-1} \Mop^T \nu - \z^{(k)}}^2_{D_k} &- \frac{\alpha^{1-q}}{q} \norm{\nu}_q^q 
    - \Gamma_1(\f^{(k)})
    - \frac{\lambda_k}{2} \norm{\nabla \Gamma_0(\f^{(k)})}^2_{D_k^{-1}} + \frac{1}{2\lambda_k}
    \norm{\z^{(k)}}^2_{D_k}
    \end{split}    
    \label{eq:VMILADualProbl_lp}    
    \end{equation}
    \item $p=1$: $\widehat{\Gamma}_1^\star$ is the indicator function of the set $B^{\sigma \Npxl}_{\infty}(0,\alpha) = B_{\infty}(0,\alpha) \times \ldots \times B_{\infty}(0,\alpha)$ ($\sigma \Npxl$-times), being $B_{\infty}(0,\alpha) \subset \R$ the ball in the $\infty$-norm centered in $0$ with radius $\alpha$.
    Therefore, the problem in~\eqref{eq:VMILAprox2}-\eqref{eq:VMILADualObj} becomes the following constrained problem:
    \begin{equation}
    \begin{split}
    \vi^{(k)} = \max_{\norm{\nu}_{\infty} \leq \alpha} \; &H(\nu, \f^{(k)})    
    \qquad \text{with} \\[0.25em] 
    H(\nu, \f^{(k)}) = 
    - \frac{1}{2\lambda_k} \norm{\lambda_k D_k^{-1} \Mop^T \nu - \z^{(k)}}^2_{D_k} &- \Gamma_1(\f^{(k)})
    - \frac{\lambda_k}{2} \norm{\nabla \Gamma_0(\f^{(k)})}^2_{D_k^{-1}} + \frac{1}{2\lambda_k}
    \norm{\z^{(k)}}^2_{D_k}
    \end{split}    
    \label{eq:VMILADualProbl_l1}    
    \end{equation}
\end{itemize}
To solve both~\eqref{eq:VMILADualProbl_l1} and~\eqref{eq:VMILADualProbl_lp} we use the scaled gradient projection (SGP) algorithm~\cite{Bonettini09} (see equations~\eqref{eq:ObjFunctSGP}-\eqref{eq:SGPiter} in appendix~\ref{ssec:VMILAnonneg}) stopping the iterations $l$, for a given $\eta \in (0,1]$, when 
\begin{equation}
h(\widetilde{\vi}^{(k,l)}, f^{(k)}) \leq \eta H(\nu^{(l)}, f^{(k)}),
\label{eq:VMILAinnerStop}    
\end{equation}
where $\widetilde{\vi}^{(k,l)} = \widetilde{\vi}^{(k)}$ and $h$ is such that 
\[
\min_{\vi \in \R^{\Npxl}} h(\vi, f^{(k)}) = 
    \max_{\nu \in \R^{\sigma \Npxl}} H(\nu,f^{(k)}).
\]

\subsection{Adding the non-negativity constaint} 
\label{ssec:VMILAnonneg}
VMILA can be easily modified to include also the non-negativity constraint. The starting point is recasting~\eqref{eq:MinProblDiscr} to include the indicator function $\iota_{\R^{\Npxl}_+}$ of the feasible region, \ie{}, the non-negative orthant $\R^{\Npxl}_+$:
\begin{equation}
\argmin_{\f \in \R^{\Npxl}} \bigg\{ \frac{1}{2}
    \norm{\RadonD_{\thetab} \f - \gNd}_2^2 + \frac{\alpha}{p} \norm{\Mop \f}_p^p + 
    \iota_{\R^{\Npxl}_+}(\f)
    \bigg\},
\label{eq:MinProblDiscrConstrained}    
\end{equation}
where we now have $\Gamma_1(\f) = \frac{\alpha}{p} \norm{\Mop \f}_p^p + \iota_{\R^{\Npxl}_+}(\f)$ (and $\Gamma_0(\f) = \frac{1}{2} \norm{\RadonD_{\thetab} \f - \gNd}_2^2$). As a consequence, when $1 < p < 2$ equation~\eqref{eq:VMILADualProbl_lp} reads as:
\begin{equation}
    \begin{split}
    \vi^{(k)} = \max_{\nu \in \R^{\sigma \Npxl}} \; &H(\nu, \f^{(k)})    
    \qquad \text{with} \\[0.25em]
H(\nu, \f^{(k)}) = 
    - \frac{1}{2\lambda_k} \norm{\lambda_k D_k^{-1} \mbB^T \nu - \z^{(k)}}^2_{D_k} -& \widehat{\Gamma}_1^\star(\nu) - \Gamma_1(\f^{(k)})
    - \frac{\lambda_k}{2} \norm{\nabla \Gamma_0(\f^{(k)})}^2_{D_k^{-1}} + \frac{1}{2\lambda_k}
    \norm{\z^{(k)}}^2_{D_k}
\end{split}    
\label{eq:VMILADualObjConstrained}    
\end{equation}
where the matrix $\mbB$ is a block matrix $\mbB = [\mbB_1^T \; \mbB_2^T]^T$. The first block $\mbB_1 \in \R^{\sigma \Npxl \times \Npxl}$ corresponds to the $p$-norm term while the second one, $\mbB_2 \in \R^{\Npxl \times \Npxl}$, accounts for the indicator function:
\begin{equation}
 \mbB \; = \; \begin{bmatrix}
\Mop \\ \mathbbm{1}_{\Npxl}
\end{bmatrix} 
\quad \in \R^{(\sigma \Npxl+\Npxl) \times \Npxl}.  
\label{eq:MatrixBVMILA}
\end{equation}

Accordingly, also the dual variable is split into blocks, namely, $\nu = [\nu_1^T \; \nu_2^T]^T \in \R^{\sigma \Npxl+\Npxl}$, with $\nu_1 \in \R^{\sigma \Npxl}$ and $\nu_2 \in \R^{\Npxl}$. Next, $\widehat{\Gamma}_1^\star = g_1^\star(\nu_1)+g_2^\star(\nu_2)$, where $g_1^\star(\nu_1) = \frac{\alpha^{1-q}}{q} \norm{\nu_1}_q^q$ is the conjugate function of $g_1 = \frac{\alpha}{p} \norm{\cdot}_p^p$, and
$g_2^\star$ is the indicator function of the set $\R^{\Npxl}_{-}$. 

When $p=1$ equation~\eqref{eq:VMILADualProbl_l1} reads as
\begin{equation}
    \begin{split}
    \vi^{(k)} = \max_{\norm{\nu_1}_{\infty} \leq \alpha, \, \nu_2 \in \R^{\Npxl}_{-}} \; &H(\nu, \f^{(k)})    
    \qquad \text{with} \\[0.25em] 
    H(\nu, \f^{(k)}) = 
    - \frac{1}{2\lambda_k} \norm{\lambda_k D_k^{-1} \mbB^T \nu - \z^{(k)}}^2_{D_k} &- \Gamma_1(\f^{(k)})
    - \frac{\lambda_k}{2} \norm{\nabla \Gamma_0(\f^{(k)})}^2_{D_k^{-1}} + \frac{1}{2\lambda_k}
    \norm{\z^{(k)}}^2_{D_k}
    \end{split}    
\label{eq:VMILADualProbl_l1constrained}    
\end{equation}
where $\mbB$ is as in~\eqref{eq:MatrixBVMILA}. The conjugate function is given by $\widehat{\Gamma}_1^\star = g_1^\star(\nu_1)+g_2^\star(\nu_2)$, where $g_2^\star$ is again the indicator function of the set $\R^{\Npxl}_{-}$ and $g_1^\star(\nu_1)$ is the indicator function of the set $B^{\sigma \Npxl}_{\infty}(0,\alpha) = B_{\infty}(0,\alpha) \times \ldots \times B_{\infty}(0,\alpha)$ ($\sigma \Npxl$-times), being $B_{\infty}(0,\alpha) \subset \R$ the ball in the $\infty$-norm centered in $0$ with radius $\alpha$. Finally, to preserve feasibility (see~\cite[section 4.3]{Bonettini16}), we consider the sequence generated by projecting the corresponding primal sequence $\widetilde{\vi}^{k,l}$ onto the feasible set $\R^{\Npxl}_{+}$.

In all our experiments, we fix $\eta = 10^{-5}$, $\lambda_{\min} = 10^{-5}, \, \lambda_{\max} = 10^5$ and $\lambda_0 = 1.3$. The threshold $L$ for the scaling matrix $D_k$ has been set equal to $10^{10}$. We initialize the algorithm with $\f^{(0)} = 0$ and the inner loop of VMILA uses ``warm restart'' (after the very first iteration where the initial guess is $\vi^{(0)}=0$).
Computations were implemented with Matlab R2021a, running on a laptop with 16GB RAM and Apple M1 chip.

\begin{remark}
\label{rm:SGP}
Notice that in the case $1<p<2$ another choice for the functionals $\Gamma_0, \, \Gamma_1$ is possible, namely, $\Gamma_0(\f) = \frac{1}{2} \norm{\RadonD_{\thetab} \f - \gNd}_2^2 + \frac{\alpha}{p} \norm{\Mop \f}^p_p$
and $\Gamma_1(\f) = \iota_{\R^{\Npxl}_+}(\f)$ (or, $\Gamma_1(\f) = 0$ in the unconstrained case). In fact, we are dealing with a smooth, possibly constrained problem:
\begin{equation}
\argmin_{\f \in \Omega_{\f}} \; \Gamma(\f) := 
 \argmin_{\f \in \Omega_{\f}} \; \bigg\{ \frac{1}{2} \norm{\RadonD_{\thetab} \f - \gNd}_2^2 + \frac{\alpha}{p} \norm{\Mop \f}^p_p \bigg\}
\qquad \text{with } \; p \in (1,2) 
\label{eq:ObjFunctSGP}    
\end{equation}
where $\Omega_{\f}$ is either the feasible region $\R^{\Npxl}_+$ (constrained formulation) or $\R^{\Npxl}$ (unconstrained formulation).
In this case, and generally whenever the objective function is smooth, VMILA concides with SGP~\cite{Bonettini09} and no inner loop for the computation of the proximal operator is needed. Indeed, the $(k+1)$-th iteration of SGP, when the objective function reads as in~\eqref{eq:ObjFunctSGP}, is given by:
\begin{equation}
\label{eq:SGPiter}
\f^{(k+1)} = \f^{(k)} + \mu_k ( \mcP_{\Omega_{\f}}^{D_k^{-1}} \big( \f^{(k)} - \lambda_k D_k \nabla \Gamma (\f^{(k)}) \big) - \f^{(k)})
\end{equation}
where the projection $\mcP_{\Omega_{\f}}^{D_k^{-1}}$ onto the feasible set $\Omega_{\f}$  with respect to the metric induced by the scaling matrix $D_k$ replaces the proximal operator. Also for SGP we can apply the same updating rules for the steplength $\lambda_k$ and the scaling matrix $D_k$ exploited by VMILA.

In all tests reported in this paper, nonetheless, we always chose to follow the strategy presented at the beginning of the appendix, namely, $\Gamma_0(\f) = \frac{1}{2} \norm{\RadonD_{\thetab} \f - \gNd}_2^2$
and $\Gamma_1(\f) = \frac{\alpha}{p} \norm{\Mop \f}^p_p+\iota_{\R^{\Npxl}_+}(\f)$. While it might seem counterintuitive choosing to treat the case $1<p<2$ as the non-smooth case $p=1$, in our experience this turned out to be numerically more stable and computationally faster than having to compute a gradient updated involving the transform $\Mop$, its adjoint and the signed power operation, especially when $\Mop$ is the shearlet operator. In fact, even if the price to pay is to have an inner loop for the approximate computation of the proximity operator, the stopping criterion~\eqref{eq:VMILAinnerStop} for the inner loop of all tests reported here is met in just one iteration.
\end{remark}
\end{appendix}

\bibliographystyle{amsplain}

\bibliography{IPETbiblio}

\end{document}